\documentclass[11pt]{article}
\usepackage{fullpage}
\usepackage{amsmath,amssymb,amsthm,bbm,booktabs}
\usepackage{threeparttable}
\numberwithin{equation}{section}
\usepackage{enumerate,color,xcolor}
\usepackage[colorlinks,linkcolor=orange,citecolor=blue]{hyperref}
\usepackage{graphicx}
\usepackage{subfigure}
\usepackage{layouts}
\usepackage{comment}

\usepackage{todonotes} 
\usepackage[normalem]{ulem}

\newtheorem{theorem}{Theorem}
\newtheorem{lemma}{Lemma}

\newtheorem{corollary}{Corollary}
\newtheorem{assumption}{Assumption}
\newtheorem{example}{Example}
\usepackage{abstract}

\newcommand{\R}{\mathbb{R}}
\newcommand{\bz}{\mathbf{z}}
\newcommand{\beq}{\begin{eqnarray}}
\newcommand{\eeq}{\end{eqnarray}}

\title{\Large Non-Convex Optimization via Non-Reversible Stochastic Gradient Langevin Dynamics}
\author{
{Yuanhan Hu}\,\footnote{Both authors have equal contributions.}
\,\footnote{Department of Management Science
and Information Systems, Rutgers Business School, Piscataway, NJ-08854, United States of America; yh586@scarletmail.rutgers.edu.},
{Xiaoyu Wang$^{*}$}\,\footnote{Department of Mathematics, Florida State University, 1017 Academic Way, Tallahassee, FL-32306, United States of America; xwang3@math.fsu.edu.},
{Xuefeng Gao}\,\footnote{Department of Systems Engineering and Engineering Management, The Chinese University of Hong Kong, Shatin, N.T. Hong Kong; xfgao@se.cuhk.edu.hk.},
    {Mert G\"{u}rb\"{u}zbalaban}\,\footnote{Department of Management Science
and Information Systems and the DIMACS Institute, Rutgers University, Piscataway, NJ-08854, United States of America;
    mg1366@rutgers.edu.},
  Lingjiong Zhu\,\footnote{Department of Mathematics, Florida State University, 1017 Academic Way, Tallahassee, FL-32306, United States of America; zhu@math.fsu.edu.
  }
}
\date{\today}

\begin{document}

\setlength{\abovedisplayskip}{15pt}
\setlength{\belowdisplayskip}{15pt}

\maketitle

\begin{abstract}
Stochastic Gradient Langevin Dynamics (SGLD) is a poweful algorithm for optimizing a non-convex objective, where 
a controlled and properly scaled Gaussian noise is added to the stochastic gradients to steer the iterates towards a global minimum. SGLD is based on the overdamped Langevin diffusion which is reversible in time. 
By adding an anti-symmetric matrix to the drift term of the overdamped Langevin diffusion, one gets a non-reversible diffusion
that converges to the same stationary distribution
with a faster convergence rate. 
In this paper, we study the Mon-reversible Stochastic Gradient Langevin Dynamics (NSGLD) which is based on discretization of the non-reversible Langevin diffusion.
We provide finite-time performance bounds for the global convergence of NSGLD for solving stochastic non-convex optimization problems. Our results lead to non-asymptotic guarantees for both population and empirical risk minimization problems. Numerical experiments
for Bayesian independent component analysis and neural network models
show that NSGLD can
outperform SGLD
with proper choices of the anti-symmetric matrix.
\end{abstract}

%%%%%%%%%%%%%%%%%%%%%%%%%%%%%%%%%%%%%%%

%%%%%%%%%%%%%%%%%%%%%%%%%%%%%%%%%%%%%%%
\section{Introduction}
Consider the stochastic non-convex optimization problem:
\begin{equation} 
\min_{x\in\mathbb{R}^{d}}F(x):=\mathbb{E}_{Z \sim \mathcal{D}}[f(x,Z)]\,,
\label{eqn:populationrisk}
\end{equation} 
where $f: \mathbb{R}^d \times \mathcal{Z} \rightarrow \mathbb{R}$ is a continuous and possibly non-convex function and $Z$ is a random vector with an probability distribution $\mathcal{D}$ supported on some space $\mathcal{Z}$. Many problems in statistical learning theory can be formulated this way where $f$ is the per-sample loss function, $x$ is the model parameter to be learned and $Z$ models the random input data \cite{vapnik2013nature}.
% Assuming we can access to independent identical distributed (i.i.d.) samples $\hat{Z} = \left(Z_1,Z_2,...,Z_n\right)$ from $\mathcal{D}$, the goal is  
% to compute an approximate minimizer $\hat{x}$ (possibly with a randomized algorithm) of the population risk $F$ in \eqref{eqn:populationrisk}, i.e. we want to compute $\hat{x}$ such that $\mathbb{E}F(\hat{x}) - F^{*} \le \hat \varepsilon$ for a given target accuracy $\hat \varepsilon >0$, 
% where $F^* = \min_{x\in\mathbb{R}^{d}}F(x)$ and the expectation is taken with respect to the training data $\hat{Z}$ and the additional randomness for generating $\hat{x}$.
% This formulation arises frequently in several contexts including machine learning. 
Because the distribution $\mathcal{D}$ is typically unknown in practice, a common approach is to consider the empirical risk minimization problem
\begin{equation}
\min_{x\in\mathbb{R}^d} F_{\mathbf{z}}(x):=\frac{1}{n}\sum_{i=1}^{n}f(x,z_i)\,,
\label{eqn:empiricalsum}
\end{equation}
based on the dateset $\mathbf{z} = \left(z_1,z_2,...,z_n\right) \in \mathcal{Z}^n$ 
as a proxy to the population risk minimization problem \eqref{eqn:populationrisk} and minimize the empirical risk
\begin{equation}
\mathbb{E}F_{\mathbf{z}}(x) - \min_{x\in\mathbb{R}^d} F_{\mathbf{z}}(x)\,, 
\label{eqn:empiricalrisk}
\end{equation}
instead, where the expectation is taken with respect to any other randomness  
encountered during the algorithm to generate $x$. Many algorithms have been proposed to solve the problem \eqref{eqn:populationrisk} and its finite-sum version \eqref{eqn:empiricalrisk}. Among these, gradient descent, stochastic gradient and their variance-reduced or momentum-based variants come with guarantees for finding a local minimizer or a stationary point for non-convex problems. In some applications, convergence to a local minimum can be satisfactory. However, in general, methods with global convergence guarantees are also desirable and preferable in many settings \cite{xu2018global,li2016preconditioned}.

Stochastic gradient algorithms based on Langevin Monte Carlo are popular variants of stochastic gradient which admit asymptotic global convergence guarantees where a properly scaled Gaussian noise is added to the gradient estimate. 
Recently, Raginsky \emph{et al.} \cite{RRT17} provided a non-asymptotic analysis
of Stochastic Gradient Langevin Dynamics (SGLD, see \cite{welling2011bayesian}) to find the global minimizers for both population risk and empirical risk minimization problems. See also \cite{xu2018global} for related results. The SGLD can be viewed as the analogue of stochastic gradient in the Markov Chain Monte Carlo (MCMC) literature.
The SGLD iterates $\{X_k\}$ takes the following update form:
\begin{equation} \label{eq:SGLD-iterates}
X_{k+1}=X_{k}-\eta g_k +\sqrt{2\eta\beta^{-1}}\xi_k\,,
\end{equation}
where $\eta>0$ is the step size, $\beta>0$ is the inverse temperature, $g_k$ is a conditionally unbiased estimator of the gradient $\nabla F_{\mathbf{z}}$ and $\xi_k$ is a standard Gaussian random vector. The analysis of SGLD in \cite{RRT17} is built on the continuous-time diffusion process known as the overdamped Langevin stochastic differential equation (SDE):
\begin{equation} 
dX(t)=-\nabla F_{\mathbf{z}}\left(X(t)\right)dt+\sqrt{2\beta^{-1}}dB(t) , \quad t \geq 0\,, \label{reversibleSDE}
\end{equation}
where $B(t)$ is a standard $d$-dimensional Brownian motion.
 The overdamped Langevin diffusion \eqref{reversibleSDE} is reversible\footnote{A diffusion process $X(t)$ is reversible
if $X(0)$ is distributed according to the stationary measure $\pi$, then $(X(t))_{0\leq t \leq T}$ and $(X(T-t))_{0\leq t\leq T}$ have the same law for each $T$.} and admits a
unique stationary (or equilibrium) distribution $\pi_{\mathbf{z}}(dx) \propto e^{-\beta F_{\mathbf{z}}(x)}dx$ under some assumptions on $F_{\mathbf{z}}$. It is well documented that (see e.g. \cite{HHS05}) considering non-reversible variants of \eqref{reversibleSDE} can help accelerate convergence of the diffusion to the equilibrium. 
% by injecting a weighted divergence-free drift to \eqref{reversibleSDE},
% the diffusion becomes non-reversible.
% By breaking the reversibility, 
% Hwang \emph{et al.} \cite{HHS05} shows that
% the diffusion has a faster convergence rate,
% see also \cite{GGZ18-2}. 
Specifically,  
consider the following non-reversible diffusion:
\begin{equation}\label{nonreversibleSDE}
dX(t)=-A_J\left( \nabla F_{\mathbf{z}}(X(t)) \right)dt + \sqrt{2\beta^{-1}}dB(t) , \quad A_J = I+J\,, 
\end{equation}
where $J\neq 0$ is a $d \times d$ anti-symmetric matrix, i.e. $J^{T} = -J$, and $I$ is a $d \times d$ identity matrix. The stationary distribution $\pi_{\mathbf{z}}$ of this non-reversible Langevin diffusion \eqref{nonreversibleSDE} is the same as that of the overdamped Langevin diffusion \eqref{reversibleSDE}. In addition, \cite{HHS05} showed by comparing the spectral gaps that adding $J\neq 0$, the convergence to the stationary distribution of \eqref{nonreversibleSDE} is at least as fast as the overdamped Langevin diffusion ($J=0$), and is strictly faster except for some rare situations.

% Note that the non-reversible Langevin diffusion \eqref{nonreversibleSDE} reduces to the overdamped Langevin diffusion \eqref{reversibleSDE} when $J=0$. 
% It has been proved in \cite{HHS05} and \cite{LNP13} that the stationary distribution $\pi_{\mathbf{z}}$ of this non-reversible Langevin diffusion \eqref{nonreversibleSDE} is the same as the stationary distribution generated by the overdamped Langevin SDE \eqref{reversibleSDE},  but the convergence of the non-reversible diffusion to the equilibrium is generally faster.

% For more general non-quadratic F , [HHMS05] showed by comparing the spectral gaps that by adding J?= 0, the convergence to the Gibbs distribution is at least as fast as the overdamped Langevin diffusion (J = 0), and is strictly faster except for some rare situations

In this paper, we study a Non-reversible Stochastic Gradient Langevin Dynamics (NSGLD) and use it to solve the non-convex population and empirical risk minimization problems. NSGLD is based on the Euler-discretization of \eqref{nonreversibleSDE} with a stochastic gradient which has the update:
\begin{equation} 
X_{k+1}=X_{k}-\eta A_J g\left( X_k,U_{\mathbf{z},k} \right)  + \sqrt{2\eta\beta^{-1}}\xi_k\,, \label{NSGLD}
\end{equation}
where $\left\{ U_{\mathbf{z},k} \right\}_{k=0}^{\infty}$ is a sequence of i.i.d. random elements such that $g$ is a conditionally unbiased estimator for the gradient of $F_{\mathbf{z}}$ and satisfies $
\mathbb{E}\left[ g\left( x,U_{\mathbf{z},k} \right)\right] = \nabla F_{\mathbf{z}}(x)$ for any $x\in \mathbb{R}^d.$ When $J=0$, the NSGLD iterates in \eqref{NSGLD} reduces to the SGLD iterates in \eqref{eq:SGLD-iterates}. 
% Since the non-reversible diffusion 
% \eqref{nonreversibleSDE} converges to the 
% Gibbs distribution faster than 
% the reversible diffusion \eqref{reversibleSDE},
% it is natural to propose an 
% algorithm based on the non-reversible diffusion
% \eqref{nonreversibleSDE} to study 
% the non-convex optimization problems
% for population risk minimization \eqref{eqn:populationrisk} and 
% empirical risk minimization \eqref{eqn:populationrisk}.
Although asymptotic convergence guarantees for non-reversible Langevin diffusions \eqref{nonreversibleSDE} exists
(see e.g. \cite{HHS93,HHS05}), there is a lack of finite-time explicit performance bounds for solving stochastic non-convex optimization problems with NSGLD in the literature. 

% In this paper, we establish the global convergence of NSGLD and
% provide finite-time guarantees of NSGLD to find approximate minimizers of both empirical and population risks. Numerical experiments are conducted to illustrate our results and showed that NSGLD can outperform SGLD in applications if the antisymetric matrix is well chosen. 

% and providing numerical experiments that confirms
% the superior performance of our algorithm compared to SGLD via examples of a simple polynomial function optimization, Bayesian independent component analysis and neural network models.
%%%%%%%%%%%%%%%%%%%%%%%%%%%%%%%%%%%%%%%%%%%

%%%%%%%%%%%%%%%%%%%%%%%%%%%%%%%%%%%%%%\
\textbf{Contributions.}
We establish the global convergence of NSGLD and
 provide finite-time guarantees of NSGLD to find approximate minimizers of both empirical and population risks. Specifically,
 
(1) Under Assumptions~\ref{A1}-\ref{A5} for the component functions $f(x,z)$ and the gradient noises, 
we show that NSGLD converges to an $\varepsilon-$approximate global minimizer of the empirical risk minimization problem after $\text{poly}\left(\frac{1}{|\lambda_{\mathbf{z},J}|},\beta,d,\frac{1}{\varepsilon}, \lVert A_J \rVert \right)$ iterations in expectation, where $\lambda_{\mathbf{z},J}$ is the spectral gap of the non-reversible Langevin SDE \eqref{nonreversibleSDE} governing the speed of convergence to its stationary distribution. See Corollary~\ref{cor:ERM-bd} and Equation~\eqref{eq:iteratoins-KJ}. 
(2) On the technical side, we adapt the proof techniques of \cite{RRT17} developed for SGLD to NSGLD and combine it with the analysis of \cite{HHS05}. 
We overcome several technical challenges and the key steps of our proofs are as follows.
First, we show in Theorem 6 the convergence of the expected empirical risk $\mathbb{E}[F_{\mathbf{z}}(X(t)]$ for the non-reversible Langevin SDE as $t \rightarrow \infty$. We build on \cite{HHS05} but their results
% which established the convergence of $X(t)$ to equilibrium in $L^2(\pi_{\mathbf{z}})$ (and in variational norm), not the $2$-Wasserstein distance in \cite{RRT17}. However, the results in \cite{HHS05} 
do not directly imply the convergence of the expected empirical risk. We overcome this challenge by establishing a novel uniform $L^{4}$ bound of $X(t)$, apply the continuous-time convergence results from \cite{HHS05} on a compact set, and provide additional estimates outside the compact set. Second, we show that NSGLD iterates track the non-reversible Langevin
 SDE closely with small step sizes. We
use the approach in \cite{RRT17} via relative entropy estimates. 
But our analysis requires establishing new uniform $L^{2}$ bound and exponential integrability of $X(t)$ in \eqref{nonreversibleSDE},
by using a different Lyapunov function from \cite{RRT17}. In addition, our discretization error improves the one in \cite{RRT17} for $J=0$, based on a tighter estimate on the exponential integrability.

(3) We complement our theoretical results with the empirical evaluations of the performance of NSGLD on a variety of optimization tasks such as optimizing a simple double well function, Bayesian Independent Component Analysis and Neural Network Models. Our experiments suggest that NSGLD can outperform SGLD in applications. % with proper choices of anti-symmetric matrices $J$.

\textbf{Related Literature.} A number of papers studied the non-reversible SDE $X(t)$
in \eqref{nonreversibleSDE} with a quadratic objective $F_{\mathbf{z}}$, in which case $X(t)$ becomes a 
Gaussian process. Using the rate of convergence of the covariance of $X(t)$
as the criterion, \cite{HHS93} showed that $J=0$ is the worst choice.
% and improvement is possible if and only if the eigenvalues of the matrix
% associated with the linear drift term are not identical.
\cite{LNP13} proved the existence of the optimal antisymmetric matrix $J$
such that the rate of convergence to equilibrium is maximized,
and provided an easily implementable algorithm for constructing them.
\cite{WHC2014} proposed two approaches to design $J$ to obtain the optimal convergence rate of Gaussian diffusion
and they also compared their algorithms with the one in \cite{LNP13}. See also 
\cite{guillin2016} for related results. However, the optimal choice of $J$ is still open when the objective is non-quadratic. 
% For more general non-quadratic objective functions, \cite{HHS05} showed by comparing the spectral gaps
% that by adding $J\neq 0$, the convergence to the Gibbs distribution
% is at least as fast as the overdamped Langevin diffusion ($J=0$),
% and is strictly faster except for some rare situations. 

Another line of related research focused on sampling and Monte Carlo methods based on the non-reversible Langevin diffusion \eqref{nonreversibleSDE}. As have been observed in the literature \cite{rey2016improving}, non-revesible Langevin sampler can outperform their reversible counterparts in terms of rate of convergence to equilibrium, asymtotic variance \cite{DLP2016, reyGraphs} and large deviation functionals \cite{reyLDP}. 
% For instance, \cite{DLP2016, reyGraphs} showed that the asymptotic variance can be reduced
% by using the non-reversible Langevin sampler.
See also \cite{DPZ17} for non-reversible samplers based on splitting schemes. 
% constructed efficient sampling algorithms based on the Lie-Trotter decomposition of a nonreversible diffusion process and showed that samplers based on this scheme can outperform standard MCMC methods.
% \cite{reyLDP} used the Donsker-Varadhan large deviations theory to analyze the speed of convergence of non-reversible 
% Langevin diffusion to the invariant measure, and showed
% acceleration of the convergence speed due to a larger
% large deviations rate function.
We also refer the readers to \cite{ma2015complete} which presented a general recipe for devising stochastic gradient MCMC samplers based on continuous diffusions including the non-reversible SDE in \eqref{nonreversibleSDE}.
Our work is different from these studies in that we focus on optimization and analyze the expected suboptimality of NSGLD iterates, while typically one studies the convergence to equilibrium for ergodic averages in sampling.

%%%%%%%%%%%%%%%%%%%%%%%%%%%%%%%%%%%%%%%%%%%%%%
\section{Preliminaries}
%\subsection{Assumptions} \label{sec:assumption}

We first state the assumptions used in this paper below. Note that we do not assume $f$ to be convex or strongly convex in any region. 
%\add{LZ: I think in Eberle, he assume that $F\geq 0$. Somehow I didn't see it below.}

%%%%%%%%%%%%%%%%%%%%%%%%%%%%%%%%%%%%%
%Assumptions
\begin{assumption} \label{A1}
The function $f$ is continuously differentiable and taking non-negative values, then there exists some constant $A,B \geq 0$, such that 
$|f(0,z)|\leq A$ and $\lVert \nabla f(0,z) \rVert \leq B$, for any $z\in \mathcal{Z}$.  
\end{assumption}
%\add{xy: $f$ is taking non-negative value, $\min F = E(f(x,Z))$ in (1.1) and $\min F=1/n \sum_{i=1}^n f(x,z_i)$ in (1.3). I think there is $F \geq 0$.}

\begin{assumption} \label{A2}
For each $z \in \mathcal{Z}$, the function $f(\cdot,z)$ is $M$-smooth: for some $M>0$,
\begin{equation*}
\lVert \nabla f(w,z) - \nabla f(v,z) \rVert \leq M \lVert w-v \rVert\,,
\qquad\text{for any $w,v\in\mathbb{R}^{d}$}.
\end{equation*}
\end{assumption}

\begin{assumption} \label{A3}
For any $z\in \mathcal{Z}$, $f(\cdot,z)$ is $(m,b)$-dissipative, that is
\begin{equation*}
\langle x, \nabla f(x,z)\rangle \geq m\lVert x \rVert^2 - b\,,
\qquad\text{for any $x\in\mathbb{R}^{d}$}.
\end{equation*}
\end{assumption}

\begin{assumption} \label{A4}
There exists $\delta \in [0,1 )$, for any data set $\mathbf{z}$, such that 
\begin{equation*}
\mathbb{E} \lVert g\left( x,U_{\mathbf{z}} \right) - \nabla F_{\mathbf{z}}(x) \rVert^2  \leq 2\delta \left(M^2\lVert x \rVert^2+B^2 \right)\,,
\qquad\text{for any $x\in\mathbb{R}^{d}$}.
\end{equation*}
\end{assumption}

\begin{assumption} \label{A5}
The initial state $X(0)$ of the NSGLD algorithm satisfies $\lVert X(0) \rVert \leq R := \sqrt{b/m}$ with probability one,
i.e. $X(0)\in B_{R}(0)$, the Euclidean ball centered at $0$ with radius $R$.
\end{assumption}

% The first assumption of non-negativity of $f$ can be assumed without loss of generality by subtracting a constant and shifting the coordinate system as long as $f$ is bounded below. Also the non-negative $f$ implies the non-negativity of $F_{\mathbf{z}}$ by \eqref{eqn:populationrisk} and \eqref{eqn:empiricalsum}. The second assumption of Lipschitz gradients is in general unavoidable for discretized Langevin algorithms to obtain convergence (see e.g. \cite{mattingly2002ergodicity}), and the third assumption is known as the \emph{dissipativity condition} (see e.g. \cite{hale1988asymptotic}) and is standard in the literature to ensure the convergence of Langevin diffusions to the stationary distribution (see e.g. \cite{RRT17, mattingly2002ergodicity}). The fourth assumption is regarding the amount of noise present in the gradient estimates and allows not only constant variance noise but allows the noise variance to grow with the norm of the iterates (which is the typical situation in mini-batch methods in stochastic gradient methods). 
% Finally, it follows from the dissipativity condition that any stationary point $y$ of $F_{\mathbf{z}}$ satisfies $|y| \le R$, which naturally leads to the fifth assumption on the starting point of the NSGLD algorithm.

%%%%%%%%%%%%%%%%%%%%%%%%%%%%%%%%%%%%%%%%%%%%%%
%\subsection{Convergence Rate to the Equilibrium of the Non-Reversible Langevin SDE} \label{sec:convergence-SDE}

We next recall the result on the convergence rate to the equilibrium of the non-reversible Langevin SDE in \eqref{nonreversibleSDE} in Hwang \emph{et al.} \cite{HHS05}. Write $\pi_{\mathbf{z}}$ for its stationary distribution.
%Hwang \emph{et al.} \cite{HHS05} considered the rate of convergence of $p_{\mathbf{z},J}(t,x,y)$ to $\pi_{\mathbf{z}}(dy)$ in variational norm, where the variational norm of two probability measures is the supremum of the difference between two probabilities over all events. To study whether the non-reversible Langevin SDE in \eqref{nonreversibleSDE} has a faster convergence to equilibrium than the reversible Langevin SDE in \eqref{reversibleSDE}, 
%Hwang \emph{et al.} \cite{HHS05} defined
%\begin{equation} 
% \rho_{\mathbf{z},J} := \inf\left\{\rho<0: \int \left| p_{\mathbf{z},J}(t,x,y) - \pi_{\mathbf{z}}(y) \right| dy \leq g_{\mathbf{z},J}(x) e^{\rho t}\right\}, \,\label{variationalnorm}
 %\end{equation} 
% where $g_{\mathbf{z},J}(x)$ is a locally bounded function and it may depend on $\mathbf{z}$ and $J$. 
 %In the following Lemma~\ref{lm:acceleration}, we will obtain a uniform upper bound for this $g$ function with respect to $\mathbf{z} \in \mathcal{Z}^n $ based on analyzing a local Harnack inequality.
%  Hwang \emph{et al.} \cite{HHS05} considered the spectral gap in $L^2 (\pi_{\mathbf{z}})$ to analyze the rate of convergence of the non-reversible Langevin SDE in \eqref{nonreversibleSDE} to the equilibrium. Specifically, 
Let $\mathcal{L}_{\mathbf{z},J}$ be the infinitesimal generator \cite{bakry2013analysis} of the SDE in \eqref{nonreversibleSDE}. Define 
\begin{equation}\label{spectrum:z}
\lambda_{\mathbf{z},J}:= \sup\left\{\text{the real part of } \phi:  \phi  \text{ is in the spectrum of } \mathcal{L}_{\mathbf{z},J}, \phi \neq 0 \right\}\,. 
\end{equation}
In general, the eigenvalues of the generator $\mathcal{L}_{\mathbf{z},J}$ are complex numbers, there is a simple eigenvalue 0 and all the other eigenvalues have negative real parts. The quantity $\lambda_{\mathbf{z},J}$ (or sometimes $|\lambda_{\mathbf{z},J}|$) is referred to as the spectral gap of the generator $\mathcal{L}_{\mathbf{z},J}$, since $|\lambda_{\mathbf{z},J}|$ is the minimal gap between the zero eigenvalue and the real parts of the rest of the non-zero eigenvalues. The existence of a spectral gap, i.e. $\lambda_{\mathbf{z},J}<0$, implies that 
% \begin{align}
% &\int_{\mathbb{R}^{d}} \left(\mathbb{E}_y g(X(t)) - \int_{\mathbb{R}^{d}} g(y) 
% \pi_\mathbf{z}(dy) \right)^2 \pi_\mathbf{z}(y) dy \nonumber
% \\
% & \le C_{\mathbf{z}, J} \cdot \int_{\mathbb{R}^{d}} \left(g(y) - \int_{\mathbb{R}^{d}} g(y) \pi_\mathbf{z}(dy) \right)^2 \pi_\mathbf{z}(y) dy  \cdot e^{2 \lambda_{\mathbf{z},J} t }, \label{eq:L2-contraction-SG}
% \end{align}
% where $C_{\mathbf{z}, J}$ is a constant that may depend on $F_\mathbf{z}$ and $J$. That is, 
the non-reversible Langevin SDE in \eqref{nonreversibleSDE} converges to equilibrium exponentially fast with rate $\lambda_{\mathbf{z},J}$ in the following sense: 
\begin{equation}\label{eq:L2-contraction-SG}
\Vert T(t)g - \pi_\mathbf{z}(g)\Vert_{L^2 (\pi_{\mathbf{z}})} \le \sqrt{C_{\mathbf{z}, J}} \cdot \Vert g - \pi_\mathbf{z}(g)\Vert_{L^2 (\pi_{\mathbf{z}})} \cdot e^{ \lambda_{\mathbf{z},J} t }, \quad \text{for any $g \in L^2 (\pi_{\mathbf{z}})$},
\end{equation}
where $T(t)= e^{t \mathcal{L}_{\mathbf{z},J}}$, $\pi_\mathbf{z}(g)$ means the integration of $g$ with respect to $\pi_\mathbf{z}$, $\Vert\cdot\Vert_{L^2 (\pi_{\mathbf{z}})}$ denote the norm in $L^2 (\pi_{\mathbf{z}})$, and $C_{\mathbf{z}, J}$ is a constant that may depend on $F_\mathbf{z}$ and $J$. See Equation (3) in \cite{HHS05}.
Note when $J=0,$ the constant $C_{\mathbf{z}, J} \equiv 1$. See e.g. \cite[Section 3.1]{rey2016improving}.

%\mtodo{Can we explain the ineq. above a bit more for CS audience?}

Using the spectral gap as one comparison criteria, Hwang \emph{et al.} \cite[Section~2]{HHS05} showed that
\begin{itemize}
\item $\lambda_{\mathbf{z},J} \le \lambda_{\mathbf{z},J=0} <0$; 
\item The equality $\lambda_{\mathbf{z},J}=\lambda_{\mathbf{z},J=0}$ holds in some rare situations: if $\lambda_{\mathbf{z},J=0}$ is in the discrete spectrum of $\mathcal{L}_{\mathbf{z},J=0}$, then $\lambda_{\mathbf{z},J} = \lambda_{\mathbf{{z}},J=0}$ if and only if $(\mathcal{L}_{\mathbf{z}, J}-\mathcal{L}_{\mathbf{z},J=0})$ or $(\mathcal{L}_{\mathbf{z}, -J}-\mathcal{L}_{\mathbf{z},J=0})$ leaves a nonzero
 subspace of the eigenspace corresponding to $\lambda_{\mathbf{z},J=0}$ to be invariant.    %\mtodo{Should we be more explicit here about when exactly equality holds}
%\item $\rho_{\mathbf{z},J} \le \lambda_{\mathbf{z},J} <0 $, and the equality holds for $J$ being a zero matrix. i.e. $\rho_{\mathbf{z},J=0} = \lambda_{\mathbf{z},J=0}$, and hence $\rho_{\mathbf{z},J}< \rho_{\mathbf{z},J=0}.$
\end{itemize}
In other words, generically the non-reversible Langevin SDE in \eqref{nonreversibleSDE} converges to the equilibrium faster than the reversible SDE in \eqref{reversibleSDE}. Note that this is a continuous-time result. 
% To facilitate the presentation, we also define 
% %the uniform variational norm by
% %\begin{equation} 
% % \rho_{*,J}  := \inf_{\mathbf{z}\in\mathcal{Z}^n} |\rho_{\mathbf{z},J} | \, ,
% %\label{unform:variationalnorm}
% %\end{equation}
% % and 
% the uniform spectral gap by
% \begin{equation}\label{uniform:spectralgap}
% \lambda_{*,J} :=\inf_{\mathbf{z}\in\mathcal{Z}^n}|\lambda_{\mathbf{z},J}|\,.
% \end{equation}
% This quantity will be used in the study of performance bound for population risk minimizations. 
% Note when $J = 0$, the diffusion process \eqref{nonreversibleSDE} reduces to the reversible Langevin diffusion \eqref{reversibleSDE}, and $\lambda_{*,J=0}$ becomes the uniform spectral gap $\lambda_*$ defined in \cite{RRT17} in the study of stochastic gradient langevin dynamics. In addition, since $\lambda_{\mathbf{z},J} \le \lambda_{\mathbf{z},J=0} <0$, we have $\lambda_{*,J} \geq \lambda_{*,J=0}> 0.$

% , where they provided a characterization
% for $\lambda_{\ast}=\lambda_{\ast,J=0}$ as follows
% \footnote{The formula in \cite{RRT17} had a typo,
% and their formula missed the $\beta^{-1}$ factor.
% See Appendix \ref{sec:spectral:gap} for a derivation.}:
% \begin{equation} \label{uniform:spectralgap-rev}
% \lambda_{*}=\inf_{\mathbf{z}\in\mathcal{Z}^n}\inf\left\{ \frac{{\color{red}\beta^{-1}}\int_{\mathbb{R}^d} \, \lVert \nabla h\rVert^2 \, d\pi_{\mathbf{z}}}{\int_{\mathbb{R}^d}h^2 \, d\pi_{\mathbf{z}}} :h\in C^1(\mathbb{R}^d)\cap L^2(\pi_{\mathbf{z}}), h\neq 0, \int_{\mathbb{R}^d}hd\pi_{\mathbf{z}}=0 \right\}\,.
% \end{equation}

%%%%%%%%%%%%%%%%%%%%%%%%%%%
\section{Main Results}

\subsection{Convergence to Equilibrium in Expectations}
We now state our first set of results. The proofs are given in Appendix~\ref{sec:proof-main}. 
% The first result translates the convergence of the non-reversible Langevin diffusion in \eqref{nonreversibleSDE}
% in spectral gaps to equilibrium to the convergence of the expectation of the empirical risk $F_{\mathbf{z}}$.
Conditional on the sample 
$\mathbf{z}$, we use $\nu_{\mathbf{z}, t}$ to denote the probability law of the continuous-time process $X(t)$ in \eqref{nonreversibleSDE} at time $t$ and  $\pi_{\mathbf{z}}$ to denote its stationary distribution. 
% Recall that $B_{R}(0)$ denotes the Euclidean ball centered at $0$ with radius $R=\sqrt{b/m}$,
 The following result establishes the convergence of the expected empirical risk $\mathbb{E}[F_{\mathbf{z}}(X(t)]$ as $t \rightarrow \infty$. Recall that $\lambda_{\mathbf{z},J}$ is defined in \eqref{spectrum:z}.

\begin{theorem} \label{thm:nonreversible2gibbs}
Considering the non-reversible Langevin SDE in \eqref{nonreversibleSDE}. If Assumptions ~\ref{A1} - \ref{A3} hold, then for any $\beta \geq 3/m$ and $\varepsilon>0 $,
\begin{equation*}
\left|\mathbb{E}_{X \sim \nu_{\mathbf{z},k\eta}} F_{\mathbf{z}}(X) - \mathbb{E}_{X \sim \pi_{\mathbf{z}}} F_{\mathbf{z}}(X) \right| \leq \mathcal{I}_{0}(\mathbf{z},J,\varepsilon) \,,
\end{equation*}
where 
\begin{equation}
\mathcal{I}_{0}(\mathbf{z},J,\varepsilon) : =   \left[\left( \frac{M+B}{2} +\frac{B}{2}+A\right) \hat{C}_{\mathbf{z},J} + (M+B)\mathcal{D}_c\right]\cdot \varepsilon \,,  \label{thm:const:I0}
\end{equation}
provided that
$ k\eta \geq \max \left\{ \frac{2}{|\lambda_{\mathbf{z},J}|}\log\left(\frac{1}{\varepsilon}\right), \, 1 \right\}.
$
Here, the constant $\mathcal{D}_c$ is defined in \eqref{lm:const:Dc} and the constant $\hat{C}_{\mathbf{z},J}$ is defined in \eqref{hat:C:z:J} in the appendix.
\end{theorem}
The next result controls the error at time $k$ between the discretized process and the stationary distribution for a given sample $\mathbf{z}$. Specifically,
 we consider the iterates $X_k$ of the NSGLD algorithm in \eqref{NSGLD}, and we denote its probability law by $\mu_{\mathbf{z},k}$ conditional on $\mathbf{z}$. Since the NSGLD algorithm is based on the Euler discretization of the non-reversible Langevin SDE in \eqref{nonreversibleSDE}, we can control the discretization error with stochastic gradients and use Theorem~\ref{thm:nonreversible2gibbs} to obtain the following result.  

\begin{corollary} \label{cor:NSGLD2gibbs}
Under the setting of Theorem~\ref{thm:nonreversible2gibbs} where the Assumptions ~\ref{A1} - \ref{A5} hold, let $\beta \geq 3/m$, for any given $\varepsilon>0$, the performance bound  of NSGLD algorithm admits
 \begin{equation*}
\left|\mathbb{E}_{X \sim \mu_{\mathbf{z},k}} F_{\mathbf{z}}(X) - \mathbb{E}_{X \sim \pi_{\mathbf{z}}} F_{\mathbf{z}}(X) \right| \leq \mathcal{I}_{0}(\mathbf{z},J,\varepsilon) + \mathcal{I}_1(\mathbf{z},J,\varepsilon) \,,
\end{equation*} 
where $\mathcal{I}_0(\mathbf{z},J,\varepsilon)$ is defined in \eqref{thm:const:I0} and
\begin{equation} \label{eq:I1}
\mathcal{I}_{1}(\mathbf{z},J,\varepsilon) := \left(M\sqrt{\mathcal{C}_d}+B \right)\left( \hat{C_0} \frac{\varepsilon}{\sqrt{\lvert \lambda_{\mathbf{z},J=0}\rvert}}+ \hat{C_1}\delta^{1/4}\sqrt{\frac{2\log(1/\varepsilon)}{\lvert \lambda_{\mathbf{z},J} \rvert}}\lVert A_J \rVert \right) \sqrt{\log\left( \frac{2\log(1/\varepsilon)}{\lvert \lambda_{\mathbf{z},J} \rvert} \right)}\,,
\end{equation}
provided that the step size $\eta$ satisfies
\begin{equation} \label{eq:eta}
\eta \leq \min\left\{1, \frac{m^2}{(m^2+ 8M^2)M\lVert A_J\rVert^2},\frac{\varepsilon^4}{4(\log(1/\varepsilon))^2\lVert A_J \rVert^4} \frac{\lvert \lambda_{\mathbf{z},J} \rvert^2}{\lvert \lambda_{\mathbf{z},J=0} \rvert^2} \right\} \,,
\end{equation}
and
\begin{equation} \label{eq:k-eta}
k\eta =  \frac{2}{|\lambda_{\mathbf{z},J}|}\log\left(\frac{1}{\varepsilon}\right)\geq e \,.
\end{equation} 
Here $\delta$ is the gradient noise level satisfying Assumption~\ref{A4},
the constants $\mathcal{C}_d, \, \hat{C_0}$, $\hat{C_1}$ are explicit and can be found in 
Lemma~\ref{lm:uniform-L2} and Lemma~\ref{lm:approximation} in the appendix respectively.
\end{corollary}
% This result suggests that if $\eta = O\left(\frac{\varepsilon^4}{4(\log(1/\varepsilon))^2\lVert A_J \rVert^4} \frac{\lvert \lambda_{\mathbf{z},J} \rvert^2}{\lvert \lambda_{\mathbf{z},J=0} \rvert^2} \right)$ and the gradient noise $\delta$ is taken to be the same as $\eta$, then $\mathcal{I}_{1}(\mathbf{z},J,\varepsilon)= \tilde O \left( \frac{\varepsilon}{\sqrt{\lvert \lambda_{\mathbf{z},J=0}\rvert}} \right)$, igonring the $\log \log (1/\varepsilon)$ term in \eqref{eq:I1} and hiding dependency on other parameters. 
In the next subsection, we will
show that this result combined with some basic properties of the equilibrium distribution $\pi_{\mathbf{z}}$ leads to performance guarantees for the empirical risk minimization.

\subsection{Performance Bound for the Empirical Risk Minimization}
Consider using the NSGLD algorithm in \eqref{NSGLD} to solve the empirical risk minimization problem given in \eqref{eqn:empiricalsum}.
The performance of the algorithm can be measured by the expected sub-optimiality: $\mathbb{E}_{X\sim\mu_{\mathbf{z},k}}F_{\mathbf{z}}(X) - \min_{x\in\mathbb{R}^d}F_{\mathbf{z}}(x)$. To obtain performance guarantees, 
in light of Corollary ~\ref{cor:NSGLD2gibbs}, one has to control the quantity
$\mathbb{E}_{X \sim \pi_{\mathbf{z}}}F_{\mathbf{z}}(X) - \min_{x\in\R^d} F_{\bz}(x)$,
which is a measure of how much the equilibrium distribution $\pi_\bz$ concentrates around a global minimizer of the empirical risk. {For finite $\beta$, \cite[Proposition 11]{RRT17} derives an explicit bound:}
	 \beq 
	\mathbb{E}_{X \sim \pi_{\mathbf{z}}}F_{\mathbf{z}}(X) - \min_{x\in\R^d} F_{\bz}(x)  \leq \mathcal{I}_2 := \frac{d}{2\beta}\log\left(\frac{eM}{m}\left(\frac{b\beta}{d}+1\right)\right).  \label{cor:const:I2}
	 \eeq
Hence we immediately obtain the following performance bound for the empirical risk minimization.

\begin{corollary}[Empirical risk minimization]\label{cor:ERM-bd}
Consider the iterates $\{X_k\}$ of the NSGLD algorithm in \eqref{NSGLD}.
Under the setting of Corollary~\ref{cor:NSGLD2gibbs}, let $\beta \geq 3/m$, for any given $\varepsilon>0$, we have
\begin{align*} 
 \mathbb{E}F_{\mathbf{z}}(X_k) - \min_{x\in\mathbb{R}^d}F_{\mathbf{z}}(x) 
 %& = \mathbb{E}F_{\mathbf{z}}(X_k) - \mathbb{E}_{X \sim \pi_{\mathbf{z}}}F_{\mathbf{z}}(X) + \mathbb{E}_{X \sim \pi_{\mathbf{z}}}F_{\mathbf{z}}(X) - \min_{x\in\mathbb{R}^d}F_{\mathbf{z}}(x)  \nonumber \\
%& 
\leq \mathcal{I}_0(\mathbf{z},J,\varepsilon)+\mathcal{I}_1(\mathbf{z}, J,\varepsilon)+ \mathcal{I}_2 \,, 
\end{align*}
provided that the step size $\eta$ satisfies \eqref{eq:eta}, and $k$ satisfies \eqref{eq:k-eta}. Here, $\mathcal{I}_0$ and $\mathcal{I}_1$ are given in Corollary~\ref{cor:NSGLD2gibbs}, and $\mathcal{I}_2$ is defined in \eqref{cor:const:I2}.
% , and $ \mathcal{I}_1(J,\varepsilon) $ is defined by 
% \begin{equation}
% \mathcal{I}_{1}(J,\varepsilon) := \left(M\sqrt{\mathcal{C}_d}+B \right)\left( \hat{C_0} \frac{\varepsilon}{\sqrt{\lambda_{\ast,J=0}}}+ \hat{C_1}\delta^{1/4}\sqrt{\frac{2\log(1/\varepsilon)}{\lvert \lambda_{\ast,J} \rvert}}  \lVert A_J \rVert  \right)\sqrt{\log\left( \frac{2\log(1/\varepsilon)}{\lvert \lambda_{\ast,J} \rvert} \right)}  \,,
% \label{cor:const:I1-star} 
% \end{equation}
% with $\lambda_{*,J} = \inf_{\mathbf{z}\in\mathcal{Z}^n} |\lambda_{\mathbf{z},J}| $ given in \eqref{uniform:spectralgap} and $\delta$ is gradient noise level under the setting in Assumption~\ref{A4}, and $\hat{C}_0$, $\hat{C}_1$ are two constants defined in Lemma~\ref{lm:approximation}, provided that
% $$
% k\eta = \frac{2}{\lambda_{*,J}}\log\left(\frac{1}{\varepsilon}\right) \geq e \,,
% $$
% and
% $$
% \eta \leq \min\left\{1,\frac{m^2}{(m^2+ 8M^2)M\lVert A_J\rVert^2} ,\frac{\varepsilon^4}{4(\log(1/\varepsilon))^2\lVert A_J \rVert^4}\frac{\lambda_{\ast,J}^2}{\lambda_{\ast,J=0}^2}\right\} \,.
% $$
\end{corollary}
Based on this result one can also derive the performance bound for the population risk minimization in \eqref{eqn:populationrisk}. See Appendix~\ref{sec:p-r-m} for details. We use the notation $\tilde{\mathcal{O}}(\cdot), \tilde{{\Omega}}(\cdot)$ gives explicit dependence on the parameters $\beta, d, \lambda_{\mathbf{z}, J}, \delta$, but hides factors that depend polynomially on other parameters.
Our result in Corollary~\ref{cor:ERM-bd} (see  Appendix~\ref{sec:dependancy-para} for further details) suggests that for empirical risk minimizations, the performance bound of NSGLD is given by (ignoring the $\log \log (1/\varepsilon)$ term):
\begin{align}
%\tilde{\mathcal{O}}\left( \hat C_{\mathbf{z},J} \cdot \varepsilon  +  \frac{\beta\sqrt{\beta+d}}{ \sqrt{\lambda_{*,J=0}} } \cdot \varepsilon
%\left(\varepsilon + \delta^{1/4} \sqrt{\log\left( \frac{1}{\varepsilon} \right)}   \right) \lVert A_J \rVert
%+ \frac{d\log(1+\beta)}{\beta} \right) \,,
\tilde{\mathcal{O}}\left( \hat C_{\mathbf{z},J} \cdot \varepsilon  +  \frac{\sqrt{\beta}(\beta+d)}{ \sqrt{|\lambda_{\mathbf{z},J=0} | } } \cdot \varepsilon
%\left(\varepsilon + \delta^{1/4} \sqrt{\log\left( \frac{1}{\varepsilon} \right)}   \right) \lVert A_J \rVert
+ \frac{d\log(1+\beta)}{\beta} \right) \,,
\label{cmp:erm-nonrev}
\end{align}
% with $k\eta = \tilde{\mathcal{O}}\left(2\log(1/\varepsilon))/\lambda_{*,J} \right)$,  $C_{\mathbf{z},J}$ is constant which may depend on $F$, $d$ and $\beta$. The NSGLD empirical risk minimization performance can be obtained by executing NSGLD algorithm with 
after $\mathcal{K}_{J}$ iterations with
\begin{equation} \label{eq:iteratoins-KJ}
%\mathcal{K}_{J}= \tilde{\Omega}\left( \frac{ \beta\sqrt{\beta+d}}{ {\lambda_{*,J} } \, \varepsilon^4}\log^3 \left(\frac{1}{\varepsilon}\right) \cdot \frac{\Vert A_J\Vert^4 \lambda_{*,J=0}^2 }{\lambda_{*,J}^2} \right) \quad and \quad \eta \leq \frac{\varepsilon^4} {  \left(\log(1/\varepsilon) \right)^2}  \frac{1}{\Vert A_J\Vert^4} \frac{ \lambda_{*,J}^2}{\lambda_{*,J=0}^2}  \,,
\mathcal{K}_{J}= \tilde{\Omega}\left( \frac{\sqrt{\beta}(\beta+d)}{ {|\lambda_{\mathbf{z},J} |} \, \varepsilon^4}\log^3 \left(\frac{1}{\varepsilon}\right) \cdot \frac{\Vert A_J\Vert^4 \lambda_{\mathbf{z},J=0}^2 }{\lambda_{\mathbf{z},J}^2} \right) \quad and \quad \eta \leq \frac{\varepsilon^4} {  4\left(\log(1/\varepsilon) \right)^2}  \frac{1}{\Vert A_J\Vert^4} \frac{ \lambda_{\mathbf{z},J}^2}{\lambda_{\mathbf{z},J=0}^2}  \,,
\end{equation}
when the gradient noise $\delta$ is set to be the same as the step size $\eta$.

%%%%%%%%%%%%%%%%%%%%%%%%%%%%%%%%%%%%%%%%%%%%%%%%%%%%%%%%%%%%%%%%%%%%%

% the
% convergence of $\mathbb{E}F_{\mathbf{z}}(X(t))$ to $\mathbb{E}_{X \sim \pi_{\mathbf{z}}}F_{\mathbf{z}}(X)$ for the non-reversible Langevin SDE in \eqref{nonreversibleSDE} when $t$ is large. 

%%%%%%%%%%%%%%%%%%%%%%%%%%%%%%%%%%%%%%%%%%
\subsection{Discussion: NSGLD vs SGLD}\label{sec:comparison}
In this section we briefly discuss the comparison of the performance of NSGLD with that of SGLD (corresponding to $J=0$) in the context of empirical risk minimizations. Note that while adding a nonzero antisymetric matrix $J$ increases the rate of convergence of diffusions to the equilibrium (i.e. $|\lambda_{\mathbf{z},J}|> |\lambda_{\mathbf{z},J=0}|$ generically), it will also give rise to a larger discretization error and amplify the gradient noise if one runs NSGLD and SGLD with the same stepsize. See the experiments in Section~\ref{sec:simple}. Building on our theoretical results in previous sections, we give some further analysis below to show that NSGLD can outperformance SGLD when the matrix $J$ is properly chosen. % and the function to optimise satisfy certain spectral properties. 

As in \cite{RRT17} and \cite{xu2018global}, we define an almost empirical risk minimizer as a point which is within the ball of the global minimizer with radius $\tilde{\mathcal{O}}(d \log(1+\beta)/\beta)$ and we discuss the performance of NSGLD and SGLD in terms of \textit{gradient complexity}, i.e., the total number of stochastic gradients required to achieve an almost empirical risk minimizer. We consider the mini-batch setting, where at each iteration of NSGLD one samples uniformly with replacement a random i.i.d. mini batch of size $\ell$.
%} 
% Note here $A_J = I+ J$. For any anti-symmetric matrix $J$, %and $x \in \mathbb{R}^d$, we have $x'Jx=0$. It readily follows that $ \lVert A_J \rVert^2 = 1 +  \lVert J \rVert^2 $ since $ \lVert A_J x \rVert^{2}= \lVert x \rVert^{2} +  \lVert J x \rVert^{2} $.
% we have $\lVert A_J \rVert^2 = 1 +  \lVert J \rVert^2$
% (see Lemma~\ref{lem:norm}).
It is generally difficult to spell out the dependency of $\hat C_{\mathbf{z},J}$ in \eqref{cmp:erm-nonrev} on the matrix $J$ for nonconvex problems. On the other hand,
\cite{RRT17} showed that $\hat C_{\mathbf{z},J=0}= \tilde{\mathcal{O}}( 1/\sqrt{ |\lambda_{\mathbf{z},J=0} |} )$, where $\frac{1}{|\lambda_{\mathbf{z},J=0}|} = e^{\tilde{\mathcal{O}}(\beta+d)}$; see also \cite{BGK05}. Hence in the following discussion we will assume the first two terms in \eqref{cmp:erm-nonrev} are both of the order $\tilde{\mathcal{O}}( 1/\sqrt{|\lambda_{\mathbf{z},J=0}|} )$.

%  If the last term in \eqref{cmp:erm-nonrev} is the dominant term, then the empirical risk upper bounds for SGLD and NSGLD will be similar except that
% $\mathcal{K}_{J}$ is smaller than $\mathcal{K}_{J=0}$ (iterations for SGLD) since generically $\lambda_{*,J}  > \lambda_{*,J=0}$ for nonzero $J$.
%  If the first term in \eqref{cmp:erm-nonrev} dominates, we can not determine whether and when NSGLD can performance better than the SGLD algorithm since the dependancy of $\hat C_{\mathbf{z},J} $ on $J$ is unknown. 

%  The more interesting case is when the second term in \eqref{cmp:erm-nonrev} dominates. On one hand, adding an antisymmetrix matrix $J$ improves the rate of convergence of diffusions where $\lambda_{*,J}  > \lambda_{*,J=0}$. On the other hand, it also gives rise to an increase in the discretizatoin error by a factor of $ \lVert A_J \rVert >1$ for nonzero $J$. The overall effect will depend on how the quantity $\frac{ \lVert A_J \rVert}{ \sqrt{\lambda_{*,J}}}$ appearing in the second term of \eqref{cmp:erm-nonrev} varies with $J$. We next provide further discussions on this trade-off. We focus on the regime $\beta \rightarrow \infty$ so that the uniform spectral gap $\lambda_{*,J}$ is computable. 

Following \cite{RRT17} and \cite{xu2018global}, we can infer from \eqref{cmp:erm-nonrev} and \eqref{eq:iteratoins-KJ} that the gradient complexity of NSGLD with anti-symmetric matrix $J$ is
\begin{equation}  \label{eq:gradient-comp}
%\hat{ \mathcal{K}}_J:=\mathcal{K}_J \cdot \ell =\mathcal{K}_J/\eta=  \tilde{\Omega}\left( \frac{ \beta\sqrt{\beta+d}}{ {\lambda_{*,J} } \, \varepsilon^8}\log^5 \left(\frac{1}{\varepsilon}\right) \cdot \frac{\Vert A_J\Vert^8 \lambda_{*,J=0}^4 }{\lambda_{*,J}^4} \right)
\hat{ \mathcal{K}}_J:=\mathcal{K}_J \cdot \ell =\mathcal{K}_J/\eta=  \tilde{\Omega}\left( \frac{ \sqrt{\beta}(\beta+d)}{ {|\lambda_{\mathbf{z},J} |} \, \varepsilon^8}\log^5 \left(\frac{1}{\varepsilon}\right) \cdot \frac{\Vert A_J\Vert^8 \lambda_{\mathbf{z},J=0}^4 }{\lambda_{\mathbf{z},J}^4} \right).
\end{equation}
Hence to compare $\hat{ \mathcal{K}}_J$ with $\hat{ \mathcal{K}}_{J=0}$, we compare $ \frac{\Vert A_J\Vert^8 \lambda_{\mathbf{z},J=0}^4 }{|\lambda_{\mathbf{z},J}|^5}$ with $1/|\lambda_{\mathbf{z}, J=0}|$ (when $J=0$). To this end, we next consider the asymptotic setting with $\beta \rightarrow \infty$ and present formulas of $\lambda_{\mathbf{z},J}$.  %In the next subsections, we compare the spectral gaps and the gradient complexities.

\textbf{Formulas of the spectral gaps $\lambda_{\mathbf{z}, J}$ and $\lambda_{\mathbf{z}, J=0}$ when $\beta \rightarrow \infty$.} 
% We summarize formulas of $\lambda_{\mathbf{z}, J=0}$ and $\lambda_{\mathbf{z}, J}$ in this section.
Suppose $F_{\mathbf{z}}$ is a Morse function admitting finite number of local minima, where the Hessian of $F_{\mathbf{z}}$ are non-degenerate (i.e. invertible) at all stationary points \cite{MBM}. 
% For the reversible SDE in \eqref{reversibleSDE}, \cite{BGK05} studied precise asymptotics for small eigenvalues of its generator $\mathcal{L}_0= - \nabla F_{\mathbf{z}} \cdot \nabla + \beta^{-1} \Delta$ as $\beta \rightarrow \infty$. 
Assuming all the valleys of $F_{\mathbf{z}}$ have different depths, there is one saddle point connecting two local valleys or minima, and the Hessian at the saddle points has one negative eigenvalue with other eigenvalues positive, \cite[Theorem 1.2]{BGK05} showed that for the reversible SDE in \eqref{reversibleSDE}, as $\beta \rightarrow \infty$, the spectral gap is given by
\begin{equation}
|\lambda_{\mathbf{z},J=0}| = \frac{\mu^{*}(\sigma)}{2\pi}\sqrt{\frac{\det  \text{Hess}   F_{\mathbf{z}}(a)}{\lvert \det  \text{Hess}   F_{\mathbf{z}}(\sigma)\rvert}} e^{-\beta[F_{\mathbf{z}}(\sigma)-F_{\mathbf{z}} (a)]} \cdot\left[1+\mathcal{O}\left(\beta^{-1/2}\lvert \log (1/\beta)\rvert\right)\right]\,. \label{ex:rev-sp}
\end{equation} 
Here, $a$ is a local minimum of $F_{\mathbf{z}}$ with second deepest valley ($a$ is just the local, not the global, minimum of $F_{\mathbf{z}}$ if $F_{\mathbf{z}}$ only has two local minima), $\sigma$ is the saddle point connecting $a$ and the global minimum of $F_{\mathbf{z}}$, and   
$-\mu^{*}(\sigma)$ is the unique negative eigenvalue of the Hessian of $F_{\mathbf{z}}$ at the saddle point $\sigma.$ For precise definitions of these quantities, see \cite{BGK05}. For the non-reversible Langevin SDE in \eqref{nonreversibleSDE}, under similar assumptions on $F_{\mathbf{z}}$, \cite[Theorem 1.9]{PM19} showed that, %the associated spectral gap is given by
\begin{equation}
|\lambda_{\mathbf{z},J} | = \frac{\mu_J^{*}(\sigma)}{2\pi}\sqrt{\frac{\det  \text{Hess}  F_{\mathbf{z}}(a)}{\lvert \det  \text{Hess}  F_{\mathbf{z}}(\sigma)\rvert}} e^{-\beta[F(\sigma)-F_{\mathbf{z}}(a)]} \cdot\left[1+\mathcal{O}\left(\beta^{-1/2}\right)\right] \,,  \label{ex:nonrev-sp}
\end{equation}
where $-\mu^{*}_J(\sigma)$ is the unique negative eigenvalue of $A_J \cdot \mathbb{L}^{\sigma}$, where $\mathbb{L}^{\sigma} = \text{Hess} F_{\mathbf{z}}(\sigma)$.
%$\mathbb{L}^{\sigma}$ is the Hessian of $F_{\mathbf{z}}$ at the saddle point $\sigma $, i.e.  

\textbf{Comparison of gradient complexities of NSGLD and SGLD.} 
We now compare the gradient complexity $\hat{ \mathcal{K}}_J$ of the NSGLD algorithm with $\hat{ \mathcal{K}}_{J=0}$ of the SGLD algorithm. 
It is clear from \eqref{eq:gradient-comp} that for a nonzero antisymmetrix matrix $J$, we have $\hat{ \mathcal{K}}_J< \hat{ \mathcal{K}}_{J=0}$ if 
 $ \frac{\Vert A_J\Vert^8 \lambda_{\mathbf{z},J=0}^4 }{|\lambda_{\mathbf{z},J}|^5} < \frac{1}{|\lambda_{\mathbf{z}, J=0}|}$.
From \eqref{ex:rev-sp} and \eqref{ex:nonrev-sp} and using $\Vert A_J\Vert^2 = 1 + \Vert J\Vert^2$, we obtain that
% \begin{equation} \label{eq:ratio-J}
% \left( \frac{\Vert A_J\Vert/  \sqrt{\lambda_{*, J}}} {1/ \sqrt{\lambda_{*, J=0}}} \right)^2= (1+\Vert J\Vert^2) \cdot \frac{\mu^*(\sigma)}{ \mu_J^{*}(\sigma)}.    
% \end{equation}
%Note that if we choose $a_i$ such that $a_1^2 \ge a_i^2$ for each $i$, then 
\begin{align} \label{eq:ratio-J}
\frac{\Vert A_J\Vert^8 \lambda_{\mathbf{z},J=0}^5 }{|\lambda_{\mathbf{z},J}|^5} & = \left(1+\Vert J\Vert^2\right)^4 \cdot \left( \frac{\mu^*(\sigma)}{ \mu_J^{*}(\sigma)} \right)^5.
% & = (1+  \max_{1 \le i \le d/2 }a_i^2 )^4 \cdot  \left(\frac{2}{\sqrt{(\lambda_1-1)^2+4(a_1^2+1)\lambda_1}-(\lambda_1-1)} \right)^5 \nonumber \\
% & = (1+ a_1^2 )^4 \cdot  \left(\frac{2}{\sqrt{(\lambda_1-1)^2+4(a_1^2+1)\lambda_1}-(\lambda_1-1)} \right)^5 \nonumber 
% & = \frac{1}{(1+ a_1^2 )} \cdot  \left(\frac{\sqrt{(\lambda_1-1)^2+4(a_1^2+1)\lambda_1}+ (\lambda_1-1)}{2 \lambda} \right)^5 \nonumber \\
\end{align}
We study when the quantity in \eqref{eq:ratio-J} is smaller than one so that NSGLD can outperform SGLD with $J=0$ in terms of gradient complexity.
Without loss of generality, we consider a diagonal Hessian matrix $\mathbb{L}^{\sigma}$ at the saddle point where $\mathbb{L}^{\sigma}=\text{diag}\{-1, \lambda_1, \lambda_2, \ldots, \lambda_{d-1} \}$ with $\lambda_i >0$ and $\mu^{*}(\sigma)=1$. The case of general symmetric Hessian matrix can be handled similarly using the fact that $QJQ^T$ remains to be anti-symmetic if $Q$ is an orthogonal matrix and $J$ is anti-symmetric.
% We next compute $ \mu_J^{*}(\sigma)$ and $ \Vert J\Vert^2$ for a particular class of anti-symmetric matrices $J$, where $-\mu^{*}_J(\sigma)$ is the unique negative eigenvalue of $(I+J)\cdot D$. 
% Before we proceed, we remark that the general case can be handled similarly. The Hessian matrix $\mathbb{L}^{\sigma}$ is symmetric and diagonalizable, so that $\mathbb{{L}}^{\sigma} = QDQ^{T}$ (up to permutations)
% for some orthogonal matrix $Q$. The eigenvalues of $(I+QJQ^T) \mathbb{{L}}^{\sigma}= Q[ (I+ J) D] Q^T$ are the same as the eigenvalues of $(I+J)D$. In addition, $QJQ^T$ remains to be an antisymmetic matrix when $J$ is anti-symmetric. 
We consider the anti-symmetric matrix $J$ has a block diagonal structure that allows explicit computations. Suppose the dimension $d$ is an even number, and we consider a block diagonal anti-symmetric matrix $J= \text{diag}\{M_1, M_2, \ldots, M_{d/2}\}$, where $M_i=  \begin{bmatrix}
        0  &  a_i \\
        -a_i & 0\\
        \end{bmatrix}$ and $a_i \in\mathbb{R}$. The case of $d$ is odd can be handled similarly by removing the last row and the last column of the $J$ matrix.
If we choose $a_i$ such that $a_1^2 \ge a_i^2$ for each $i=2, \ldots, d/2$,
one can readily verify that the quantity in \eqref{eq:ratio-J} is smaller than one if and only if
% \begin{equation}
%  2 \left(1+ a_1^2 \right)^{4/5} + (\lambda_1 -1) <  \sqrt{(\lambda_1-1)^2+4\left(a_1^2+1\right)\lambda_1},
% \end{equation}
% which is equivalent to
% % \begin{equation}
% %   4(1+ a_1^2 )^{8/5} + 4 (\lambda_1 -1) (1+ a_1^2 )^{4/5} <  4(a_1^2+1)\lambda_1 ,
% % \end{equation}
% % or 
% \begin{equation}
%   \left(1+ a_1^2 \right)^{3/5} +  (\lambda_1 -1) \left(1+ a_1^2 \right)^{-1/5} <  \lambda_1 ,
% \end{equation}
% which is equivalent to
$\lambda_{1}>\frac{(1+a_{1}^{2})^{3/5}-(1+a_{1}^{2})^{-1/5}}{1-(1+a_{1}^{2})^{-1/5}}
%=\frac{(1+a_{1}^{2})^{4/5}-1}{(1+a_{1}^{2})^{1/5}-1}
= \left( 1+ (1+ a_1^2 )^{2/5} \right) \cdot \left( 1+ (1+ a_1^2 )^{1/5} \right)$.
Hence when $\lambda_1>4$, we can choose $a_1>0$ small so that $\hat{ \mathcal{K}}_J< \hat{ \mathcal{K}}_{J=0}$. Thus NSGLD can reduce the gradient complexity compared with SGLD when the anti-symmetric matrix $J$ is properly chosen.

\section{Numerical Experiments}\label{sec:simple}

In this section, we conduct several experiments to assess the performance of NSGLD algorithm and compare it with SGLD algorithm via three examples: a simple non-convex example in dimension two, Bayesian Independent Component Analysis and Neural Network models.

%%%%%%%%%%%%%%%%%%%%%%%%%%%%%%%%%%%%%%%%%%%%%%%%%%%%%%%%%
%\subsection{A two-dimensional example}
\textbf{A two-dimensional example.} We first demonstrate the performance of the NSGLD algorithm on a simple example, where the objective is a two-dimensional non-convex function given by
% . Through this simple example we can learn how the NSGLD algorithm works on a non-convex optimization problem, and how it may outperform the SGLD method. We take the objective function as:
\begin{equation} \label{eq:ex1}
    f(x)=
    \begin{cases}
    \frac{1}{4}-\frac{\left\lVert x \right\rVert^2}{2}+\alpha x & \text{if $\left\lVert x \right\rVert \leq \frac{1}{2}$}\\
    \frac{1}{2}(\left\lVert x \right\rVert-1)^2+\alpha x & \text{if $\left\lVert x \right\rVert > \frac{1}{2}$}
    \end{cases}\,
    \quad \text{where}\ x\in \R^2,\ \alpha=(0.2,0.2).
\end{equation}
Since the function lives on $\R^2$, the $2\times 2$ anti-symmetric matrix $J$ must be of the form
    $ J=  \begin{bmatrix}
        0  &  \tau \\
        -\tau & 0\\
        \end{bmatrix}$, where $\tau\in\mathbb{R}$.
The function in \eqref{eq:ex1} is non-convex and has two minima. One is the local minimum $(\frac{1}{5},\frac{1}{5})$ and function value is 0.29. The other is the global minimum $(-\frac{\sqrt{2}}{2}-\frac{1}{5},-\frac{\sqrt{2}}{2}-\frac{1}{5})$ and minimal value is -0.3228. The contour plot of the function is given in Figure~\ref{fig:exp1_1} with two minima shown on the plot.
In the experiments, the initial point of the NSGLD algorithm is assigned to $(1,1)$ with corresponding function value $0.4858$, and this starting point is near the local minimum. We tuned the SGLD method and found the optimal step size is $1$ and $\beta=200$. We also used the same step size and $\beta$ in the NSGLD method. We compare the SGLD method and NSGLD method with different $\tau$ values. To see the expectation of the suboptimality, we use $50$ samples and calculate the average over these samples. 
%Since both SGLD and NSGLD method include the random noise $\xi$ in the iteration, we use the same random source for both methods in each sample.

The results are shown in Figure~\ref{fig:exp1_2}, which shows the expected function value of SGLD and NSGLD iterates with different $\tau$'s. We observe that NSGLD can outperform SGLD with proper choices of $\tau$, and one can tune $\tau$ to achieve faster convergence in this experiment. On the other hand, we also observe that $\tau$ can not be too big. In Figure~\ref{fig:exp1_2}, when $\tau=1.612$ the function will not converge to the global minimum; when $\tau$ increases further, the objective function will go to infinity.

\begin{figure}[t]
\centering
    \subfigure[Contour of function]{
    \includegraphics[width=0.35\columnwidth]{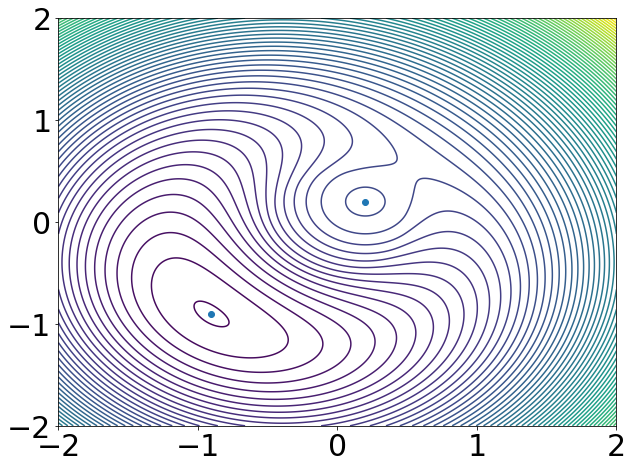}
    \label{fig:exp1_1}
    }
    \hfill
    \subfigure[Comparison]{
    \includegraphics[width=0.35\columnwidth]{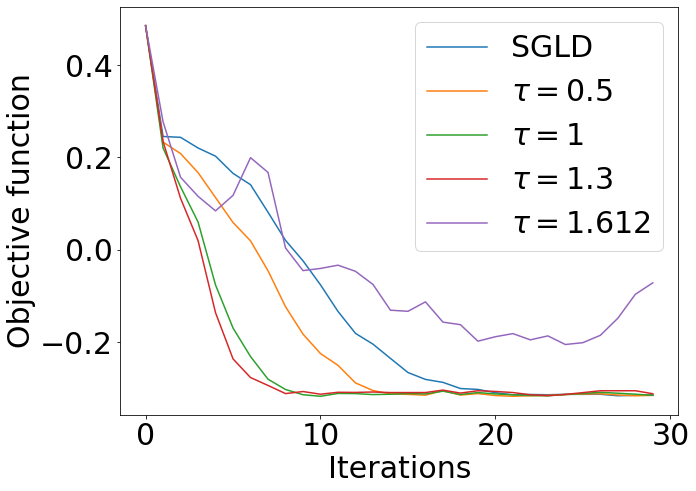}
    \label{fig:exp1_2}
    }
\caption{Performances of SGLD and NSGLD on the simple non-convex problem in \eqref{eq:ex1}.}
\label{fig:exp1}
\end{figure}

%%%%%%%%%%%%%%%%%%%%%%%%%%%%%%%%%%%%
%\subsection{Bayesian Independent Component Analysis}
\textbf{Bayesian Independent Component Analysis.} The Bayesian ICA attempts to decompose a multivariate signal into independent non-Gaussian signals and arises commonly in machine learning applications \cite{HYVARINEN2000411,1058079}. In the following, we will briefly review the Bayesian ICA model and compare the performance of SGLD with NSGLD. Given the data set $\{x_i;i=1,2,...,m\}$ and $x_i \in \R ^n$, Bayesian ICA aims to recover the independent sources $s=Wx$, where $s\in \R^n$ and $W\in \R ^{n\times n}$. We assume that the distribution of each independent component source $s_i$ is given by the density $p_s(s_i)$. The joint distribution of the sources $s$ is given by $p(s)=\prod_{i=1}^{n} p_s(s_i)$. Then the log likelihood is given by:
$\ell (W)=\sum _{i=1}^{m} \left(\sum _{j=1}^n \log{g'(w_j^T x_i)}+\log{\left\lVert W \right\rVert}\right)$, where $w_j$ is the $j$-th column of matrix $W$.
The goal then becomes finding the optimal unmixing matrix $W$ which maximizes the log likelihood \cite{mackay1996maximum}. In our experiments, we used two datasets: the Iris plants dataset\protect\footnotemark[1] and the Diabetes dataset\protect\footnotemark[2]. The Iris plants dataset consists of 3 different types of irises? (Setosa, Versicolour, and Virginica) petal and sepal length. The Diabetes dataset consists of 10 baseline variables, age, sex, body mass index, average blood pressure, and six blood serum measurements for each diabetes patient. For these datasets, the Bayesian ICA model can extract a number of features from the original data and can improve subsequent tasks such as classification and regression \cite{huang2002efficient,wang2005iris,bang2011independent}. In our experiments, we let the distribution $s_i$ follow the sigmoid function, i.e. $p_s(s)=g'(s)$ where $g(s)=1/(1+\exp(-s))$. We choose the anti-symmetric matrix $J$ randomly according to
$J=(j_{m,n})_{1\leq m,n\leq d}$, 
where $j_{m,m}:=0$ for $1\leq m\leq d$,
$j_{m,n}\sim\mathcal{N}(0,\frac{\tau^{2}}{d^{2}})$
for $m<n$, and $j_{m,n}:=-j_{n,m}$ for $m>n$.
%\begin{equation}
%    J=  \begin{bmatrix}
%        0  &  j_{1,2}  & \cdots\ &j_{1,d}\\
%        -j_{1,2}  &  0  & \cdots\ & j_{2,d}\\
%        \vdots  & \vdots & \ddots & \vdots \\
%        -j_{1,d} & -j_{2,d}  & \cdots\ & 0\\
%        \end{bmatrix},\ \qquad\text{where}\ j_{m,n}\sim \mathcal N \left(0,\frac{\tau^2}{d^2}\right)\ %\text{for}\ m\neq n.
%    \label{Jvalue}
%\end{equation}

%Since both SGLD and NSGLD methods include the random noise $\xi$ in the iteration, we use the same random source for both methods in each sample. 
In order to compute the expectation of the suboptimality, we run both methods over 20 times with i.i.d. samples at every iteration and calculate the average over these runs. For both datasets, we chose a decaying stepsize of the form $a/(b+ct)$ for SGLD and tune the constants $a$, $b$ and $c$ to the dataset. We used the same stepsize for NSGLD and tuned the constant $\tau$ %in~\eqref{Jvalue} 
to the dataset as well. For the Iris plants dataset, the tuned parameters were $a=0.01$, $b=1$, $c=0.1$, $\beta=200$ and $\tau=1$ whereas for the Diabetes dataset we used the values $a=0.1$, $b=1$, $c=0.1$, $\beta=200$ and $\tau=7$.
The results are shown in Figure~\ref{fig:exp2_1} and Figure~\ref{fig:exp2_2}.
%%
%In the Iris plants dataset, we tuned the SGLD method and set the decaying step size equal to $0.01/(1+0.1t)$, where $t=1,2,...,T$ is the iteration number. We used the same stepsize for the NSGLD method. We tune the parameter $\tau$ in~\eqref{Jvalue} to the dataset, where we choose $\tau=0.02$. The result of Iris is shown in Figure~\ref{fig:exp2_1}. In the Diabetes dataset, we tuned the decaying stepsize to be $0.1/(1+0.1t)$, where $t=1,2,...,T$ is the iteration number. We tuned $\tau$ in~\eqref{Jvalue} to be $7$ and the same stepsize in NSGLD method. The result of Diabetes is shown in Figure~\ref{fig:exp2_2}. 
In both experiments, we can observe that the NSGLD algorithm converges faster than the SGLD method in the ICA task.

\footnotetext[1]{The dataset is available \url{https://archive.ics.uci.edu/ml/datasets/Iris}.}
\footnotetext[2]{The dataset is available \url{https://archive.ics.uci.edu/ml/datasets/diabetes}.}

\begin{figure}[t]
\centering
    \subfigure[Iris dataset]{
    \includegraphics[width=0.35\columnwidth]{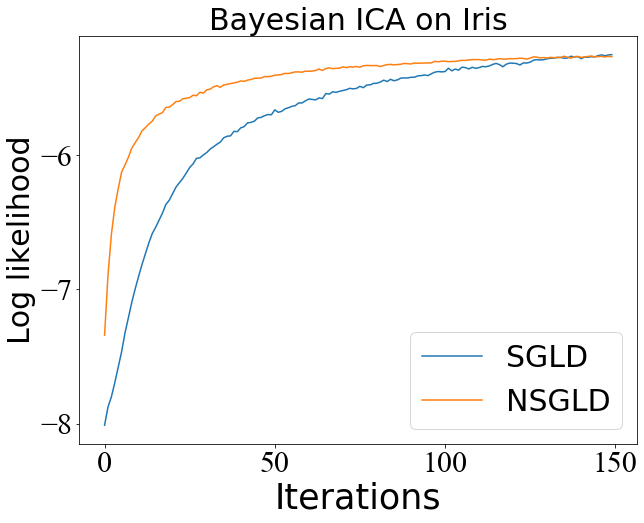}
    \label{fig:exp2_1}
    }
    \hfill
    \subfigure[Diabetes dataset]{
    \includegraphics[width=0.35\columnwidth]{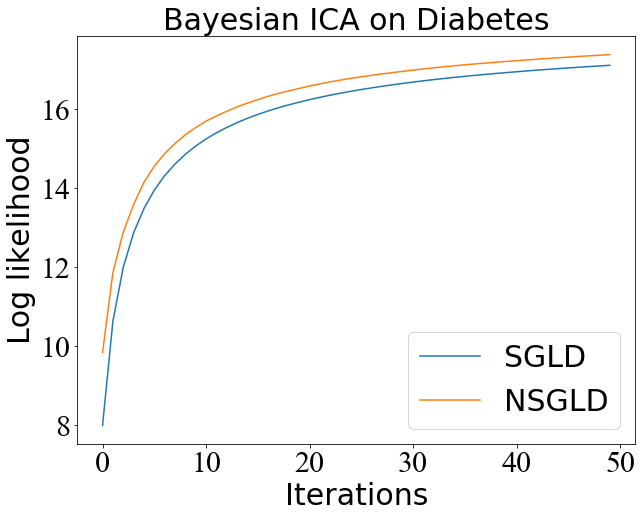}
    \label{fig:exp2_2}        
    }
    \hfill
\caption{SGLD and NSGLD on the Bayesian ICA.}
\label{fig:exp2}
\end{figure}

%%%%%%%%%%%%%%%%%%%%%%%%%%%%%%%%%%%
%\subsection{Neural Network Model}
\textbf{Neural Network Model.}
\begin{figure}[t]
\centering
    \subfigure[Fully-connected NN]{
    \includegraphics[width=0.3\columnwidth]{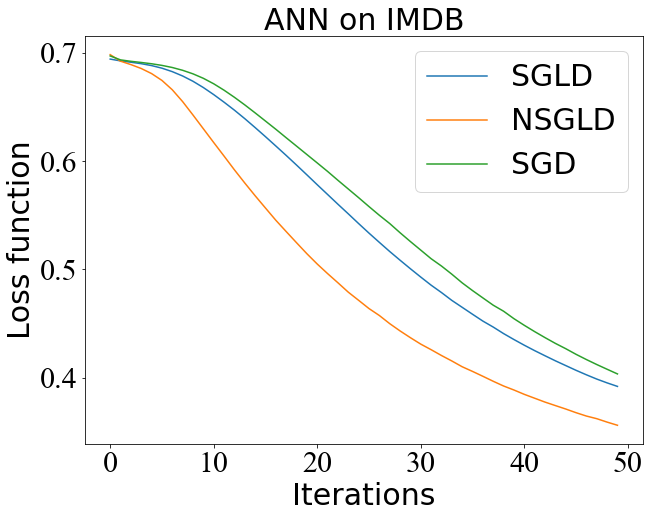}
    \label{fig:exp3_1}
    }
    \hfill
    \subfigure[LSTM]{
    \includegraphics[width=0.3\columnwidth]{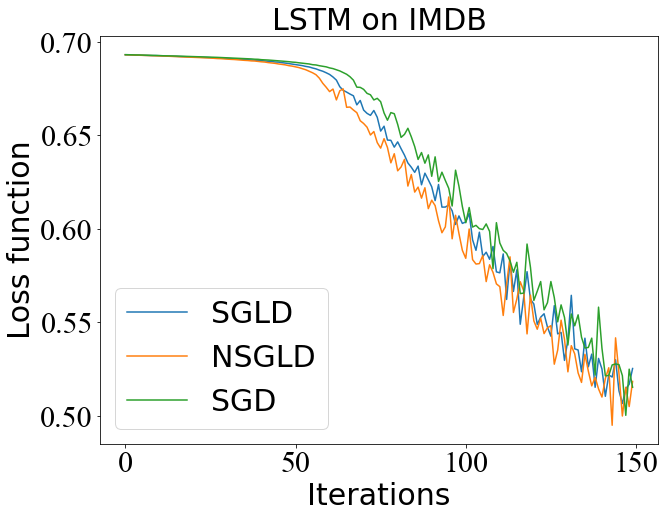}
    \label{fig:exp3_2}
    }
    \hfill
    \subfigure[CNN LSTM]{
    \includegraphics[width=0.3\columnwidth]{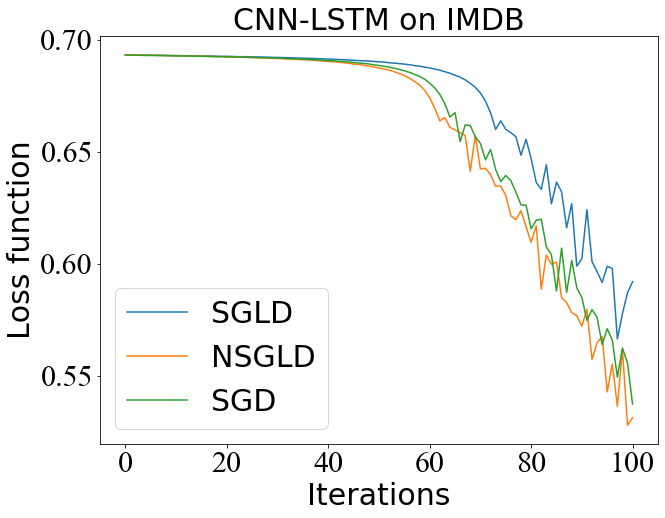}
    \label{fig:exp3_3}
    }
\caption{SGD, SGLD and NSGLD on the Neural Network model.}
\label{fig:exp3}
\end{figure}
In the next set of experiments, we focus on applying the methods on the Neural Network model. All the experiments are based on the IMDB dataset\protect\footnotemark[3]. The IMDB dataset contains 25,000 movies reviews, and reviews are labeled by sentiment (positive/negative). The purpose of the Neural Network model is to do the classification based on the IMDB dataset. We will test our NSGLD algorithm and compare it with stochastic gradient descent (SGD) and SGLD.

We test three Neural Network structures on this dataset. The first one is the Fully-connected Neural Network, which has one hidden layer, and the result is shown in Figure~\ref{fig:exp3_1}. The second one is the Long Short-Term Memory (LSTM) Neural Network, and the result is shown in Figure~\ref{fig:exp3_2}. The third one is the Convolutional Neural Network and Long Short-Term Memory (CNN LSTM) Neural Network, and the result is shown in Figure~\ref{fig:exp3_3}. In all these experiments, the step size is $0.1$, the batch size is $1000$ and $\beta=10^{6}$. We use the antisymmetric matrix $J$ %in~\eqref{Jvalue} 
with $\tau =0.5$ for Fully-connected Neural Network, $\tau=2$ for LSTM Network, and $\tau=0.1$ for CNN LSTM Network. We again observe that NSGLD can outperform SGLD and SGD in different model architectures.

\footnotetext[3]{The dataset is available \url{https://datasets.imdbws.com/}.}

%%%%%%%%%%%%%%%%%%%%%%%%%%%%%%%%%%%%%%%%%%%%%%%%%%%%%%%
% \section{Conclusion}

% In our paper, we studied a non-reversible stochastic gradient Langevin dynamics (NSGLD) by adding an anti-symmetric matrix
% to the drift term of the Langevin dynamics. We provided finite-time performance bounds for the global convergence of NSGLD for solving stochastic non-convex optimization problems.
% Our results led to non-asymptotic guarantees for both population and empirical risk minimization problems.
% We conducted numerical experiments for several problems including a simple polynomial function optimization, Bayesian independent component analysis and neural network models, and  
% we showed that NSGLD can outperform SGLD
% with proper choices of the anti-symmetric matrix.

%%%%%%%%%%%%%%Reference%%%%%%%%%%%%%%%%%%%%

\bibliographystyle{alpha}
\bibliography{langevin}

%%%%%%%%%%%%%%%%Reference%%%%%%%%%%%%%%

%%%%%%%%%%%%%%%%Appendix%%%%%%%%%%%%%%

\pagebreak
\appendix

%%%%%%%%%%%%%%%%%%%%%%%%%%%%%%%

%%%%%%%%%%%%%%%%%%%%%%%%%%%%%%%
\section{Proofs of Theorem~\ref{thm:nonreversible2gibbs} and Corollary~\ref{cor:NSGLD2gibbs}} \label{sec:proof-main}
%%%%%%%%%%%%%%%%%%%%%%%%%%
\subsection{Technical Lemmas} \label{sec:tech-lemmas}
We first state a few technical lemmas and corollaries that will be used in the proofs of Theorem~\ref{thm:nonreversible2gibbs} and Corollary~\ref{cor:NSGLD2gibbs}. Their proofs are deferred to the appendix.

To prove Corollary~\ref{cor:NSGLD2gibbs}, we need the following results.

\begin{lemma}[Uniform $L^{2}$ bounds on NSGLD  \cite{GGZ18-2} and non-reversible Langevin SDE] \label{lm:uniform-L2}
Under Assumptions~\ref{A1},~\ref{A2},~\ref{A3} and \ref{A5}. For any data set $\mathbf{z} \in \mathcal{Z}^n$,
\begin{equation}
\sup_{t>0}\mathbb{E}_{\mathbf{z}}\lVert X(t) \rVert^2 \leq \mathcal{C}_c :=\frac{3MR^2+3BR+3B+6A+3b\log3}{2m} + \frac{3b(M+B)}{m^2} + \frac{6M\beta^{-1}d(M+B)}{m^3} \,.
\label{L2-SDE}
\end{equation}
Moreover, for $\eta \leq \frac{m^2}{(m^2+ 8M^2)M\lVert A_J\rVert^2} \wedge 1 $, we have
\begin{align} 
\sup_{k>0}\mathbb{E}_{\mathbf{z}}\lVert X_k \rVert^2 
\leq \mathcal{C}_d:= &  \frac{3MR^2 + 6BR + 3B + 6A + 3b\log3}{2m} +\frac{6\delta(2bM^2 + B^2m)(M+B)}{m^4} \,\nonumber \\
& \qquad + \frac{12M\beta^{-1}d(M+B)}{m^3} + \frac{3b(M+B)}{m^2}   \,. \label{L2-NSGLD}
\end{align}
\end{lemma}

\begin{lemma}[Exponential integrability of non-reversible Langevin SDE]\label{lm:exponential}
If Assumptions~\ref{A1},~\ref{A2},~\ref{A3} and \ref{A5} hold, given  $\beta \geq 3/m$, for any $t\geq 0$, the exponential integrability of non-reversible SDE admits,
\begin{equation}
\mathbb{E}_{\mathbf{z}} \left[ e^{\lVert X(t) \rVert ^2} \right] \leq L_0 + L_1 \cdot t < \infty \,,
\end{equation}
where
\begin{align}
& L_0 := \exp\left(\frac{MR^2}{2}+BR+A + \frac{3b}{2m}\log 3\right) \label{lm:const:l0} \,, \\
& L_1:=\frac{(3m-9\beta^{-1})(B/2+A)+3b(M+B)}{2(M+B)} \nonumber \\
& \qquad\qquad + \frac{6\beta^{-1}Md-9b\beta^{-1}}{2m} \exp\left\{ \frac{3}{m}\left(\frac{B}{2}+A + \frac{M+B}{m-3\beta^{-1}}\left( b+\frac{\beta^{-1}(2Md-3b)}{m} \right)\right) \right\} \,. \label{lm:const:l1}
\end{align}
\end{lemma}

Before we state the next lemma, 
let us first introduce the definition of the 2-Wasserstein distance, 
which is a common choice measuring the distance between two probability measures. 
For any two probability measures $\mu$, $\nu$,
the 2-Wasserstein distance is defined as:
\begin{equation*}
\mathcal{W}_2\left( \mu,\nu \right) = \left( \inf_{U \sim \mu, V \sim \nu} \mathbb{E}\lVert U-V \rVert^{2} \right)^{\frac{1}{2}} \,,
\end{equation*}
where the infimum is taken over all random couplings of $U \sim \mu$ and $V \sim \nu$, 
with the marginal distributions being $\mu$ and $\nu$. 

\begin{lemma}[Diffusion approximation]\label{lm:approximation} 
Suppose Assumptions~\ref{A1}-\ref{A5} hold, let $\mu_{\mathbf{z},k}   $ be the probability law of $X_k $ in \eqref{NSGLD} and $ \nu_{\mathbf{z},k \eta}  $ be the probability law of $X(k\eta)$ in \eqref{nonreversibleSDE}. For any $ \eta \leq \frac{m^2}{(m^2+ 8M^2)M\lVert A_J\rVert^2}  \wedge 1 $ such that $k\eta \geq e $, the diffusion approximation under 2-Wasserstein metric is
\begin{equation} 
\mathcal{W}_2(\mu_{\mathbf{z},k},\nu_{\mathbf{z},k\eta}) \leq \left(\hat{C_0}\eta^{1/4}+\hat{C_1}\delta^{1/4}\right) \sqrt{k\eta}\sqrt{\log(k\eta)} \lVert A_J \rVert\,,
\end{equation}
where 
\begin{align}
\hat{C_0}:= \left(16\log\left(L_0+L_1\right) \left(C_0+\sqrt{C_0} \right)\right)^{1/2}\,, \label{lm:const:C0hat} \\
\hat{C_1}:=\left(16\log\left(L_0+L_1\right) \left(C_1+\sqrt{C_1} \right)\right)^{1/2}\,, \label{lm:const:C1hat}
\end{align}
with
\begin{equation} 
C_0:=2\beta M^2 \left( M^2 \mathcal{C}_d + B^2 + d\beta^{-1} \right), \qquad C_1:=(1+2M^2)\beta \left( M^2 \mathcal{C}_d +B^2 \right)\,, \label{lm:const:approximation}
\end{equation}
where $L_0, L_1$ are constants defined in Lemma~\ref{lm:exponential}, and $\mathcal{C}_d  $ is defined in \eqref{L2-NSGLD}.
\end{lemma}

Lemma~\ref{lm:approximation} states that NSGLD recursion \eqref{NSGLD} tracks the
continuous-time non-reversible Langevin SDE \eqref{nonreversibleSDE} in 2-Wasserstein distance. This lemma 
is the key ingredient in the proof of Corollary~\ref{cor:NSGLD2gibbs}, and its proof relies on Lemmas~\ref{lm:uniform-L2} and ~\ref{lm:exponential}.

To prove Theorem~\ref{thm:nonreversible2gibbs}, we need the following three results.
\begin{lemma} [Uniform $L^4$ bound on non-reversible Langevin SDE] \label{lm:uniform-L4}
With Assumptions~\ref{A1}, ~\ref{A2}, ~\ref{A3} and \ref{A5}, we have
\begin{align}
\sup_{t>0}\mathbb{E} \lVert X(t) \rVert^4 \leq \mathcal{D}_c &:= \frac{9}{m^2}\left(\frac{M}{2}R^2+BR + A \right)^2 + \frac{9U+9b(M+B)\mathcal{C}_c}{m^2} \nonumber \\
& \qquad +  \frac{6M(M+B)^2}{m^3}\left(B+2B\sqrt{2b/m} +  \frac{2bM}{m} +  4A  \right) \beta^{-1}d \,, \label{lm:const:Dc}
\end{align}
where $\mathcal{C}_c$ is given in Lemma~\ref{lm:uniform-L2} and
\begin{align}
& U:= \frac{(B+2A)^2}{2} + \frac{18(M+B)^2}{m^2}\left(b+\beta^{-1}+\frac{2(M+B)\beta^{-1}}{m^2} \right)^2 \, \nonumber \\
& \qquad\qquad + \frac{24\beta^{-1}(2bM^2+mB^2)(M+B)^2}{m^4} + 2bB + 2A + b^2 \,. \label{lm:const:U}
\end{align}
\end{lemma}

%\cite{HHS05} analyze the diffusion acceleration and use variational norm $\rho$ to describe the convergence rate. They provide the next lemma with the general assumptions: First, the lower-order coefficient in $\mathcal{L}$ is weakly divergence free, i.e. for any smooth function with compact support $f \in C_0^{\infty} $, there is $\int \left(-J\nabla F \cdot \nabla f\right) \pi = 0$. Second, the existence of spectral gap, i.e. $\lambda_{J=0} < 0$.  

%\begin{lemma}[\cite{HHS05} Theorem 4]\label{Hwang:main2} 
%If there exists an invariant measure $\pi_{\mathbf{z}}$ and $\mathcal{L}_{\mathbf{z},J=0}$ has a spectral gap $\lambda_{\mathbf{z},J=0}$ in $L^2(\pi_{\mathbf{z}})$ and the lower-order coefficient in $\mathcal{L}_{\mathbf{z},J}$ is weakly divergence free. Suppose Assumption~\ref{A1}, ~\ref{A2} and \ref{A5} hold, then 
%\begin{equation} 
%\int_{\mathbb{R}^d} \left| p_{\mathbf{z},J}(t,x,y) - \pi_{\bz}(y) \right| dy \leq g_{\mathbf{z},J}(x) \cdot e^{\lambda_{\mathbf{z},J}t} \,,
%\label{lm:eqn:hwang:acceleration}  
%\end{equation}
%where $g_{\mathbf{z},J}(x)$ is a locally bounded function (with respect to $x$).
%\end{lemma}

 Write $p_{\mathbf{z},J}(t,x,y)$ for the transition probability of the SDE in \eqref{nonreversibleSDE} from state $x$ to state $y$ in $t$ units of time.
\cite[Theorem 4]{HHS05} showed that
 there exists a locally bounded function $g_{\mathbf{z},J}(x)$ that may depend on $\mathbf{z}$ and $J$, such that 
\begin{equation} 
\int_{\mathbb{R}^d} \left| p_{\mathbf{z},J}(t,x,y) - \pi_{\mathbf{z}}(y) \right| dy \leq \text{constant} \cdot g_{\mathbf{z},J}(x) e^{\rho_{\mathbf{z},J} t}. \,\label{variationalnorm}
\end{equation} 
\cite{HHS05} did not specify the function $g_{\mathbf{z},J}(x)$ which comes from a local Harnack inequality (see e.g. \cite{Trudinger68}).
In the following Lemma~\ref{lm:acceleration}, we build upon Hwang \cite[Theorem 4]{HHS05} and discuss
the dependence of $g_{\mathbf{z},J}(x)$ on $\mathbf{z}$ and $J$
by applying a Harnack inequality with a more transparent Harnack constant in \cite{BRS2008}. We have the following result.

\begin{lemma}\label{lm:acceleration} 
Suppose Assumption~\ref{A1} and~\ref{A2} hold, then 
\begin{equation} 
\int_{\mathbb{R}^d} \left| p_{\mathbf{z},J}(t,x,y) - \pi_{\bz}(y) \right| dy \leq C_{\mathbf{z}, J}\cdot g_{\mathbf{z},J}(\Vert x\Vert) \cdot e^{\lambda_{\mathbf{z},J}t} \,,
\end{equation}
where $C_{\mathbf{z},J}$ is from the spectral gap inequality in \eqref{eq:L2-contraction-SG} and
\begin{equation}\label{g:z:J}
g_{\mathbf{z},J}(\Vert x\Vert)
=e^{|\lambda_{\mathbf{z},J}|\frac{27}{64}}
\left(\frac{16\Gamma(\frac{d}{2}+1)e^{\tilde{C}2^{-3d}\left(1+\beta+(\beta^{1/2}+\beta)\left(\frac{1}{4}\Vert A_{J}\Vert(M\Vert x\Vert+M+B)+\sqrt{\beta^{-1}d}\right)\right)^{2}}}
{\left(\frac{3}{2m\beta} \right)^{-d/2}e^{-\beta b(\log3 )/2}
e^{-\beta(M+B)\left( \lVert x \rVert^2 + \frac{B}{2} + A + \frac{1}{16} \right)}}+1\right) \,,
\end{equation}
where $\tilde{C}>0$ is some universal constant.
It follows that uniformly in $\mathbf{z}$, we have
\begin{equation} 
\int_{\mathbb{R}^d} \left| p_{\mathbf{z},J}(t,x,y) - \pi_{\bz}(y) \right| dy \leq C_{\ast, J}\cdot g_{J}(\Vert x\Vert) \cdot e^{\lambda_{\ast,J}t} \,,
\end{equation}
where $C_{\ast,J}:=\sup_{\mathbf{z}\in\mathcal{Z}^{n}}C_{\mathbf{z}, J}$ and
\begin{equation}
g_{J}(\Vert x\Vert)
=e^{|\lambda_{\ast,J}|\frac{27}{64}}
\left(\frac{16\Gamma(\frac{d}{2}+1)e^{\tilde{C}2^{-3d}\left(1+\beta+(\beta^{1/2}+\beta)\left(\frac{1}{4}\Vert A_{J}\Vert(M\Vert x\Vert+M+B)+\sqrt{\beta^{-1}d}\right)\right)^{2}}}
{\left(\frac{3}{2m\beta} \right)^{-d/2}e^{-\beta b(\log3 )/2}
e^{-\beta(M+B)\left( \lVert x \rVert^2 + \frac{B}{2} + A + \frac{1}{16} \right)}}+1\right) \,.
\end{equation}
\end{lemma}

\begin{lemma}\label{lm:Ktail}
Under Assumptions~\ref{A1}, ~\ref{A2}, ~\ref{A3} and \ref{A5}, taking $\beta>3/m$ and $ k\eta \geq 1 $.
For $x \in \mathbb{R}^d$, we have the following estimate:
\begin{align}
& \int_{\mathbb{R}^d} \int_{\lVert x \rVert > K}\lVert x \rVert^2\left| p_{\mathbf{z},J}(k \eta,w,x)-\pi_{\mathbf{z}}(x) \right| dx \nu_{\mathbf{z},0}(dw) 
\leq  2\mathcal{D}_{c}e^{\lambda_{\mathbf{z},J}k\eta/2}  \,,
\end{align}
where $K$ is defined as 
\begin{equation}\label{const:K}
K := e^{\lvert \lambda_{\mathbf{z},J} \rvert k\eta/4}\,, 
\end{equation}
and $\lambda_{\mathbf{z},J}<0$ is defined in \eqref{spectrum:z}, $\mathcal{D}_c$ is a constant in \eqref{lm:const:Dc}.
\end{lemma}

%%%%%%%%%%%%%%%%%%%%%%%%%%%%
\subsection{Proof of Theorem~\ref{thm:nonreversible2gibbs}}
\begin{proof}
We can compute 
\begin{align} 
& \left|\mathbb{E}_{X \sim \nu_{\mathbf{z},k\eta}} F_{\mathbf{z}}(X) - \mathbb{E}_{X \sim \pi_{\mathbf{z}}} F_{\mathbf{z}}(X) \right|  = \left| \int_{\mathbb{R}^d}F_{\mathbf{z}}(x)\nu_{\mathbf{z},k\eta}(dx) - \int_{\mathbb{R}^d}F_{\mathbf{z}}(x)\pi_{\mathbf{z}}(dx) \right| 
\nonumber \\
& \quad = \left| \int_{\mathbb{R}^d}F_{\mathbf{z}}(x)\left(\int_{\mathbb{R}^d}p_{\mathbf{z},J}(t,w,x)\nu_{\mathbf{z},0}(dw) - \int_{\mathbb{R}^d}\nu_{\mathbf{z},0}(dw)\pi_{\mathbf{z}}(x) \right)dx\right| 
\nonumber \\
& \quad \leq \int_{\mathbb{R}^d}\int_{\mathbb{R}^d}F_{\mathbf{z}}(x)\left| p_{\mathbf{z},J}(t,w,x)-\pi_{\mathbf{z}}(x) \right| dx \nu_{\mathbf{z},0}(dw) \,,
\end{align}
where in the last inequality we have used the Fubini's Theorem. From the result of the quadratic bound \eqref{slm:objfunc-bd} for the function $F_{\mathbf{z}}$ in Lemma~\ref{slm:quadratic-bd}, we obtain
\begin{align}
& \int_{\mathbb{R}^d}\int_{\mathbb{R}^d}F_{\mathbf{z}}(x)\left| p_{\mathbf{z},J}(t,w,x)-\pi_{\mathbf{z}}(x) \right| dx\nu_{\mathbf{z},0}(dw) 
\nonumber \\
&\leq \int_{\mathbb{R}^d}\nu_{\mathbf{z},0}(w)\int_{\mathbb{R}^d}\left( \frac{M+B}{2}\lVert x \rVert^2 + \frac{B}{2} + A \right)\left| p_{\mathbf{z},J}(t,w,x)-\pi_{\mathbf{z}}(x) \right| dxdw
\nonumber \\
& = \frac{M+B}{2}\int_{\mathbb{R}^d}\int_{\mathbb{R}^d}\lVert x \rVert^2\left| p_{\mathbf{z},J}(t,w,x)-\pi_{\mathbf{z}}(x) \right| dx\nu_{\mathbf{z},0}(dw) 
\nonumber \\
& \quad\quad + \left( \frac{B}{2}+A \right)\int_{\mathbb{R}^d}\int_{\mathbb{R}^d}\left| p_{\mathbf{z},J}(t,w,x)-\pi_{\mathbf{z}}(x) \right| dx\nu_{\mathbf{z},0}(dw)\,. \label{pf:eqn:NSGLD2gibbs}
\end{align}
To bound the first term, we use the constant $K>0$ defined in \eqref{const:K} to break the integral into two parts and consider the bounds for each term,
\begin{align}
& \int_{\mathbb{R}^d}  \int_{\mathbb{R}^d} \lVert x \rVert^2 \left| p_{\mathbf{z},J}(t,w,x)-\pi_{\mathbf{z}}(x) \right| dx  \nu_{\mathbf{z},0}(dw) \nonumber \\
 & =  \int_{\mathbb{R}^d}  \int_{\lVert x \rVert \leq K} \lVert x \rVert^2 \left| p_{\mathbf{z},J}(t,w,x)-\pi_{\mathbf{z}}(x) \right| dx  \nu_{\mathbf{z},0}(dw) 
 \nonumber
 \\
 &\qquad\qquad\qquad\qquad\qquad+  \int_{\mathbb{R}^d}  \int_{\lVert x \rVert > K} \lVert x \rVert^2 \left| p_{\mathbf{z},J}(t,w,x)-\pi_{\mathbf{z}}(x) \right| dx  \nu_{\mathbf{z},0}(dw)\,.
\end{align}
By Lemma~\ref{lm:acceleration} and $K = e^{\lvert \lambda_{\mathbf{z},J} \rvert k\eta/4}$, 
we have
\begin{align}
& \int_{\mathbb{R}^d}  \int_{\lVert x \rVert \leq K} \lVert x \rVert^2 \left| p_{\mathbf{z},J}(k\eta,w,x)-\pi_{\mathbf{z}}(x) \right| dx  \nu_{\mathbf{z},0}(dw) \, \nonumber \\
& \leq \int_{\mathbb{R}^d} K^2 C_{\mathbf{z},J}\cdot g_{\mathbf{z},J}(\Vert w\Vert)\cdot e^{\lambda_{\mathbf{z},J} k\eta} \nu_{\mathbf{z},0}(dw) =e^{ \lambda_{\mathbf{z},J}  k\eta/2} C_{\mathbf{z},J}\int_{\lVert w \rVert \leq R}g_{\mathbf{z},J}(\Vert w\Vert)\nu_{\mathbf{z},0}(dw),  \label{pf:eqn:Kbound}
\end{align} 
where $\nu_{\mathbf{z},0}$ is supported on an Euclidean ball with radius $R$ by Assumption~\ref{A5}. 
The definition of $g_{\mathbf{z},J}(\Vert x\Vert)$ implies that it is increasing in $\Vert x\Vert$.  
It follows from \eqref{pf:eqn:Kbound} that
\begin{equation*} 
\int_{\mathbb{R}^d}  \int_{\lVert x \rVert \leq K} \lVert x \rVert^2 \left| p_{\mathbf{z},J}(k\eta,w,x)-\pi_{\mathbf{z}}(x) \right| dx  \nu_{\mathbf{z},0}(dw)
\leq
\hat{C}_{\mathbf{z},J}e^{\lambda_{\mathbf{z},J}k\eta/2},
\end{equation*}
where 
\begin{equation}\label{hat:C:z:J}
    \hat{C}_{\mathbf{z},J}:=C_{\mathbf{z},J}g_{\mathbf{z},J}(R),
\end{equation} 
with $g_{\mathbf{z},J}$ function defined in~\eqref{g:z:J} from  Lemma~\ref{lm:acceleration} and ${C}_{\mathbf{z},J}$ defined in  \eqref{eq:L2-contraction-SG}.
In addition, Lemma~\ref{lm:Ktail} implies:
\begin{align}
& \int_{\mathbb{R}^d} \int_{\lVert x \rVert > K}\lVert x \rVert^2\left| p_{\mathbf{z},J}(t,w,x)-\pi_{\mathbf{z}}(x) \right| dx \nu_{\mathbf{z},0}(dw) \leq
2\mathcal{D}_c  e^{\lambda_{\mathbf{z},J}k\eta/2}  \,. \label{pf:eqn:Ktail}
\end{align}
As a result, the first term in \eqref{pf:eqn:NSGLD2gibbs} is bounded by  \eqref{pf:eqn:Kbound} and \eqref{pf:eqn:Ktail}.

To bound the second term in \eqref{pf:eqn:NSGLD2gibbs}, we apply Lemma~\ref{lm:acceleration} directly with $\nu_{\mathbf{z},0}$ is supported on $\lVert X(0) \rVert \leq R$ by Assumption~\ref{A5} 
\begin{align}
\int_{\mathbb{R}^d}\int_{\mathbb{R}^d}\left| p_{\mathbf{z},J}(t,w,x)-\pi_{\mathbf{z}}(x) \right| dx\nu_{\mathbf{z},0}(dw) 
\leq C_{\mathbf{z},J}\cdot g_{\mathbf{z},J}(R)e^{\lambda_{\mathbf{z},J} k\eta}
=\tilde{C}_{\mathbf{z},J}e^{\lambda_{\mathbf{z},J} k\eta}\,. \label{pf:eqn:gJ-integration}
\end{align}
Hence, we infer from \eqref{pf:eqn:Kbound},~\eqref{pf:eqn:Ktail} and \eqref{pf:eqn:gJ-integration} to get the bound in \eqref{pf:eqn:NSGLD2gibbs} that
\begin{align}
&\left|\mathbb{E}_{X \sim \nu_{\mathbf{z},k\eta}} F_{\mathbf{z}}(X) - \mathbb{E}_{X \sim \pi_{\mathbf{z}}} F_{\mathbf{z}}(X) \right|\nonumber
\\
&\leq  \int_{\mathbb{R}^d}  \int_{\mathbb{R}^d} F_{\mathbf{z}}(x) \left| p_{\mathbf{z},J}(k \eta,w,x)-\pi_{\mathbf{z}}(x) \right| dx  \nu_{\mathbf{z},0}(dw) \, \nonumber \\
&\leq \frac{M+B}{2}\left( \hat{C}_{\mathbf{z},J}+ 2\mathcal{D}_c\right) e^{\lambda_{\mathbf{z},J}k\eta/2}  + \left( \frac{B}{2}+A \right)\cdot \hat{C}_{\mathbf{z},J} \, e^{\lambda_{\mathbf{z},J} k\eta}  \, \nonumber \\
&<\left[ \hat{C}_{\mathbf{z},J}\left( \frac{M+B}{2} +\frac{B}{2}+A\right) + (M+B)\mathcal{D}_c\right]e^{\lambda_{\mathbf{z},J}k\eta/2} \,,
\end{align}
where we used the condition $ k\eta \geq 1$ with $\lambda_{\mathbf{z},J}<0 $, which implies $ e^{\lambda_{\mathbf{z},J}k\eta} < e^{\lambda_{\mathbf{z},J}k\eta/2}\, $, 
that was used to infer the strict inequality above. Therefore, for $\beta \geq 3/m$, we obtain
\begin{align}
& \left|\mathbb{E}_{X \sim \nu_{\mathbf{z},k\eta}} F_{\mathbf{z}}(X) - \mathbb{E}_{X \sim \pi_{\mathbf{z}}} F_{\mathbf{z}}(X) \right| \leq  \mathcal{I}_{0}(\mathbf{z},J,\varepsilon)\,, 
\end{align}
with any given $\varepsilon > 0$, 
\begin{align}
& \mathcal{I}_{0}(\mathbf{z},J,\varepsilon) : = \left[ \hat{C}_{\mathbf{z},J} \left( \frac{M+B}{2} +\frac{B}{2}+A\right) + (M+B)\mathcal{D}_c\right] \cdot \varepsilon \,,
\end{align}
provided that 
\begin{equation*}
k\eta \geq \max \left\{ \frac{2}{|\lambda_{\mathbf{z},J}|}\log\left(\frac{1}{\varepsilon}\right) \,, 1 \right\} \,.
\end{equation*}
The proof is complete.
\end{proof}
%%%%%%%%%%%%%%%%%%%%%%%%%%
\subsection{Proof of Corollary~\ref{cor:NSGLD2gibbs}}
\begin{proof}
Recall the two probability measures $\mu_{\mathbf{z},k}=\mathcal{L}\left( X_k \bigl| \mathcal{Z}=\mathbf{z} \right)$ and $\nu_{\mathbf{z},k\eta}=\mathcal{L}\left( X(t)\bigl| \mathcal{Z}=\mathbf{z} \right).$ Then we can use the triangular inequality and obtain
\begin{align}
& \left|\mathbb{E}_{X \sim \mu_{\mathbf{z},k}} F_{\mathbf{z}}(X) - \mathbb{E}_{X \sim \pi_{\mathbf{z}}} F_{\mathbf{z}}(X) \right| \nonumber \\
& \quad \leq \left|\mathbb{E}_{X \sim \mu_{\mathbf{z},k}} F_{\mathbf{z}}(X) - \mathbb{E}_{X \sim \nu_{\mathbf{z},k\eta}} F_{\mathbf{z}}(X) \right|+\left|\mathbb{E}_{X \sim \nu_{\mathbf{z},k\eta}} F_{\mathbf{z}}(X) - \mathbb{E}_{X \sim \pi_{\mathbf{z}}} F_{\mathbf{z}}(X) \right|\,. \label{pf:cor:decomposition}
\end{align}
First, we consider the first term, inferring from the 2-Wasserstein continuity for functions of quadratic growth in Lemma~\ref{slm:W2-quadratic-bd},
\begin{align}
& \left|\mathbb{E}_{X \sim \mu_{\mathbf{z},k}} F_{\mathbf{z}}(X) - \mathbb{E}_{X \sim \nu_{\mathbf{z},k\eta}} F_{\mathbf{z}}(X) \right| \nonumber \\
& \quad =  \left| \int_{\mathbb{R}^d}F_{\mathbf{z}}(x)\mu_{\mathbf{z},k}(dx) - \int_{\mathbb{R}^d}F_{\mathbf{z}}(x)\nu_{\mathbf{z},k\eta}(dx) \right| \leq \left(M\sigma+B \right)\mathcal{W}_2(\mu_{\mathbf{z},k},\nu_{\mathbf{z},k\eta})\,,
\end{align}
where the constant $\sigma = \sqrt{\mathcal{C}_d}$ with $\mathcal{C}_d$ defined in Lemma~\ref{lm:uniform-L2}. Applying the diffusion approximation in Lemma~\ref{lm:approximation}, we have
 \begin{align}
& \left|\mathbb{E}_{X \sim \mu_{\mathbf{z},k}} F_{\mathbf{z}}(X) - \mathbb{E}_{X \sim \nu_{\mathbf{z},k\eta}} F_{\mathbf{z}}(X) \right| \nonumber \\
& \qquad \leq \left(M\sqrt{\mathcal{C}_d}+B \right)(\hat{C_0}\eta^{1/4}+\hat{C_1}\delta^{1/4})\sqrt{k\eta}\sqrt{\log(k\eta)}\lVert A_J \rVert \,. \label{pf:cor:approximation-bd}
\end{align}
Next, we consider the upper bound term in the above inequality where it requires
\begin{equation}
k\eta = \frac{2}{|\lambda_{\mathbf{z},J}|}\log\left( 1/\varepsilon\right) \geq e \,, 
\end{equation}
By choosing
$$
\eta \leq \min\left\{1, \frac{m^2}{(m^2+ 8M^2)M\lVert A_J\rVert^2},\frac{\varepsilon^4}{4(\log(1/\varepsilon))^2\lVert A_J \rVert^4} \frac{\lvert \lambda_{\mathbf{z},J} \rvert ^2}{\lvert \lambda_{\mathbf{z},J=0}\rvert ^2} \right\} \,.
$$
Then we get
\begin{align*}
& \left|\mathbb{E}_{X \sim \mu_{\mathbf{z}}} F_{\mathbf{z}}(X) - \mathbb{E}_{X \sim \nu_{\mathbf{z}}} F_{\mathbf{z}}(X) \right|   \nonumber \\
& \leq \mathcal{I}_1(\mathbf{z},J,\varepsilon)
\\
&= \left(M\sqrt{\mathcal{C}_d}+B \right)\left( \hat{C_0} \frac{\varepsilon}{\sqrt{\lvert \lambda_{\mathbf{z},J=0}\rvert}}+ \hat{C_1}\delta^{1/4}\sqrt{\frac{2\log(1/\varepsilon)}{\lvert \lambda_{\mathbf{z},J} \rvert}} \lVert A_J \rVert \right) \sqrt{\log\left( \frac{2\log(1/\varepsilon)}{\lvert \lambda_{\mathbf{z},J} \rvert} \right)} \,. 
\end{align*}
Next, the upper bound for the second term in \eqref{pf:cor:decomposition} can be found in Theorem~\ref{thm:nonreversible2gibbs} as the following
$$
\left|\mathbb{E}_{X \sim \nu_{\mathbf{z},k\eta}} F_{\mathbf{z}}(X) - \mathbb{E}_{X \sim \pi_{\mathbf{z}}} F_{\mathbf{z}}(X) \right| \leq \mathcal{I}_{0}(\mathbf{z},J,\varepsilon) \,,
$$
Therefore, %by choosing
%$$
%k\eta =  \frac{2}{|\lambda_{\mathbf{z},J}|}\log\left(\frac{1}{\varepsilon}\right)\geq 1 \,,
%$$
inferring from \eqref{pf:cor:decomposition},
\begin{equation*}
\left|\mathbb{E}_{X \sim \mu_{\mathbf{z},k}} F_{\mathbf{z}}(X) - \mathbb{E}_{X \sim \pi_{\mathbf{z}}} F_{\mathbf{z}}(X) \right| \leq \mathcal{I}_{0}(\mathbf{z},J,\varepsilon) + \mathcal{I}_{1}(\mathbf{z},J,\varepsilon)\,.
\end{equation*}
The proof is complete.
\end{proof}

%%%%%%%%%%%%%%%%%%%%%%%%%%%%%%%%%%%%%%%%%%%%%%%%%%%%%%%%%%
\section{Proofs of technical lemmas in Appendix~\ref{sec:tech-lemmas}}

\subsection{Proof of Lemma~\ref{lm:uniform-L2}}
\begin{proof}
The uniform $L^2$ bound on non-reversible Langevin SDE follows \cite{GGZ18-2}, and we will prove uniform $L^2$ bound on NSGLD algorithm \eqref{NSGLD}. Recall the dynamics for NSGLD algorithm follows:
\begin{align}
X_{k+1} = X_k - \eta \, A_J \, g(X_k,U_{\mathbf{z},k}) + \sqrt{2\eta\beta^{-1}}\xi_k \,,
\end{align}
with stochastic gradient $g(x,U_{\mathbf{z},k})$ which is a conditionally unbiased estimator for $F_{\mathbf{z}}(x)$, 
$$
\mathbb{E}\left[g(x,U_{\mathbf{z},k})\right] = \nabla F_{\mathbf{z}}(x)\,, \qquad x \in \mathbb{R}^d \,.
$$
We have a quadratic bound in  Lemma~\ref{slm:quadratic-bd}:
$$
\lVert x \rVert^2 \leq \frac{3}{m} F_{\mathbf{z}}(x) + \frac{3b}{2m}\log 3 \,.
$$
Our aim is to find an uniform bound for $\mathbb{E}F_{\mathbf{z}}(X_{k})$. Inferring the proof in \cite[Lemma 30]{GGZ18-1}, suppose we can establish
\begin{equation}
\frac{\mathbb{E}F_{\mathbf{z}}(X_{k+1})-\mathbb{E}F_{\mathbf{z}}(X_{k})}{\eta} \leq -\varepsilon \mathbb{E}F_{\mathbf{z}}(X_k)+b\varepsilon \,,
\end{equation}
uniformly for small $\eta$ and $\varepsilon$, $b$ are positive constants which are independent of $\eta$, then we have
\begin{equation} \label{expectation:bd}
\mathbb{E}F_{\mathbf{z}}(X_{k+1}) \leq \mathbb{E}F_{\mathbf{z}}(X_0)+ b\,.
\end{equation}
Note that $\nabla F$ is Lipschitz continuous with Lipschitz constant $M$. We have
$$
F(y) \leq F(x) + \nabla F(x)(y-x) +\frac{M}{2}\lVert y-x \rVert^2 \,.
$$
We can compute that
\begin{align}
& \frac{\mathbb{E}F_{\mathbf{z}}(X_{k+1}) - \mathbb{E}F_{\mathbf{z}}(X_k)}{\eta} 
\\
\nonumber
&= \frac{1}{\eta}\left( \mathbb{E}\left[ F_{\mathbf{z}}\left(X_k-\eta A_J g(X_k,U_{\mathbf{z},k}) + \sqrt{2\eta\beta^{-1}\xi_k} \right) \right] - \mathbb{E}F_{\mathbf{z}}(X_k) \right) \, \nonumber \\
& \leq - \mathbb{E}\left[\nabla F_{\mathbf{z}}(X_k) A_J \nabla F_{\mathbf{z}}(X_k)\right] 
\nonumber
\\
&\qquad
+ \frac{M}{2\eta} \mathbb{E}\biggl\lVert -\eta A_J (g(X_k,U_{\mathbf{z},k}) - \nabla F_{\mathbf{z}}(X_k)) - \eta A_J \nabla F_{\mathbf{z}}(X_k)  + \sqrt{2\eta\beta^{-1}}\xi_k\biggr\rVert^2 \, \nonumber \\
& \leq -\mathbb{E}\lVert \nabla F_{\mathbf{z}}(X_k) \rVert^2 
\nonumber
\\
&\qquad
+ \frac{M}{2\eta}\bigg(\eta^2\lVert A_J \rVert^2
\mathbb{E}\lVert g(X_k,U_{\mathbf{z},k}) -  \nabla F_{\mathbf{z}}(X_k) \rVert^2 
\nonumber
\\
&\qquad\qquad\qquad\qquad
+ \eta^2 \lVert A_J \rVert^2 
\mathbb{E}\lVert \nabla F_{\mathbf{z}}(X_k) \rVert^2 + 2\eta\beta^{-1} \mathbb{E}\lVert \xi_k \rVert^2 \bigg) \,,
\end{align}
where the first inequality is using $\mathbb{E}[g(x,U_{\mathbf{z},k})] = \nabla F_{\mathbf{z}}(x)$ for $x \in \mathbb{R}^d$, 
and the last equality is due to the fact that the inner product of independent random vectors is $0$. Then using Assumption~\ref{A4}, we have
\begin{align}
& \frac{\mathbb{E}F_{\mathbf{z}}(X_{k+1}) - \mathbb{E}F_{\mathbf{z}}(X_k)}{\eta} \, \nonumber \\
&\leq -\mathbb{E}\lVert \nabla F_{\mathbf{z}}(X_k) \rVert^2  + M\eta \delta \lVert  A_J\rVert^2 \left(M^2\mathbb{E}\lVert X_k \rVert^2 + B^2 \right) 
\nonumber
\\
&\qquad\qquad\qquad
+ \frac{M}{2}\eta\lVert A_J \rVert^2 \mathbb{E}\lVert \nabla F_{\mathbf{z}}(X_k) \rVert^2 + M\beta^{-1}d \, \nonumber \\
& = -\left(1 - \frac{M}{2}\eta \lVert A_J \rVert^2 \right) \mathbb{E}\lVert \nabla F_{\mathbf{z}}(X_k) \rVert^2 + M^3 \eta \delta \lVert A_J \rVert^2 \mathbb{E}\lVert X_k \rVert^2 + M\left(\eta \delta \lVert  A_J\rVert^2 B^2 + \beta^{-1}d \right) \,. \label{pf:lm:uniform-L2:discrete}
\end{align}
If $\lVert x \rVert \geq \sqrt{2b/m}$, $(m,b)$-dissipative in Assumption~\ref{A3} implies,
\begin{align}
\lVert \nabla F(x) \rVert \geq m\lVert x \rVert - \frac{b}{\lVert x \rVert} \geq \frac{m}{2}\lVert x \rVert \,. \label{pf:lm:dissipative}
\end{align}
Then \eqref{pf:lm:uniform-L2:discrete} implies
\begin{align*}
\frac{\mathbb{E}F_{\mathbf{z}}(X_{k+1}) - \mathbb{E}F_{\mathbf{z}}(X_{k})}{\eta} 
&\leq  -\left(1 - \frac{M\lVert A_J \rVert^2 }{2}\eta - \frac{4M^3\delta\lVert A_J \rVert^2}{m^2}\eta \right) \mathbb{E} \lVert \nabla F_{\mathbf{z}}(X_k) \rVert^2 
\\
&\qquad\qquad\qquad
+  M\left(\eta \delta \lVert  A_J\rVert^2 B^2 + \beta^{-1}d \right) \,.
\end{align*}
According to Assumption~\ref{A4}, $0\leq \delta<1$,
and under the assumption for the stepsize $\eta$, we have
$$
\eta\leq\frac{m^2}{Mm^2\lVert A_J \rVert^2 + 8M^3 \lVert A_J \rVert^2}\leq \frac{m^2}{Mm^2\lVert A_J \rVert^2 + 8M^3\delta \lVert A_J \rVert^2} \,,
$$
it implies 
$$
1 - \frac{M\lVert A_J \rVert^2}{2}\eta  - \frac{4M^3\delta\lVert A_J \rVert^2}{m^2}\eta \geq \frac{1}{2} \,.
$$
By applying \eqref{pf:lm:dissipative} again, we can compute that
$$
\frac{\mathbb{E}F_{\mathbf{z}}(X_{k+1}) - \mathbb{E}F_{\mathbf{z}}(X_{k})}{\eta} \leq  -\frac{m^2}{8} \mathbb{E} \lVert X_k \rVert^2 +  M\left(\eta \delta \lVert  A_J\rVert^2 B^2 + \beta^{-1}d \right) \,.
$$
If $\lVert x \rVert < \sqrt{2b/m}$, under the assumption
for the stepsize $\eta$, we have
$$
\eta \leq \frac{m^2}{Mm^2\lVert A_J \rVert^2 + 8M^3 \lVert A_J \rVert^2} < \frac{2}{M\lVert A_J \rVert^2} \,,
$$
so that 
$$
1-\frac{M\lVert A_J \rVert^2}{2} \eta > 0 \,.
$$
Hence, \eqref{pf:lm:uniform-L2:discrete} implies,
$$
\frac{\mathbb{E}F_{\mathbf{z}}(X_{k+1}) - \mathbb{E}F_{\mathbf{z}}(X_{k})}{\eta} \leq M^3\eta \delta \lVert A_J \rVert^2(2b/m) + M\left(\eta \delta \lVert  A_J\rVert^2 B^2 + \beta^{-1}d \right) \,.
$$
Overall, for all $x \in \mathbb{R}^d$, we have
\begin{align}
\frac{\mathbb{E}F_{\mathbf{z}}(X_{k+1}) - \mathbb{E}F_{\mathbf{z}}(X_k)}{\eta} 
&\leq -\frac{m^2}{8}\mathbb{E} \lVert X_k \rVert^2 + \frac{2bM^3}{m}\eta \delta \lVert A_J \rVert^2+ M\left(\eta \delta \lVert  A_J\rVert^2 B^2 + \beta^{-1}d \right) + \frac{mb}{4} \, \nonumber \\
&< -\frac{m^2}{8}\mathbb{E} \lVert X_k \rVert^2 + \delta \left(\frac{bM^2}{m} + \frac{B^2}{2} \right) + M\beta^{-1}d + \frac{mb}{4} \,, 
\end{align}
where we used $\eta< 1/(2M\lVert A_J \rVert^2)$ to get the strict inequality. 
We recall the quadratic bound for objective function:
$$
F_{\mathbf{z}}(x) \leq \frac{M+B}{2}\lVert x \rVert^2 + \frac{B}{2}+A \,.
$$
Hence, we get
$$
\frac{\mathbb{E}F_{\mathbf{z}}(X_{k+1}) - \mathbb{E}F_{\mathbf{z}}(X_k)}{\eta} \leq -\frac{m^2}{4(M+B)}F_{\mathbf{z}}(X_k) + \frac{m^2(B+2A)}{8(M+B)} + \delta \left(\frac{bM^2}{m} + \frac{B^2}{2} \right) + M\beta^{-1}d + \frac{mb}{4}\,.
$$
Therefore, for $ \eta \leq \frac{m^2}{(m^2+8M^2)M\lVert A_J \rVert^2} $, we have the following inequality by~\eqref{expectation:bd},
$$
\mathbb{E}F_{\mathbf{z}}(X_{k+1}) \leq \mathbb{E}F_{\mathbf{z}}(X_0) + \frac{2\delta(2bM^2/m + B^2)(M+B)}{m^2} + \frac{4M\beta^{-1}d(M+B)}{m^2} + \frac{b(M+B)}{m} + \frac{B}{2} + A \,.
$$
With $\lVert X(0) \rVert \leq R \, $, $F_{\mathbf{z}}(X_0)$ is bounded by
$$
F_{\mathbf{z}}(X_0) \leq \frac{M}{2}\lVert X_0 \rVert^2 + B\lVert X_0 \rVert + A \leq \frac{M}{2}R^2 + BR + A \,.
$$
In addition, we recall that 
$$
\lVert x \rVert^2 \leq \frac{3}{m} F_{\mathbf{z}}(x) + \frac{3b}{2m}\log 3 \,.
$$
Therefore, we obtain the uniform $L^2$ bound for $\lVert X_k \rVert$:
\begin{align}
\sup_{k\geq 0 } \mathbb{E}_{\mathbf{z}}\lVert X_k \rVert^2 & \leq \frac{3}{m}\mathbb{E}F_{\mathbf{z}}(X_{k}) + \frac{3b}{2m} \log 3 \, \nonumber \\
& \leq \frac{3MR^2 + 6BR + 3B + 6A + 3b\log3}{2m} +\frac{6\delta(2bM^2 + B^2m)(M+B)}{m^4} \, \nonumber \\
& \qquad\qquad + \frac{12M\beta^{-1}d(M+B)}{m^3} + \frac{3b(M+B)}{m^2} \,.
\label{uniform:L2:bound}
\end{align}
The proof is complete.
\end{proof}

%%%%%%%%%%%%%%%%%%%%%%%%%%%%%%%%%%%%%%%%%%%%%%%%%%%%%%%%%%
\subsection{Proof of Lemma~\ref{lm:exponential}}
\begin{proof}
First, notice that the quadratic bound in Lemma~\ref{slm:quadratic-bd} gives 
that uniformly in $\mathbf{z}\in\mathcal{Z}^{n}$,
\begin{equation*}
\frac{3}{m}F_{\mathbf{z}}(x) \geq  \lVert x \rVert^2 - \frac{3b}{2m}\log 3 \,.
\end{equation*}
Thus it suffices for us to get a uniform bound for $\mathbb{E} \left[ e^{\frac{3}{m}F_{\mathbf{z}}(X(t))} \right]$. 
We recall that the non-reversible Langevin SDE is given by
\begin{equation} 
dX(t)=-A_J\left( \nabla F_{\mathbf{z}}(X(t)) \right)dt + \sqrt{2\beta^{-1}}dB(t)\,, 
\end{equation}
whose infinitesimal generator for this system is
\begin{equation*}
\mathcal{L}_J=-A_J \nabla F_{\mathbf{z}} \cdot \nabla + \beta^{-1} \Delta \,. 
\end{equation*}
For any $x\in\mathbb{R}^d  $, we can compute that 
\begin{align}
\mathcal{L}_J F_{\mathbf{z}}(x)& =-A_J \nabla F_{\mathbf{z}}(x) \cdot \nabla F_{\mathbf{z}}(x)+\beta^{-1} \Delta F_{\mathbf{z}}(x) = - \lVert \nabla F_{\mathbf{z}}(x) \rVert^2+\beta^{-1} \Delta F_{\mathbf{z}}(x) \,. 
\end{align}
where $J$ is an anti-symmetric matrix and $\langle J \nabla F, \nabla F \rangle = 0$. For any constant $\alpha>0$, we get
\begin{align}
\mathcal{L}_J\left(e^{\alpha F_{\mathbf{z}}}\right) &= \left( \alpha \mathcal{L}_J F_{\mathbf{z}} + \alpha^2\beta^{-1} \lVert \nabla F \rVert^2\right)e^{\alpha F_{\mathbf{z}}} 
\nonumber \\
&= \left( \left( \alpha - \alpha^2 \beta^{-1} \right)\mathcal{L}_J F_{\mathbf{z}} + \alpha^2 \beta^{-2}\Delta F_{\mathbf{z}} \right)e^{\alpha F_{\mathbf{z}}}  \nonumber \,, \\
&= \left(-\left( \alpha - \alpha^2 \beta^{-1} \right) \lVert \nabla F_{\mathbf{z}} \rVert^2 + \alpha\beta^{-1} \Delta F_{\mathbf{z}} \right) e^{ \alpha F_{\mathbf{z}}} \,  \label{pf:lm:generator} \,.
\end{align}
where we used the property of the anti-symmetric matrix $J$, such that $\langle \nabla F_{\mathbf{z}},J \,\nabla F_{\mathbf{z}} \rangle = 0  $. Recall $F_{\mathbf{z}}$ is $M$-smooth, then $\Delta F \leq Md$. In addition, by assuming $\beta>\alpha >0$, then $\alpha - \alpha^2\beta^{-1}>0$, and the relation \eqref{pf:lm:generator} implies
\begin{equation*}
\mathcal{L}_J\left(e^{\alpha F_{\mathbf{z}}}\right) \leq \left( \alpha \beta^{-1} \Delta F_{\mathbf{z}}  \right) e^{ \alpha F_{\mathbf{z}}} \leq \left( \alpha \beta^{-1} Md  \right) e^{ \alpha F_{\mathbf{z}}} \,.
\end{equation*}
Recall that the initial condition satisfies $\lVert X(0) \rVert \leq R $, and the quadratic bound Lemma~\ref{slm:quadratic-bd} for function $F_{\mathbf{z}}$:
\begin{equation*}
F_{\mathbf{z}}(x) \leq \frac{M}{2} \lVert x \rVert^2 + B\lVert x \rVert + A \,. 
\end{equation*}
Then it follows from Corollary 2.4 \cite{CHJ13} that we have the exponential integrability:
\begin{equation}
\mathbb{E} \left[ e^{ \alpha F_{\mathbf{z}}(X(t))} \right] \leq \mathbb{E} \left[ e^{ \alpha F_{\mathbf{z}}(X(0))} \right] e^{\alpha\beta^{-1}Md \, t} \leq e^{\alpha\left(MR^2/2+BR+A+\beta^{-1}Md \, t \right)} <  \infty \label{pf:lm:exponential}\,.
\end{equation}
Next, by applying It\^{o}'s formula to $e^{(3/m)F_{\mathbf{z}}}$, we get:
\begin{equation}
e^{\frac{3}{m}F_{\mathbf{z}}(X(t))} = e^{\frac{3}{m}F_{\mathbf{z}}(X(0))}  + \int_0^{t} \mathcal{L}_J \left( e^{\frac{3}{m}F_{\mathbf{z}}(X(s))} \right) ds + \int_0^{t} \frac{3\sqrt{2\beta^{-1}}}{m} \nabla F_{\mathbf{z}}(X(s)) e^{\frac{3}{m}F_{\mathbf{z}}(X(s))} dB(s) \,.\label{pf:lm:martingale}
\end{equation}
We can check that the square integrability condition holds for the diffusion term in \eqref{pf:lm:martingale}.
That is, for any $t>0$, 
\begin{align}
& \int_0^{t} \frac{18\beta^{-1}}{m} \mathbb{E} \left[  \left\Vert \nabla F_{\mathbf{z}}(X(s))  e^{\frac{3}{m}F_{\mathbf{z}}(X(s))} \right\Vert^2 \right] ds \nonumber \\
& \quad \leq \int_0^{t} \frac{18\beta^{-1}}{m} \mathbb{E} \left[ \left( 2M^2\lVert X(s) \rVert^2 + 2B^2 \right) e^{\frac{6}{m}F_{\mathbf{z}}(X(s))} \right] ds \nonumber \\
& \quad \leq \int_0^{t} \frac{18\beta^{-1}}{m} \mathbb{E} \left[ \left( 2M^2\left(\frac{3}{m}\left(F_{\mathbf{z}}(X(s)) + \frac{b}{2}\log 3 \right)\right) + 2B^2 \right) e^{\frac{6}{m}F_{\mathbf{z}}(X(s))} \right] ds \nonumber \\
& \quad \leq \int_0^{t} \frac{18\beta^{-1}}{m}\mathbb{E}\left[M^2 e^{\frac{12}{m}F_{\mathbf{z}}(X(s))}+\left(\frac{3bM^2\log3}{m} + 2B^2\right)e^{\frac{6}{m}F_{\mathbf{z}}(X(s))} \right] ds < \infty \,. \label{pf:lm:sqrtintegrability}
\end{align}
The first and second inequalities above are due to the quadratic bounds in Lemma~\ref{slm:quadratic-bd}, \eqref{slm:gradient-bd} and \eqref{slm:objfunc-bd} and the third inequality above is due to the fact that
$x \leq e^x, x>0$ and the exponential integrability property in \eqref{pf:lm:exponential} with $\alpha = 6/m $. As a result, we know $\int_0^t \frac{3\sqrt{2\beta^{-1}}}{m} \nabla F_{\mathbf{z}}(X(s)) e^{\frac{3}{m}F_{\mathbf{z}}(X(s))} dB(s)$ is a martingale, then we can take the expectation 
in \eqref{pf:lm:martingale} and obtain:
\begin{equation*}
\mathbb{E} \left[ e^{\frac{3}{m}F_{\mathbf{z}}(X(t))} \right] = \mathbb{E} \left[ e^{\frac{3}{m}F_{\mathbf{z}}(X(0))}  \right] + \int_0^t \mathbb{E} \left[ \mathcal{L}_J \left( e^{\frac{3}{m}F_{\mathbf{z}}(X(s))} \right) \right] ds \,.
\end{equation*}
Next, we compute an upper bound for $\mathbb{E} \left[ \mathcal{L}_J \left( e^{\frac{3}{m}F_{\mathbf{z}}(X(\cdot))} \right) \right] $ in the above equation. Lemma 28 in \cite{GGZ18-2} gives the Lyapunov condition for $F_{\mathbf{z}}(x)$ such that
\begin{equation}
\mathcal{L}_J F_{\mathbf{z}}(x) \leq -\frac{m^2}{2(M+B)} F_{\mathbf{z}}(x) + \frac{m^2(B/2+A)}{2(M+B)} + \frac{mb}{2} + \beta^{-1}Md \,.
\end{equation} 
Then by applying \eqref{pf:lm:generator} and $\beta > 3/m$, we have 
\begin{align}
\mathcal{L}_J \left( e^{\frac{3}{m}F} \right) \leq & \Bigg.\Bigg[-\left( \frac{3m-9\beta^{-1}}{2(M+B)}\right)F \nonumber \\
 & \quad  + \left( \frac{(3m-9\beta^{-1})(B/2+A)+3b(M+B)}{2(M+B)}  + \frac{6\beta^{-1}Md-9b\beta^{-1}}{2m} \right) \Bigg.\Bigg] e^{\frac{3}{m}F} \,. \label{pf:lm:generator-ineqn}
\end{align}
If 
$$
-\left( \frac{3m-9\beta^{-1}}{2(M+B)}\right)F_{\mathbf{z}} + \left( \frac{(3m-9\beta^{-1})(B/2+A)+3b(M+B)}{2(M+B)}  + \frac{6\beta^{-1}Md-9b\beta^{-1}}{2m} \right) < 0 \,,
$$
then we have $ \mathcal{L}_J \left( e^{\frac{3}{m}F_{\mathbf{z}}} \right) < 0$. 
Otherwise, 
$$
-\left( \frac{3m-9\beta^{-1}}{2(M+B)}\right)F_{\mathbf{z}} + \left( \frac{(3m-9\beta^{-1})(B/2+A)+3b(M+B)}{2(M+B)}  + \frac{6\beta^{-1}Md-9b\beta^{-1}}{2m} \right) \geq 0 \,,
$$
which implies
\begin{equation}
F_{\mathbf{z}} \leq \frac{B}{2}+A + \frac{M+B}{m-3\beta^{-1}}\left( b+\frac{\beta^{-1}(2Md-3b)}{m} \right) \,.
\end{equation}
Since the objective function $F_{\mathbf{z}}$ is non-negative, it follows from \eqref{pf:lm:generator-ineqn} that
\begin{equation}
\mathcal{L}_J \left( e^{\frac{3}{m}F_{\mathbf{z}}} \right) \leq \left(\frac{(3m-9\beta^{-1})(B/2+A)+3b(M+B)}{2(M+B)} + \frac{6\beta^{-1}Md-9b\beta^{-1}}{2m} \right) e^{\frac{3}{m}F_{\mathbf{z}}} \,.
\end{equation}
Using the upper bound of $F_{\mathbf{z}}$ in the previous calculation, we get, for $\beta > 3/m$, 
\begin{equation}
\mathcal{L}_J \left( e^{\frac{3}{m}F_{\mathbf{z}}} \right) \leq l_1e^{l_2} \,,
\end{equation}
with 
\begin{equation}
l_1 := \frac{(3m-9\beta^{-1})(B/2+A)+3b(M+B)}{2(M+B)}  + \frac{6\beta^{-1}Md-9b\beta^{-1}}{2m} \,, 
\end{equation}
and
\begin{equation}
l_2:= \frac{3}{m} \left(\frac{B}{2}+A + \frac{M+B}{m-3\beta^{-1}}\left( b+\frac{\beta^{-1}(2Md-3b)}{m} \right)\right) \,. 
\end{equation}
Therefore, it follows from \eqref{pf:lm:martingale} that
\begin{align}
\mathbb{E} \left[ e^{\frac{3}{m}F_{\mathbf{z}}(X(t))} \right] &= \mathbb{E} \left[ e^{\frac{3}{m}F_{\mathbf{z}}(X(0))}  \right] + \int_0^{t} \mathbb{E} \left[ \mathcal{L}_J \left( e^{\frac{3}{m}F_{\mathbf{z}}(X(s))} \right) \right] ds \nonumber \\
& \quad \leq \mathbb{E} \left[ e^{\frac{3}{m}F_{\mathbf{z}}(X(0))}  \right] + \int_0^{t} l_1 e^{l_2}ds = \mathbb{E} \left[ e^{\frac{3}{m}F_{\mathbf{z}}(X(0))}  \right] + l_1 e^{l_2}\cdot t \,.
\end{align}
With the initial condition satisfying $\lVert X(0) \rVert \leq R $, we can bound $F_{\mathbf{z}}(X(0))$ by the quadratic bound \eqref{slm:objfunc-bd} in Lemma~\ref{slm:quadratic-bd}, that is,
\begin{equation*}
F_{\mathbf{z}}(X(0)) \leq \frac{M}{2}\lVert X(0) \rVert^2+B\lVert X(0) \rVert + A \leq \frac{MR^2}{2} + BR+A \,.
\end{equation*}
As a result, for any $t>0$, with $\beta>3/m$ and $\lVert X(0) \rVert \leq R $, we get 
\begin{equation*}
\mathbb{E} \left[ e^{\frac{3}{m}F_{\mathbf{z}}(X(t))} \right] \leq e^{\frac{MR^2}{2}+BR+A} + l_1 e^{l_2}(t) \,.
\end{equation*}
Moreover, the quadratic bound in \eqref{slm:gradient-bd} from Lemma~\ref{slm:quadratic-bd} gives 
\begin{equation*}
\frac{3}{m}F_{\mathbf{z}}(x) \geq  \lVert x \rVert^2 - \frac{3b}{2m}\log 3 \,,
\end{equation*}
which implies
\begin{equation*}
\mathbb{E} \left[ e^{\lVert X(t) \rVert ^2} \right] \leq e^{\frac{MR^2}{2}+BR+A + \frac{3b}{2m}\log 3} + l_1 e^{l_2}(t) < \infty \,.
\end{equation*}
The proof is complete.
\end{proof}

%%%%%%%%%%%%%%%%%%%%%%%%%
\subsection{Proof of Lemma \ref{lm:approximation}}
\begin{proof}
The proof closely follows from the proof of Lemma 7 in Raginsky et al. \cite{RRT17}.
Defining the continuous-time interpolation for $X_k$,
\begin{align}
\overline{X}(t) = X_0 -\int_0^{t}  A_J g\left(\overline{X}(\lfloor{s/\eta}\rfloor \eta ),\overline{U}_{\mathbf{z}}(s)\right)ds + \sqrt{2\beta^{-1}}\int_0^t dB(s) \,, 
\end{align}
where $\overline{U}_{\mathbf{z}}(t):=U_{\mathbf{z},k}$ for all $t \in [k\eta,(k+1)\eta)  $. Here $\overline{X}(k\eta)$ and $X_k$ follow the same probability law $\mu_{\mathbf{z},k}$. The result from Gy\"ongy \cite{Gyo86} implies $\overline{X}$ has the same marginals as $\widetilde{X}$ which is a Markov process:
\begin{align}
\widetilde{X}(t)=X_0 - \int_0^t  A_J g_{\mathbf{z},s}\left(X(s)\right) ds + \sqrt{2\beta^{-1}}\int_0^t dB(s) \,, 
\end{align}
with 
\begin{align}
g_{\mathbf{z,s}}(x) := \mathbb{E}_{\mathbf{z}}\left[  g\left(\overline{X}(\lfloor{s/\eta}\rfloor \eta ),\overline{U}_{\mathbf{z}}(s)\right) | \overline{X}(t) = x  \right] \,. 
\end{align}
Suppose the non-reversible Langevin diffusion $X(t)$ follows the probability measure $\mathbb{P}$, 
and $\widetilde{X}(t)$ follows the probability measure $\widetilde{\mathbb{P}}$. 
The Radon-Nikodym derivative is represented by the Girsanov formula 
under the filtration $\mathcal{F}_t  $ for $t>0  $
\begin{align}
\frac{d\mathbb{P}}{d\widetilde{\mathbb{P}}}\bigg\rvert_{\mathcal{F}_t} = e^{\sqrt{\frac{\beta}{2}}\int_0^t A_J \left(\nabla F_{\mathbf{z}}(\widetilde{X}(s)) - g_{\mathbf{z},s}(\widetilde{X}(s))\right) \cdot dB(s) - \frac{\beta}{4} \int_0^t \lVert A_J \left(\nabla F_{\mathbf{z}}(\widetilde{X}(s)) - g_{\mathbf{z},s}(\widetilde{X}(s))\right) \rVert^2 ds} \,. 
\end{align} 
Since the probability law of $\overline{X}(t)$ and $\widetilde{X}(t)$ are the same for each $t>0$, we can use the martingale property of Ito integral and compute the relative entropy as follows:
\begin{align}
D\left( \widetilde{\mathbb{P}}_t \Vert \mathbb{P}_t  \right) & : = - \int d\widetilde{\mathbb{P}}_t \log\frac{d\mathbb{P}_t}{d\widetilde{\mathbb{P}}_t} \nonumber \\
& = \frac{\beta}{4}\lVert A_J \rVert^2\int_0^t \mathbb{E}_{\mathbf{z}} \lVert  \nabla F_{\mathbf{z}}(\widetilde{X}(s)) - g_{\mathbf{z},s} (\widetilde{X}(s)) \rVert^2 ds \nonumber \\
& = \frac{\beta}{4}\lVert A_J \rVert^2\int_0^t \mathbb{E}_{\mathbf{z}}\lVert  \nabla F_{\mathbf{z}}(\overline{X}(s)) - g_{\mathbf{z},s}( \overline{X}(s))\rVert^2 ds\,.  
\end{align}
It follows that
\begin{align}
&D\left( \widetilde{\mathbb{P}}_{k\eta} \Vert \mathbb{P}_{k \eta} \right) 
\nonumber
\\
& = \frac{\beta}{4}\lVert A_J \rVert^2 \sum_{j=0}^{k-1} \int_{j\eta}^{(j+1)\eta} \mathbb{E}_{\mathbf{z}} \Big\lVert \nabla F_{\mathbf{z}}(\overline{X}(s)) - g_{\mathbf{z},s}(\overline{X}(s)) \Big\rVert^2 ds \, \nonumber \\
&\leq \frac{\beta}{4}\lVert A_J \rVert^2 \sum_{j=0}^{k-1} \int_{j\eta}^{(j+1)\eta} \mathbb{E}_{\mathbf{z}} \Big\lVert \nabla F_{\mathbf{z}}(\overline{X}(s)) -  g\left(\overline{X}(\lfloor{s/\eta}\rfloor \eta ),\overline{U}_{\mathbf{z}}(s)\right)  \Big\rVert^2 ds \, \nonumber \\
&\leq \frac{\beta}{2}\lVert A_J \rVert^2 \sum_{j=0}^{k-1} \int_{j\eta}^{(j+1)\eta} \mathbb{E}_{\mathbf{z}} \Big\lVert \nabla F_{\mathbf{z}}(\overline{X}(s)) - \nabla F_{\mathbf{z}}(\overline{X}(\lfloor{s/\eta}\rfloor \eta )) \Big\rVert^2 ds \, \nonumber \\
&\qquad\qquad + \frac{\beta}{2} \lVert A_J \rVert^2 \sum_{j=0}^{k-1} \int_{j\eta}^{(j+1)\eta} \mathbb{E}_{\mathbf{z}} \Big\lVert \nabla F_{\mathbf{z}}(\overline{X}(\lfloor{s/\eta}\rfloor \eta ))-g\left(\overline{X}(\lfloor{s/\eta}\rfloor \eta ),\overline{U}_{\mathbf{z}}(s)\right)  \Big\rVert^2 ds\,. \label{pf:lm:entropy}
\end{align}
With Assumption~\ref{A2}, we can infer that the first term in \eqref{pf:lm:entropy} is bounded as follows,
\begin{align}
& \frac{\beta}{2}\lVert A_J \rVert^2 \sum_{j=0}^{k-1} \int_{j\eta}^{(j+1)\eta} \mathbb{E}_{\mathbf{z}} \Big\lVert \nabla F_{\mathbf{z}}(\overline{X}(s)) - \nabla F_{\mathbf{z}}(\overline{X}(\lfloor{s/\eta}\rfloor \eta )) \Big\rVert^2 ds \nonumber \\
& \qquad\qquad \leq \frac{\beta}{2} M^2\lVert A_J \rVert^2 \sum_{j=0}^{k-1} \int_{j\eta}^{(j+1)\eta} \mathbb{E}_{\mathbf{z}} \Big\lVert \overline{X}(s) - \overline{X}(\lfloor{s/\eta}\rfloor \eta ) \Big\rVert^2 ds\,.
\end{align}
In addition, for some $s>0  $, $j\eta \leq s < (j+1)\eta  $, 
$$
\overline{X}(s)-\overline{X}(j\eta) = -(s-j\eta)g(X_j,U_{\mathbf{z},j})+ \sqrt{2\beta^{-1}}\left( B_s - B_{j\eta} \right)\,.
$$
Then, we can get
\begin{align}
& \mathbb{E}_{\mathbf{z}} \Big\lVert \overline{X}(s) - \overline{X}(\lfloor{s/\eta}\rfloor \eta ) \Big\rVert^2 = (s-j\eta)^2\mathbb{E}_{\mathbf{z}} \Big\lVert g(X_j,U_{\mathbf{z},j}) \Big\rVert^2 + 2\beta^{-1}(s-j\eta)^2d \, \nonumber \\
& \qquad\leq (s-j\eta)^2\mathbb{E}_{\mathbf{z}} \lVert g(X_j,U_{\mathbf{z},j}) - \nabla F_{\mathbf{z}}(X_j) \rVert^2 \, \nonumber \\
& \qquad\qquad\qquad + (s-j\eta)^2 \mathbb{E}_{\mathbf{z}} \Big\lVert \nabla F_{\mathbf{z}}(X_j) \Big\rVert^2 + 2\beta^{-1}(s-j\eta)^2d \, \nonumber \\
& \qquad \leq 2(s-j\eta)^2(1 + \delta)\left( M^2\sup_{j>0} \mathbb{E}_{\mathbf{z}}\lVert X_j \rVert^2 + B^2 \right) +2\beta^{-1}(s-j\eta)^2 d\,. 
\end{align}
The last inequality is from Assumption~\ref{A4} and the quadratic bound for $\nabla F_{\mathbf{z}}$ in Lemma~\ref{slm:quadratic-bd}, such that $\mathbb{E}_{\mathbf{z}}\lVert \nabla F_{\mathbf{z}}(X_j) \rVert^2 \leq 2\left(M^2 \sup_{j>0} \mathbb{E}_{\mathbf{z}}\lVert X_j \rVert^2+B^2\right)  $. For some $s-j\eta < \eta < 1  $ and $\delta<1$, the bound for the first term in \eqref{pf:lm:entropy} can be computed as,
\begin{align}
& \frac{\beta}{2} M^2\lVert A_J \rVert^2 \sum_{j=0}^{k-1} \int_{j\eta}^{(j+1)\eta} \mathbb{E}_{\mathbf{z}} \Big\lVert \overline{X}(s) - \overline{X}(\lfloor{s/\eta}\rfloor \eta ) \Big\rVert^2 ds \nonumber \\
& \qquad\qquad \leq \beta M^2 \lVert A_J \rVert^2 k\eta \left( 2\eta(1+\delta)\left(M^2 \sup_{j>0} \mathbb{E}_{\mathbf{z}}\lVert X_j \rVert^2+B^2\right) + 2\beta^{-1}d \eta \right)\,.
\end{align}
With Assumption~\ref{A4}, for $ j\eta \leq s < (j+1)\eta $, we can rewrite the second term in \eqref{pf:lm:entropy} as
\begin{align}
& \frac{\beta}{2} \lVert A_J \rVert^2 \sum_{j=0}^{k-1} \int_{j\eta}^{(j+1)\eta} \mathbb{E}_{\mathbf{z}} \Big\lVert \nabla F_{\mathbf{z}}(\overline{X}_j)-g\left(\overline{X}_j,\overline{U}_{\mathbf{z},j}\right)  \Big\rVert^2 ds \nonumber \\
& \qquad\qquad \leq \beta \lVert A_J \rVert^2 \left( M^2\sup_{j>0} \mathbb{E}_{\mathbf{z}}\lVert X_j \rVert^2 + B^2 \right)k\eta\delta\,. 
\end{align}
Combining these two inequality, we can have an upper bound for the relative entropy in \eqref{pf:lm:entropy},
\begin{align}
D\left( \widetilde{\mathbb{P}}_{k\eta} \Vert \mathbb{P}_{k \eta} \right) & \leq  \beta M^2 \lVert A_J \rVert^2 k\eta \left( 2\eta(1+\delta)\left(M^2 \sup_{j>0} \mathbb{E}_{\mathbf{z}}\lVert X_j \rVert^2+B^2\right) + 2\beta^{-1}d \eta \right) \nonumber \\
& \qquad\qquad + \beta \lVert A_J \rVert^2 \left( M^2\sup_{j>0} \mathbb{E}_{\mathbf{z}}\lVert X_j \rVert^2 + B^2 \right)k\eta\delta \,. 
\end{align}
From Lemma~\ref{lm:uniform-L2}, we conclude that, for any $\eta \leq \frac{m^2}{(m^2+ 8M^2)M\lVert A_J\rVert^2} \wedge 1 $,  
\begin{align} \label{pf:lm3:rel-etrp2}
D\left( \mu_{\mathbf{z},k} \Vert \nu_{\mathbf{z},k\eta}  \right) & \leq 2M^2 \beta \left( M^2 \mathcal{C}_d + B^2 + \beta^{-1}d \right)  \lVert A_J \rVert^2k\eta^2 
\nonumber
\\
&\qquad\qquad
+  (1+2M^2)\beta \left( M^2 \mathcal{C}_d +B^2 \right) \lVert A_J \rVert^2k\eta\delta \, \nonumber \\
& = (C_0\eta + C_1\delta)\lVert A_J \rVert^2(k\eta) \,,
\end{align}
where
\begin{align}
C_0=2\beta M^2 \left( M^2 \mathcal{C}_d + B^2 + \frac{d}{\beta} \right), \qquad C_1=(1+2M^2)\beta \left( M^2 \mathcal{C}_d +B^2 \right)\,.
\end{align}
With $ k\eta \geq e  $, $\eta \leq \sqrt{\eta} \leq 1  $ and $\delta < \sqrt{\delta} < 1  $, then we can also compute
\begin{align}
D(\mu_{\mathbf{z},k} \Vert \nu_{\mathbf{z},k\eta}) +\sqrt{D(\mu_{\mathbf{z},k} \Vert \nu_{\mathbf{z},k\eta})}
& = (C_0\eta+C_1\delta)\lVert A_J \rVert^2k\eta+(\sqrt{C_0\eta} + \sqrt{C_1\delta})\lVert A_J \rVert \sqrt{k\eta} \nonumber \\
& \leq \left((C_0+\sqrt{C_0})\sqrt{\eta}  + (C_1+\sqrt{C_1})\sqrt{\delta} \right) \lVert A_J \rVert^2 k\eta \nonumber \\
\end{align} 
The result from Bolley and Villani \cite{BV05} states that for any two Borel probability measures $\mu, \nu   $ on $\mathbb{R}^d  $ with finite second moment,
$$
\mathcal{W}_2(\mu,\nu) \leq C_{\nu}\left[\sqrt{D(\mu \Vert \nu)} +\left( \frac{D(\mu \Vert \nu)}{2} \right)^{1/4}\right]\,.
$$ 
with 
$$
C_\nu = 2\inf_{\lambda>0} \left( \frac{1}{\lambda}\left( \frac{3}{2}+\log \int_{\mathbb{R}^d} e^{\lambda\lVert x \rVert^2} \nu(dx) \right) \right)^{1/2}\,.
$$
Let $\mu=\mu_{\mathbf{z},k}  $, $\nu = \nu_{\mathbf{z},k\eta}$ and take $\lambda=1$, inferring from Lemma~\ref{lm:exponential}, $k\eta \geq 1$, we can compute
\begin{align}
\mathcal{W}^2_2(\mu_{\mathbf{z},k},\nu_{\mathbf{z},k\eta}) & \leq 4\log(L_0+L_1(k\eta)) \left[ \sqrt{D(\mu_{\mathbf{z},k} \Vert \nu_{\mathbf{z},k\eta})} +\left( \frac{D(\mu_{\mathbf{z},k} \Vert \nu_{\mathbf{z},k\eta})}{2} \right)^{1/4} \right]^2 \,
\nonumber \\
& \leq 8(\log(L_0+L_1)+\log(k\eta))\left[ D(\mu_{\mathbf{z},k} \Vert \nu_{\mathbf{z},k\eta}) +\sqrt{D(\mu_{\mathbf{z},k} \Vert \nu_{\mathbf{z},k\eta})} \right] \,,
\end{align}
where $\lVert X(0) \rVert \leq R = \sqrt{b/m}$. Additionally, let $k\eta \geq e$, we have $(k\eta) \log(k\eta) > k\eta$, hence
\begin{equation*}
\mathcal{W}^2_2(\mu_{\mathbf{z},k},\nu_{\mathbf{z},k\eta}) \leq \left(\hat{C}_0^2\sqrt{\eta}+\hat{C}_1^2\sqrt{\delta}\right) (k\eta)\log(k\eta) \lVert A_J \rVert^2\,, 
\end{equation*}
with 
$$
\hat{C}_0= \left(16\log\left(L_0+L_1\right) \left(C_0+\sqrt{C_0} \right)\right)^{1/2}\,, \qquad \hat{C}_1=\left(16\log\left(L_0+L_1\right) \left(C_1+\sqrt{C_1} \right)\right)^{1/2}\,,
$$
where $C_0 = \tilde{\mathcal{O}}(\beta+d)$ and $C_1=\tilde{\mathcal{O}}(\beta)$ are in \eqref{lm:const:approximation}, $L_0 = \tilde{\mathcal{O}}(1)$ in~\eqref{lm:const:l0} and $L_1 = e^{\tilde{\mathcal{O}}(\beta)}$ in ~\eqref{lm:const:l1}. 
The proof is complete.
\end{proof}

%%%%%%%%%%%%%%%%%%%%%%%%
\subsection{Proof of Lemma~\ref{lm:uniform-L4}}
\begin{proof}
To prove the uniform $L^4$ bound for the non-reversible Langevin SDE \eqref{nonreversibleSDE}, we first recall the quadratic bound for $F_{\mathbf{z}}$ in \eqref{slm:objfunc-bd},
\begin{equation*}
\frac{m}{3}\lVert x \rVert^2 - b < \frac{m}{3}\lVert x \rVert^2 - \frac{b}{2} \log 3 \leq F_{\mathbf{z}}(x) \,,
\end{equation*} 
which implies the following:
\begin{equation}
\lVert x \rVert^4 \leq \left(\frac{3}{m}F_{\mathbf{z}}(x)+\frac{3b}{m}\right)^2 \leq \frac{9}{m^2}F_{\mathbf{z}}^2(x) + \frac{9b(M+B)}{m^2}\lVert x \rVert^2 + \frac{18b}{m^2}\left(\frac{B}{2} +A \right) + \frac{9b^2}{m^2}\,.
\label{pr:lm:x^4-bd}
\end{equation}
For $X(0) = x_0 \in \mathbb{R}^d$ with $\lVert x_0 \rVert \leq R = \sqrt{b/m} $, there is a uniform bound for $\mathbb{E}\lVert X(t) \rVert^2$, i.e. $\mathbb{E}\lVert X(t) \rVert^2 \leq \mathcal{C}_c$ where $\mathcal{C}_c$ is a constant in \eqref{L2-SDE} from Lemma~\ref{lm:uniform-L2}. 
Next, we focus on computing an upper bound for $\mathbb{E} F_{\mathbf{z}}^2(X(t)) $.

Recall the the infinitesimal generator of the SDE in \eqref{nonreversibleSDE}:
\begin{equation*}
\mathcal{L}_J=-A_J \nabla F_{\mathbf{z}} \cdot \nabla + \beta^{-1} \Delta \,. 
\end{equation*}
Then, we can compute that
\begin{equation*}
\mathcal{L}_J F_{\mathbf{z}}(x) =-A_J \nabla F_{\mathbf{z}}(x) \cdot \nabla F_{\mathbf{z}}(x)+\beta^{-1} \Delta F_{\mathbf{z}}(x) = - \lVert \nabla F_{\mathbf{z}}(x) \rVert^2+\beta^{-1} \Delta F_{\mathbf{z}}(x)\,,
\end{equation*}
where $J$ is an $d \times d$ anti-symmetric matrix so that $\langle J \, \nabla F_{\mathbf{z}},\nabla F_{\mathbf{z}} \rangle = 0 $.
Moreover, we can compute that
\begin{align}
\mathcal{L}_J F_{\mathbf{z}}^2(x) & = \left(2F_{\mathbf{z}}\mathcal{L}_JF_{\mathbf{z}} + 2\beta^{-1}\lVert \nabla F_{\mathbf{z}} \rVert^2 \right) \nonumber \\
& = -2(F_{\mathbf{z}}(x)-\beta^{-1})\lVert \nabla F_{\mathbf{z}}(x) \rVert^2 + 2\beta^{-1}F_{\mathbf{z}}(x) \Delta F_{\mathbf{z}}(x)\,. \label{pf:lm:generator2}
\end{align}
By $(m,b)$-dispassive property, we have
\begin{equation*}
\lVert x \rVert \lVert \nabla F_{\mathbf{z}}(x) \rVert \geq \langle x, \nabla F_{\mathbf{z}}(x) \rangle \geq m \lVert x \rVert^2-b\,,
\end{equation*}
so that for any $\lVert x \rVert \geq \sqrt{2b/m}$, we have
\begin{equation*}
\lVert \nabla F_{\mathbf{z}}(x) \rVert \geq m \lVert x \rVert - \frac{b}{\lVert x \rVert} \geq \frac{m}{2} \lVert x  \rVert\,,
\end{equation*}
which yields $\lVert \nabla F_{\mathbf{z}}(x) \rVert^2 \geq (m^2/4)\lVert x \rVert^2$. Moreover, the objective function $F_{\mathbf{z}}$ is $M$-smooth, so that $\Delta F_{\mathbf{z}}(x) \leq Md $.  

For any $\lVert x \rVert \geq \sqrt{2b/m}$, under the assumption that $\beta \geq 3/m$, the quadratic bound for $F_{\mathbf{z}}$ in equation \eqref{slm:objfunc-bd} shows that 
\begin{equation}
\frac{m}{3}\lVert x \rVert^2 - b < F_{\mathbf{z}}(x) \leq \frac{M+B}{2}\lVert x \rVert^2 + \frac{B}{2} +A \,, \label{pf:lm:Fquadratic-bd}
\end{equation}
Then it follows from \eqref{pf:lm:generator2}, for $\lVert x \rVert \geq \sqrt{2b/m}$, we have
\begin{align}
\mathcal{L}_JF_{\mathbf{z}}^2(x) & \leq -2\left(\frac{m}{3}\lVert x \rVert^2-b-\beta^{-1}\right) \cdot \frac{m^2}{4}\lVert x \rVert^2 + 2\beta^{-1}\left( \frac{M+B}{2}\lVert x \rVert^2 + \frac{B}{2}+A \right)Md \nonumber \\
& = -\frac{m^2}{2}\lVert x \rVert^2 \left( \frac{m}{3}\lVert x \rVert^2 - b -\beta^{-1} - \frac{2(M+B)\beta^{-1}}{m^2}  \right) + 2\beta^{-1}\left(\frac{B}{2}+A\right)Md\,.
\label{equivalent:to}
\end{align}
Let
\begin{equation*}
S :=  \left(\frac{6b}{m}+\frac{6\beta^{-1}}{m}+\frac{12(M+B)\beta^{-1}}{m^3}\right)^{1/2} > \sqrt{2b/m}\,.
\end{equation*}
Then \eqref{equivalent:to} is equivalent to
\begin{equation}
\mathcal{L}_JF_{\mathbf{z}}^2(x) \leq -\frac{m^2}{2}\lVert x \rVert^2 \left( \frac{m}{3}\lVert x \rVert^2 - \frac{m}{6}S^2  \right) + 2\beta^{-1}\left(\frac{B}{2}+A\right)Md \,. \label{pf:lm:Sbreak}
\end{equation}
Therefore, if $\lVert x \rVert > S$, for \eqref{pf:lm:Sbreak}, we have 
\begin{align}
\mathcal{L}_JF_{\mathbf{z}}^2(x) & \leq -\frac{m^2}{2}\lVert x \rVert^2 \left( \frac{m}{3}\lVert x \rVert^2 - \frac{m}{6}\lVert x \rVert^2 \right) + 2\beta^{-1}\left(\frac{B} {2}+A\right)Md \nonumber \\
& =-\frac{m^3}{12}\,\lVert x \rVert^4 + (B+2A)\,\beta^{-1}Md\,. 
\end{align}
On the other hand, if $\sqrt{2b/m} < \lVert x \rVert \leq S$, we obtain from \eqref{pf:lm:Sbreak} that
\begin{align*}
\mathcal{L}_JF_{\mathbf{z}}^2(x) 
&\leq -\frac{m^3}{6}\,\lVert x \rVert^4 + \frac{m^3}{12}\lVert x \rVert^2 S^2 + (B+2A)\,\beta^{-1}Md
\\
&\leq -\frac{m^3}{12}\,\lVert x \rVert^4 + \frac{m^3}{12} S^4 + (B+2A)\,\beta^{-1}Md \,.
\end{align*}
To summarize, for any $\lVert x \rVert \geq \sqrt{2b/m}$, we have, 
\begin{align}
\mathcal{L}_JF_{\mathbf{z}}^2(x) &\leq -\frac{m^3}{12}\,\lVert x \rVert^4 + \frac{m^3}{12}S^4 + (B+2A)\,\beta^{-1}Md \nonumber \\
&\leq -\frac{m^3}{12} \lVert x \rVert^4 + 3m\left( b+\beta^{-1}+\frac{2(M+B)\beta^{-1}}{m^2} \right)^2 + (B+2A)\beta^{-1}Md\,. \label{pf:lm:f^2-x^4a}
\end{align}
Next we consider the case $\lVert x \rVert \leq \sqrt{2b/m}$ and obtain from the equation \eqref{pf:lm:generator2} that,
\begin{align*}
\mathcal{L}_J F_{\mathbf{z}}^2(x) &= -2 F_{\mathbf{z}}(x) \lVert \nabla F_{\mathbf{z}}(x) \rVert^2 + 2\beta^{-1}\lVert \nabla F_{\mathbf{z}}(x) \rVert^2  + 2\beta^{-1}F_{\mathbf{z}}(x) \Delta F_{\mathbf{z}}(x)
\\
&\leq
2\beta^{-1}\lVert \nabla F_{\mathbf{z}}(x) \rVert^2  + 2\beta^{-1}F_{\mathbf{z}}(x) \Delta F_{\mathbf{z}}(x) \,,
\end{align*}
where we use the fact $F_{\mathbf{z}}$ function is non-negative in Assumption~\ref{A1}. 
By applying the quadratic bounds in Lemma~\ref{slm:quadratic-bd} that 
\begin{equation*}
\lVert \nabla F_{\mathbf{z}}(x) \rVert \leq M \lVert x \rVert + B \,, \qquad \qquad F_{\mathbf{z}}(x) \leq \frac{M}{2} \lVert x \rVert^2 + B \lVert x \rVert + A \,,
\end{equation*}
and $M$-smoothness of $F_{\mathbf{z}}$ so that $\Delta F_{\mathbf{z}}\leq Md$, we get
\begin{align}
\mathcal{L}_J F_{\mathbf{z}}^2(x) &\leq 2\beta^{-1}\left( 2M^2 \lVert x \rVert^2+2B^2 \right)  + 2\beta^{-1}\left( \frac{M}{2}\lVert x \rVert^2 + B\lVert x \rVert + A \right) Md \,
\nonumber \\
&\leq \frac{8b\beta^{-1}M^2}{m}+4\beta^{-1}B^2 + \left( \frac{2b\beta^{-1}M}{m} + 2\beta^{-1}B\sqrt{2b/m} + 2\beta^{-1}A \right)Md \,. \label{pf:lm:f^2-x^4b}
\end{align}
Hence, for any $x \in \mathbb{R}^d$, we can compute from \eqref{pf:lm:f^2-x^4a} and \eqref{pf:lm:f^2-x^4b},
\begin{align}
\mathcal{L}_JF_{\mathbf{z}}^2(x) &\leq -\frac{m^3}{12}\,\lVert x \rVert^4 + \frac{m^3}{12}S^4 + (B+2A)\,\beta^{-1}Md  \nonumber \\
& \qquad + \frac{8b\beta^{-1}M^2}{m}+4\beta^{-1}B^2 + \left( \frac{2b\beta^{-1}M}{m} + 2\beta^{-1}B\sqrt{2b/m} + 2\beta^{-1}A \right)Md \nonumber \\
& \leq -\frac{m^3}{12} \lVert x \rVert^4 + 3m\left( b+\beta^{-1}+\frac{2(M+B)\beta^{-1}}{m^2} \right)^2 + \frac{8b\beta^{-1}M^2}{m}+4\beta^{-1}B^2 \, \nonumber \\
& \qquad\qquad\qquad + \beta^{-1}\left(B+2B\sqrt{2b/m} +  \frac{2bM}{m} +  4A  \right) Md \,.
\label{pf:lm:f^2-x^4}
\end{align}
Then using the quadratic bounds for $F_{\mathbf{z}}$ in \eqref{pf:lm:Fquadratic-bd}:
\begin{equation}
\frac{2}{M+B}\left(F_{\mathbf{z}}(x)-\frac{B}{2}-A\right) \leq \lVert x \rVert^2 \leq \frac{3}{m}F_{\mathbf{z}}(x)+\frac{3b}{m}\,.
\end{equation}
we get
\begin{equation*}
2\lVert x \rVert^4 +  \frac{8(B/2+A)^2}{(M+B)^2} \geq \left( \lVert x \rVert^2 + \frac{2(B/2+A)}{M+B}\right)^2 \geq \frac{4}{(M+B)^2}F_{\mathbf{z}}^2(x)\,.
\end{equation*}
Hence, we have
\begin{equation*}
\lVert x \rVert^4 \geq \frac{2}{(M+B)^2}F_{\mathbf{z}}^2(x) - \frac{(B+2A)^2}{(M+B)^2}\,.
\end{equation*}
Then by \eqref{pf:lm:f^2-x^4}, we can compute,
\begin{align}
\mathcal{L}_JF_{\mathbf{z}}^2(x) \leq & -\frac{m^3}{6(M+B)^2} F_{\mathbf{z}}^2(x) + \frac{m^3(B+2A)^2}{12(M+B)^2}+ 3m\left( b+\beta^{-1}+\frac{2(M+B)\beta^{-1}}{m^2} \right)^2 \nonumber\\
& \qquad\qquad + \frac{8b\beta^{-1}M^2}{m}+4\beta^{-1}B^2 + \beta^{-1}\left(B+2B\sqrt{2b/m} +  \frac{2bM}{m} +  4A  \right) Md \,.
\end{align}
Let 
$$
\mathcal{L}_JF_{\mathbf{z}}^2(x) \leq -c_1 F_{\mathbf{z}}^2(x) + K_1 \,.
$$
By It\^{o}'s formula, we get
\begin{align}
 e^{c_1t}F^{2}_{\mathbf{z}}(X(t)) &= F_{\mathbf{z}}^2(X(0)) + \int_0^t e^{c_1s} \mathcal{L}_J F^2_{\mathbf{z}}(X(s)) ds \nonumber
 \\
 &\qquad\qquad\qquad
 + \int_0^t e^{c_1s} F_{\mathbf{z}}(X(s)) \nabla F_{\mathbf{z}}(X(s)) \sqrt{2\beta^{-1}}dB(s) \,, \label{pf:lm:martingale2}
\end{align}
By using Corollary 2.4 \cite{CHJ13} in \eqref{pf:lm:exponential} and the similar argument in \eqref{pf:lm:sqrtintegrability}, we can show
$$
\mathbb{E} \int_0^t \biggl\lVert e^{c_1s} F_{\mathbf{z}}(X(s)) \nabla F_{\mathbf{z}}(X(s)) \sqrt{2\beta^{-1}} \biggr\rVert^2 ds \leq \mathbb{E} \int_0^t e^{2c_1s} \biggl\lVert e^{F_{\mathbf{z}}(X(s))} \nabla F_{\mathbf{z}}(X(s)) \biggr\rVert^2 2\beta^{-1} ds < \infty \,.
$$
Hence, the last term in \eqref{pf:lm:martingale2} is a martingale. Taking expectation for both side, we have
\begin{align}
e^{c_1t}\mathbb{E}F_{\mathbf{z}}^2(X(t)) \leq \mathbb{E}F_{\mathbf{z}}^2(X(0)) + \int_0^t e^{c_1s} K_1 ds \,.
\end{align}
It implies,
$$
\mathbb{E} F_{\mathbf{z}}^2(X(t)) \leq e^{-c_1t}\mathbb{E} F_{\mathbf{z}}^2(X(0))+\frac{K_1(1-e^{-c_1t})}{c_1} \,,
$$
taking $ t \rightarrow \infty$, we have
$$
\mathbb{E} F_{\mathbf{z}}^2(X(t)) \leq \mathbb{E} F_{\mathbf{z}}^2(X(0))+\frac{K_1}{c_1}\,.
$$
Therefore, we have
\begin{align}
\mathbb{E} F_{\mathbf{z}}^2(X(t)) 
&\leq F_{\mathbf{z}}^2(x_0) + \frac{(B+2A)^2}{2} + \frac{18(M+B)^2}{m^2}\left(b+\beta^{-1}+\frac{2(M+B)\beta^{-1}}{m^2} \right)^2  \nonumber \\
&\qquad  + \frac{24\beta^{-1}(2bM^2+mB^2)(M+B)^2}{m^4} 
\nonumber
\\
&\qquad\qquad
+ \frac{6\beta^{-1}(M+B)^2}{m^3}\left(B+2B\sqrt{2b/m} +  \frac{2bM}{m} +  4A  \right) Md \,,
\label{pf:lm:lyapunov}
\end{align}
where $X(0) = x_0 \in \mathbb{R}^d$. 
Furthermore, we can use the the quadratic bound for $F_{\mathbf{z}}$ in \eqref{slm:objfunc-bd} 
to get the bound for the first term:
\begin{equation*}
F_{\mathbf{z}}^2(x_0) \leq \left(\frac{M}{2}\lVert x_0 \rVert^2+B\lVert x_0 \rVert + A \right)^2 \leq \left(\frac{M}{2}R^2+BR + A \right)^2 \,.
\end{equation*}
As a result, we can compute the uniform $L^4$ bound $\mathbb{E}\lVert X(t) \rVert^4$ by using the relations in \eqref{pr:lm:x^4-bd} and \eqref{pf:lm:lyapunov}: 
\begin{equation*}
\mathbb{E}\lVert X(t) \rVert^4 \leq \frac{9}{m^2}\mathbb{E} F_{\mathbf{z}}^2(X(t)) +  \frac{9b(M+B)}{m^2}\mathbb{E}\lVert X(t) \rVert^2 + \frac{18b}{m^2}\left(\frac{B}{2} +A \right) + \frac{9b^2}{m^2}\,.
\end{equation*}
Hence, we conclude that
\begin{align}
\mathbb{E} \lVert X(t) \rVert^4 &\leq \mathcal{D}_c = \frac{9}{m^2}\left(\frac{M}{2}R^2+BR + A \right)^2 + \frac{9U+9b(M+B)\mathcal{C}_c}{m^2} \nonumber \\
& \qquad +  \frac{6M(M+B)^2}{m^3}\left(B+2B\sqrt{2b/m} +  \frac{2bM}{m} +  4A  \right) \beta^{-1}d \,,
\end{align}
with 
\begin{align*}
U&= \frac{(B+2A)^2}{2} + \frac{18(M+B)^2}{m^2}\left(b+\beta^{-1}+\frac{2(M+B)\beta^{-1}}{m^2} \right)^2 
\\
&\qquad + \frac{24\beta^{-1}(2bM^2+mB^2)(M+B)^2}{m^4} + 2bB + 2A + b^2 \,.
\end{align*}
The proof is complete. 
\end{proof}

%%%%%%%%%%%%%%%%%%%%%%%%%%%
\subsection{Proof of Lemma~\ref{lm:acceleration}}
\begin{proof}
By following the same arguments as in the proof of Theorem~4 in \cite{HHS05}, 
we have
\begin{align*}
&\int_{\mathbb{R}^{d}}|p_{\mathbf{z},J}(t,x,y)-\pi_{\mathbf{z}}(y)|dy
\\
&=\int_{\mathbb{R}^{d}}
\left|\int_{\mathbb{R}^{d}}(p_{s,\mathbf{z},J}(x,z)p_{t-s,\mathbf{z},J}(z,y)-1)\pi_{\mathbf{z}}(z)dz\right|\pi_{\mathbf{z}}(y)dy
\\
&=\int_{\mathbb{R}^{d}}
\left|\int_{\mathbb{R}^{d}}(p_{s,\mathbf{z},J}(x,z)p^{\ast}_{t-s,\mathbf{z},J}(y,z)-1)\pi_{\mathbf{z}}(z)dz\right|\pi_{\mathbf{z}}(y)dy
\\
&=\int_{\mathbb{R}^{d}}
\left|\int_{\mathbb{R}^{d}}p^{\ast}_{t-s,\mathbf{z},J}(y,z)p_{s,\mathbf{z},J}(x,z)\pi_{\mathbf{z}}(z)dz
-\int_{\mathbb{R}^{d}}p_{s,\mathbf{z},J}(x,z)\pi_{\mathbf{z}}(z)dz\right|\pi_{\mathbf{z}}(y)dy
\\
&=\int_{\mathbb{R}^{d}}|\mathcal{T}_{\mathbf{z},J}^{\ast}(t-s)(p_{s,\mathbf{z},J}(x,\cdot))(y)-\pi_{\mathbf{z}}(p_{s,\mathbf{z},J}(x,\cdot))|\pi_{\mathbf{z}}(y)dy
\\
&\leq\left(\int_{\mathbb{R}^{d}}|\mathcal{T}_{\mathbf{z},J}^{\ast}(t-s)(p_{s,\mathbf{z},J}(x,\cdot))(y)
-\pi_{\mathbf{z}}(p_{s,\mathbf{z},J}(x,\cdot))|^{2}\pi_{\mathbf{z}}(y)dy\right)^{1/2}
\\
&\leq C_{\mathbf{z}, J}\cdot e^{|\lambda_{\mathbf{z},J}|s}\Vert p_{s,\mathbf{z},J}(x,\cdot)-1\Vert
e^{\lambda_{\mathbf{z},J}t},
\end{align*}
for any given $0<s<t$, 
where $C_{\mathbf{z},J}$ is from the
spectral inequality \eqref{eq:L2-contraction-SG}
and $p_{s,\mathbf{z},J}(x,y):=p_{\mathbf{z},J}(s,x,y)/\pi_{\mathbf{z}}(y)$,
where $\mathcal{T}_{\mathbf{z},J}^{\ast}$ is the adjoint of
the semigroup $\mathcal{T}_{\mathbf{z},J}$, that is defined as
for any $x \in \mathbb{R}^d$ and $s \in \mathbb{R}^+$,
\begin{equation*}
\mathcal{T}_{\mathbf{z},J}(s)f(x) = \int_{\mathbb{R}^d} p_{\mathbf{z},J}(s,x,y)f(y)dy \,,
\end{equation*}
where $\mathcal{T}_{\mathbf{z},J}(s) = e^{s\mathcal{L}_{J}}$ and $\mathcal{L}_J$ is the corresponding infinitesimal generator. 

According to the proof of Theorem~4 in \cite{HHS05}, 
\begin{equation}
\Vert p_{s,\mathbf{z},J}(x,\cdot)-1\Vert
\leq\Vert p_{s,\mathbf{z},J}(x,\cdot)\Vert+1
\leq\frac{C_{\mathbf{z},J}(N,x)}{\pi_{\mathbf{z}}(B(x,N/2))}+1,
\end{equation}
where $C_{\mathbf{z},J}(N,x)$ is the Harnack constant
in the following Harnack inequality:
\begin{equation}
\sup_{y\in B(x,N/2)}\mathcal{T}_{\mathbf{z},J}(s)f(y)
\leq
C_{\mathbf{z},J}(N,x)
\inf_{y\in B(x,N/2)}\mathcal{T}_{\mathbf{z},J}(2s)f(y),
\end{equation}
for any $x\in\mathbb{R}^{d}$, $N>0$, and any $f$ with $\pi_{\mathbf{z}}(f)=1$ and $f\geq 0$.

By applying Theorem 2.4. in \cite{BRS2008}\footnote{In \cite{BRS2008}, it is backward in time,
and by taking $t\mapsto-t$, we can apply their result forward in time.}, we have
\begin{equation}
\sup_{(y,t)\in Q^{\ast}(r)}\mathcal{T}_{\mathbf{z},J}(t)f(y)
\leq C(d,\alpha,\gamma,\tilde{B},r)\inf_{(y,t)\in Q(r)}\mathcal{T}_{\mathbf{z},J}(t)f(y),
\end{equation}
where
\begin{equation}
C(d,\alpha,\gamma,\tilde{B},r)=\exp\left\{c(d)\left(1+\alpha^{-1}+(\alpha^{-1/2}+\alpha^{-1})(\tilde{B}r+\gamma)\right)^{2}\right\},    
\end{equation}
for some constant $c(d)$ depending only on $d$, with
\begin{equation}
\alpha=\beta^{-1},
\qquad
\gamma=\sqrt{\beta^{-1}d},
\qquad
\tilde{B}=\sup_{y\in Q}\Vert A_{J}\nabla F_{\mathbf{z}}(y)\Vert,
\end{equation}
where 
\begin{equation}
Q(r)=B(x,r)\times(t_{0},t_{0}+r^{2}),
\qquad
Q^{\ast}(r)=B(x,r)\times(t_{0}+7r^{2},t_{0}+8r^{2}),
\end{equation}
provided that $x\in Q(3r)\subset Q$, and $Q(r),Q^{\ast}(r)\subset Q$,
with $Q:=B(x,4r)\times(0,1)$, where
\begin{equation}
r=N/2,
\qquad
t_{0}=6r^{2},
\end{equation}
so that 
\begin{equation}
s\in(t_{0},t_{0}+r^{2})=(6r^{2},7r^{2}),
\qquad
2s\in(t_{0}+7r^{2},t_{0}+8r_{2})=(13r^{2},14r^{2}),
\end{equation}
and we can take
\begin{equation}
s=\frac{27}{4}r^{2},
\end{equation}
and $r=1/4$ so that $Q(3r),Q(r),Q^{\ast}(r)\subset Q$.

Hence, we conclude that with $s=\frac{27}{4}r^{2}=\frac{27}{64}$, 
we have
\begin{align}
C_{\mathbf{z},J}(N,x)
&\leq
\exp\left\{c(d)\left(1+\beta+(\beta^{1/2}+\beta)\left(\frac{1}{4}\sup_{y\in B(x,1)}\Vert A_{J}\nabla F_{\mathbf{z}}(y)\Vert+\sqrt{\beta^{-1}d}\right)\right)^{2}\right\}\nonumber
\\
&\leq
\exp\left\{c(d)\left(1+\beta+(\beta^{1/2}+\beta)\left(\frac{1}{4}\Vert A_{J}\Vert\sup_{y\in B(x,1)}(M\Vert y\Vert+B)+\sqrt{\beta^{-1}d}\right)\right)^{2}\right\}\nonumber
\\
&\leq
\exp\left\{c(d)\left(1+\beta+(\beta^{1/2}+\beta)\left(\frac{1}{4}\Vert A_{J}\Vert(M\Vert x\Vert+M+B)+\sqrt{\beta^{-1}d}\right)\right)^{2}\right\}\,,
\end{align}
with
\begin{equation}
N=2r=\frac{1}{2}.
\end{equation}

Next, let us provide a lower bound for $\pi_{\mathbf{z}}(B(x,N/2))$, 
where $B(x,r)$ is an $\mathbb{R}^d$ Euclidean ball centered at $x$ with radius equals to $r$.
For a fixed $x \in \mathbb{R}^d$, we can compute 
$$
\pi_{\mathbf{z}}(B(x,N/2)) = \frac{1}{\Lambda_{\mathbf{z}}}\int_{\lVert y-x \rVert \leq N/2} e^{-\beta F_{\mathbf{z}}(y)}dy = \frac{1}{\Lambda_{\mathbf{z}}}\int_{\lVert w \rVert \leq N/2} e^{-\beta F_{\mathbf{z}}(w+x)}dw  \,,
$$
In addition, $F_{\mathbf{z}}$ function has quadratic bounds in Lemma~\ref{slm:quadratic-bd},
\begin{equation*}
\frac{m}{3}\lVert w+x \rVert^2 - \frac{b}{2} \log3 \leq F_{\mathbf{z}}(w+x) \leq \frac{M+B}{2}\lVert w+x \rVert^2 + \frac{B}{2} + A \leq  (M+B)\left(\lVert w \rVert^2 + \lVert x \rVert^2\right) + \frac{B}{2} + A \,.
\end{equation*}
It then follows that
\begin{align*}
\pi_{\mathbf{z}}(B(x,N/2)) 
&\geq \frac{e^{-\beta(M+B)\left( \lVert x \rVert^2 + \frac{B}{2} + A\right)}}{\Lambda_{\mathbf{z}}} \int_{\lVert w \rVert \leq N/2} e^{-\beta (M+B)\lVert w \rVert^2}dw 
\\
&\geq \frac{e^{-\beta(M+B)\left( \lVert x \rVert^2 + \frac{B}{2} + A+\frac{N^{2}}{4}\right)}}{\Lambda_{\mathbf{z}}} \frac{(2\pi)^{d/2}}{\Gamma(d/2+1)}\left(\frac{N}{2}\right)^{2} \,.
\end{align*}
The normalized constant $\Lambda_{\mathbf{z}}$ is bounded by using Gaussian integral and the quadratic bounds for $F$ in Lemma~\ref{slm:quadratic-bd},
$$
\Lambda_{\mathbf{z}} = \int_{\mathbb{R}^d} e^{-\beta F_{\mathbf{z}}(y)} dy \leq e^{\beta b(\log3 )/2} \int_{\mathbb{R}^d} e^{-\frac{m\beta}{3}\lVert y \rVert^2} dy \leq \left(\frac{3\pi}{m\beta} \right)^{d/2}e^{\beta b(\log3 )/2} \,.
$$ 
Therefore, we have
\begin{align}
\pi_{\mathbf{z}}(B(x,N/2)) 
\geq 
\left(\frac{3\pi}{m\beta} \right)^{-d/2}e^{-\beta b(\log3 )/2}
e^{-\beta(M+B)\left( \lVert x \rVert^2 + \frac{B}{2} + A+\frac{N^{2}}{4}\right)} \frac{(2\pi)^{d/2}}{\Gamma(d/2+1)}\left(\frac{N}{2}\right)^{2}.
\end{align}

%\gao{how do you obtain the last inequality and why the gamma function appears here? gaussian integral over a ball? It will make the final estimate to be too big. see comments at the end of proof}
Therefore, with $N=\frac{1}{2}$, $s=\frac{27}{64}$, we have
\begin{equation}
\int_{\mathbb{R}^{d}}|p_{\mathbf{z},J}(t,x,y)-\pi_{\mathbf{z}}(y)|dy
\leq C_{\mathbf{z}, J}\cdot g_{\mathbf{z},J}(\Vert x\Vert)\cdot e^{\lambda_{\mathbf{z},J}t},    
\end{equation}
where
\begin{align}
&e^{|\lambda_{\mathbf{z},J}|s}\Vert p_{s,\mathbf{z},J}(x,\cdot)-1\Vert
\nonumber
\\
&\leq
e^{|\lambda_{\mathbf{z},J}|\frac{27}{64}}
\left(\frac{e^{c(d)\left(1+\beta+(\beta^{1/2}+\beta)\left(\frac{1}{4}\Vert A_{J}\Vert(M\Vert x\Vert+M+B)+\sqrt{\beta^{-1}d}\right)\right)^{2}}}
{\left(\frac{3\pi}{m\beta} \right)^{-d/2}e^{-\beta b(\log3 )/2}
e^{-\beta(M+B)\left( \lVert x \rVert^2 + \frac{B}{2} + A + \frac{1}{16} \right)}\frac{(2\pi)^{d/2}}{\Gamma(d/2+1)}\frac{1}{16}}+1\right).
\end{align}

Next, let us compute constant $c(d)$ in $g_{\mathbf{z},J}(\lVert x \rVert)$. Inferring from the proof of Theorem 2.4 in \cite{BRS2008},
\begin{equation*}
c(d) = \tilde{c}(d)C^2(d) = \frac{\log(3 \cdot 2^{d+1})C^2(d)}{\tilde{\lambda}(d)}\,,
\end{equation*}
where $C(d)$ is defined in Lemma 2.4 \cite{BRS2008}, 
and the last equation from the proof given in section 6 of \cite{AS1967} gives the relation between $C(d)$ and $d$ as follows:
\begin{equation}
C(d) = 4\sqrt{2} \cdot 2^{-3d/2} \,.
\end{equation}
 Then $\tilde{\lambda}(d) = \frac{B^2_1(d,\beta)}{\lambda(d)}$ in Corollary 2.2 \cite{BRS2008} follows the well-known Moser lemma in \cite{Moser1964}. Inferring from Main Lemma in \cite{Moser1964} \footnote{The discussions from Page 128 to Page 130 \cite{Moser1964} indicate that $2c\epsilon < \frac{3}{4}$ and $0<\epsilon<\frac{1}{1+\tilde{M}^{d+2}}$ with $\tilde{M}\geq 2$. We take $c\epsilon = \frac{1}{4}$ and $\epsilon = \frac{1}{2(1+\tilde{M}^{d+2})}$ here. Also, we can see from Main Lemma and the proof of Theorem 4 in Page 124 \cite{Moser1964} that $\tilde{\lambda}(d)$ in \cite{BRS2008} equals to $\alpha$ in \cite{Moser1964} and it follows that $\Phi(s) = c^{-1}e^{\tilde{\lambda}(d)s}$.}, we get
\begin{equation*}
\tilde{\lambda}(d) =-\log\left( \frac{8\Phi(1)}{1+\tilde{M}^{2+d}} \right), \qquad \text{with }\qquad \tilde{M} \geq 2. 
\end{equation*}
And $\Phi(s)$ is a continuous function which is equal to zero for any $s\leq 0$ and it is strictly increasing for any positive $s$. For example, \cite{Moser1964} takes $\Phi(s) = \sqrt{s}$ and $\Phi(s) = \log(1+s)$ in their proofs.
Therefore, we have
\begin{equation}
\tilde{\lambda}(d)
=\log\left( \frac{1+\tilde{M}^{2+d}}{8\Phi(1)} \right)
\geq\tilde{C}_{0}d,
\end{equation}
for some universal constant $\tilde{C}_{0}>0$.
Therefore, we have
\begin{equation}
c(d)
=\frac{\log(3 \cdot 2^{d+1})C^2(d)}{\tilde{\lambda}(d)}
\leq
\frac{\log(3 \cdot 2^{d+1})32\cdot 2^{-3d}}{\tilde{C}_{0}d}
\leq\tilde{C}2^{-3d},
\end{equation}
for some universal constant $\tilde{C}>0$.

Hence, with $N=\frac{1}{2}$ and $s=\frac{27}{64}$,
we conclude that
\begin{equation}
\int_{\mathbb{R}^{d}}|p_{\mathbf{z},J}(t,x,y)-\pi_{\mathbf{z}}(y)|dy
\leq C_{\mathbf{z}, J}\cdot g_{\mathbf{z},J}(\Vert x\Vert)\cdot e^{\lambda_{\mathbf{z},J}t},    
\end{equation}
with
\begin{equation}
g_{\mathbf{z},J}(\Vert x\Vert)
=e^{|\lambda_{\mathbf{z},J}|\frac{27}{64}}
\left(\frac{16\Gamma(\frac{d}{2}+1)e^{\tilde{C}2^{-3d}\left(1+\beta+(\beta^{1/2}+\beta)\left(\frac{1}{4}\Vert A_{J}\Vert(M\Vert x\Vert+M+B)+\sqrt{\beta^{-1}d}\right)\right)^{2}}}
{\left(\frac{3}{2m\beta} \right)^{-d/2}e^{-\beta b(\log3 )/2}
e^{-\beta(M+B)\left( \lVert x \rVert^2 + \frac{B}{2} + A + \frac{1}{16} \right)}}+1\right) \,,
\end{equation}
for some universal constant $\tilde{C}>0$.
The proof is complete.
\end{proof}

%%%%%%%%%%%%%%%%%%%%%%%%%%%
\subsection{Proof of Lemma~\ref{lm:Ktail}}
\begin{proof}
First, we have the following estimate:
\begin{align}
& \int_{\mathbb{R}^d} \int_{\lVert x \rVert > K}\lVert x \rVert^2\left| p_{\mathbf{z},J}(t,w,x)-\pi_{\mathbf{z}}(x) \right| dx \, \nu_{\mathbf{z},0}(dw)  \nonumber\\
& \quad \leq \int_{\lVert x \rVert > K}\lVert x \rVert^2  \int_{\mathbb{R}^d} p_{\mathbf{z},J}(t,w,x) \nu_{\mathbf{z},0}(dw) \, dx +\int_{\lVert x \rVert > K}\lVert x \rVert^2\pi_{\mathbf{z}}(dx) \nonumber \\
& \quad =  \int_{\lVert x \rVert>K} \lVert x \rVert^2 \nu_{\mathbf{z},k \eta}(dx) + \int_{\lVert x \rVert > K}\lVert x \rVert^2\pi_{\mathbf{z}}(dx)  \nonumber \\
& \quad \leq \frac{\mathbb{E}_{\mathbf{z}}\lVert X(k\eta) \rVert^4}{K^2} + 
\frac{\mathbb{E}_{\mathbf{z}}\lVert X(\infty) \rVert^4}{K^2}\nonumber
\\
&\quad \leq \frac{\mathbb{E}_{\mathbf{z}}\lVert X(k\eta) \rVert^4}{K^2} + 
\frac{\limsup_{t\rightarrow\infty}\mathbb{E}_{\mathbf{z}}\lVert X(t) \rVert^4}{K^2}
\end{align}
where the first inequality is due the fact that $\lvert a-b \rvert \leq a+b, a,b>0$, 
and the second inequality is a result of Chebyshev's inequality
and the stationary distribution of $X(t)$ process is $\pi_{\mathbf{z}}$,
and the third inequality follows from Fatou's lemma. %in Lemma~\ref{slm:Chebyshev-ineqn}.
By the uniform $L^4$ bound in Lemma~\ref{lm:uniform-L4}, we get
\begin{equation}
\int_{\mathbb{R}^d} \int_{\lVert x \rVert > K}\lVert x \rVert^2\left| p_{\mathbf{z},J}(t,w,x)-\pi_{\mathbf{z}}(x) \right| dx \, \nu_{\mathbf{z},0}(dw)
\leq\frac{2\mathcal{D}_{c}}{K^{2}},
\end{equation}
with $\mathcal{D}_c$ being a constant for the uniform $L^4$ bound defined in \eqref{lm:const:Dc}.
Next, we take
\begin{equation}
K = e^{\lvert \lambda_{\mathbf{z},J} \rvert k\eta/4},\label{pf:lm:const:K} 
\end{equation}
so that
\begin{equation}
\int_{\mathbb{R}^d} \int_{\lVert x \rVert > K}\lVert x \rVert^2\left| p_{\mathbf{z},J}(t,w,x)-\pi_{\mathbf{z}}(x) \right| dx \, \nu_{\mathbf{z},0}(dw)
\leq
2\mathcal{D}_{c}e^{-\lvert \lambda_{\mathbf{z},J} \rvert k\eta/2}.
\end{equation}
The proof is complete.
\end{proof}
%%%%%%%%%%%%%%%%%%%%%%%%%%%%

%%%%%%%%%%%%%%%%%%%%%%%%
\section{Performance Bound for the Population Risk Minimization} \label{sec:p-r-m}
To obtain the performance bound for the population risk minimization in \eqref{eqn:populationrisk}, we control the expected population risk of $X_k$ in \eqref{NSGLD}: $\mathbb{E}F(X_k)-F^*.$ To this end, in addition to the empirical risk, one has to account for the differences between the finite sample size problem \eqref{eqn:empiricalsum} and the original problem \eqref{eqn:populationrisk}. In particular we have the following corollary. Define the uniform spectral gap $\lambda_{*,J} = \inf_{\mathbf{z}\in\mathcal{Z}^n} |\lambda_{\mathbf{z},J}| $. 

\begin{corollary}[Population risk minimization] \label{cor:PRM-bd}
Consider the iterates $\{X_k\}$ of the NSGLD algorithm in \eqref{NSGLD}.
Under the setting in Corollary~\ref{cor:NSGLD2gibbs}, for some constant $\beta \geq 3/m$ and $\varepsilon>0$, the upper bound for the expected population risk of $X_k$ is given by
\begin{equation}\label{pop:inequality}
\mathbb{E}F(X_k)-F^*  
\leq \overline{\mathcal{I}}_0(J,\varepsilon)+\mathcal{I}_1(J,\varepsilon)+\mathcal{I}_2+\mathcal{I}_3(n)\,,
\end{equation}
provided that 
$$
k\eta = \frac{2}{\lambda_{*,J}}\log\left(\frac{1}{\varepsilon}\right) \geq e \,, 
$$
and the step size $\eta$ satisfies
$$
\eta \leq \min\left\{1,\frac{m^2}{(m^2+ 8M^2)M\lVert A_J\rVert^2} ,\frac{\varepsilon^4}{4(\log(1/\varepsilon))^2\lVert A_J \rVert^4} \frac{\lambda_{\ast,J}^2}{\lambda_{\ast,J=0}^2}\right\}  \,.
$$
Here, 
\begin{equation}
 \overline{\mathcal{I}}_0(J,\varepsilon) : = \sup_{\mathbf{z} \in \mathcal{Z}^n} \mathcal{I}_0(\mathbf{z},J,\varepsilon)= \left[ \left( \frac{M+B}{2} +\frac{B}{2}+A\right) \sup_{\mathbf{z} \in \mathcal{Z}^n}\hat{C}_{\mathbf{z},J} + (M+B)\mathcal{D}_c\right]
 \cdot \varepsilon  \,,
 \label{cor:const:I0-bar}
\end{equation}
\begin{equation}
\mathcal{I}_{1}(J,\varepsilon)= \left(M\sqrt{\mathcal{C}_d}+B \right)\left( \hat{C_0} \frac{\varepsilon}{\sqrt{\lambda_{\ast,J=0}}}+ \hat{C_1}\delta^{1/4}\sqrt{\frac{2\log(1/\varepsilon)}{\lvert \lambda_{\ast,J} \rvert}}  \lVert A_J \rVert  \right)\sqrt{\log\left( \frac{2\log(1/\varepsilon)}{\lvert \lambda_{\ast,J} \rvert} \right)}  \,,
\label{cor:const:I1-star} 
\end{equation}
$\mathcal{I}_2$ is defined in \eqref{cor:const:I2},
and $\mathcal{I}_3(n)$ is provided by Proposition 12 of \cite{RRT17}:
$$
\mathcal{I}_{3}(n) :=\frac{4\left( \frac{M^2}{m}(b+d/\beta)+B^2 \right)\beta c_{LS}}{n} \, \text{with}\, c_{LS}\leq \frac{2m^2+8M^2}{m^2M\beta}+\frac{1}{\lambda_{*, J=0}}\left(\frac{6M(d+\beta)}{m}+2\right)\,.
$$ 
% where $\lambda_{*,J=0} = \lambda_{*}$ and $c_{LS}$ is a constant satisfying
% $$
% c_{LS}\leq \frac{2m^2+8M^2}{m^2M\beta}+\frac{1}{\lambda_*}\left(\frac{6M(d+\beta)}{m}+2\right)\,.
% $$
\end{corollary}

\begin{proof}[Proof of Corollary~\ref{cor:PRM-bd}]
Let $X_k$ follow the probability law $\mathcal{L}(X_k|\mathbf{Z}=\mathbf{z})=\mu_{\mathbf{z},k}  $ and the samples drawn by Gibbs algorithm $\pi_{\mathbf{z}} =\mathcal{L}(\hat{X}^*|\mathbf{Z}=\mathbf{z}) $ with $\mathbf{Z} = (Z_1,Z_2,...,Z_n) $ being a random variable from an unknown distribution and $\mathbf{z} = (z_1,z_2,...,z_n)$ being a deterministic data sample. The decomposition for population risk minimization problem admits the following inequality,
\begin{equation}\label{three:terms}
\mathbb{E} F(X_k)-F^*  = \left(\mathbb{E}F(X_k)-\mathbb{E}F(\hat{X}^*) \right) + \left(\mathbb{E}F(\hat{X}^*)-\mathbb{E}F_{\mathbf{Z}}(\hat{X}^*)\right) + \left(\mathbb{E}F_{\mathbf{Z}}(\hat{X}^*)- F^* \right)\,.
\end{equation}
We can write the first term in \eqref{three:terms} as the following identity over all possible training data set $\mathbf{z}$ in $\mathcal{Z}^n$.
\begin{equation}
\mathbb{E}F(X_k)-\mathbb{E}F(\hat{X}^*)=\int_{\mathcal{Z}^n}P^{\otimes n} (d\mathbf{z})\left( \int_{\mathbb{R}^d}F_{\mathbf{z}}(x)\mu_{\mathbf{z},k}(dx) - \int_{\mathbb{R}^d}F_{\mathbf{z}}(x)\pi_{\mathbf{z}}(dx) \right)\,,
\end{equation}
where $P^{\otimes n}$ is the product measures over the independent and identically distributed random variables $Z_1, Z_2, ..., Z_n  $ supported on $\mathcal{Z}^n $. 
To find an upper bound for the first term, we can consider an uniform bound over $\mathbf{z} \in \mathcal{Z}^n$ by using Corollary~\ref{cor:NSGLD2gibbs}. For a deterministic $\mathbf{z} \in \mathcal{Z}^n$, Corollary~\ref{cor:NSGLD2gibbs} states that, 
$$
\left|\mathbb{E}_{X \sim \mu_{\mathbf{z},k}} F_{\mathbf{z}}(X) - \mathbb{E}_{X \sim \pi_{\mathbf{z}}} F_{\mathbf{z}}(X) \right| \leq \mathcal{I}_{0}(\mathbf{z},J, \varepsilon) + \mathcal{I}_1(\mathbf{z},J, \varepsilon) \,.
$$
Recall we define $\mathcal{I}_{0}(\mathbf{z},J, \varepsilon)$ in \eqref{thm:const:I0} and we have
$$
\sup_{\mathbf{z} \in \mathcal{Z}^n}\hat{C}_{\mathbf{z},J}
=\sup_{\mathbf{z} \in \mathcal{Z}^n}C_{\mathbf{z},J}g_{\mathbf{z},J}(R)
\leq\sup_{\mathbf{z}\in \mathcal{Z}^n}C_{\mathbf{z},J}g_{J}(R)=:\overline{C}_{J},
$$
where $g_{J}(R)=\sup_{\mathbf{z}\in \mathcal{Z}^n}g_{\mathbf{z},J}(R)$.
Therefore, we can bound $\sup_{\mathbf{z}\in\mathcal{Z}^{n}}\mathcal{I}_{0}(\mathbf{z},J,\varepsilon)$ by:
\begin{align}
& \overline{\mathcal{I}}_0(J,\varepsilon): =\left[  \overline{C}_{J}\left( \frac{M+B}{2} +\frac{B}{2}+A\right) + (M+B)\mathcal{D}_c\right]\cdot \varepsilon\,. 
\end{align}
It follows that we can bound the first term in \eqref{three:terms} as
$$
\left| \mathbb{E}F(X_k)-\mathbb{E}F(\hat{X}^*) \right| \leq \overline{\mathcal{I}}_0(J,\varepsilon) + \mathcal{I}_1(J,\varepsilon) \,,
$$
where $\mathcal{I}_{1}(J,\varepsilon)$ is given in Corollary~\ref{cor:ERM-bd}, which uniformly
bounds $\mathcal{I}_{1}(\mathbf{z},J,\varepsilon)$.

The second term in \eqref{three:terms} is the generalization error of Gibbs algorithm that bounded in Lemma~\ref{slm:uniform-stability}, also see Proposition 12 in Raginsky \emph{et al.} \cite{RRT17}. In our notation, this part is bounded by $\mathcal{I}_3(n)$ where $n$ is the size of training set,
$$
\left| \mathbb{E}F(\hat{X}^*)-\mathbb{E}F_{\mathbf{Z}}(\hat{X}^*) \right|\leq \frac{4\left( \frac{M^2}{m}(b+d/\beta)+B^2 \right)\beta c_{LS}}{n} = \mathcal{I}_3(n) \,,
$$
The third term in \eqref{three:terms} is bounded by:
\begin{align*}
\mathbb{E}F_{\mathbf{Z}}(\hat{X}^*)- F^*
&= \mathbb{E}\left[F_{\mathbf{Z}}(\hat{X}^*) - \min_{x\in\mathbb{R}^d}F_{\hat{\mathbf{Z}}}(x)\right]+\mathbb{E}\left[\min_{x\in\mathbb{R}^d}F_{\hat{\mathbf{Z}}}(x)-F_{\hat{\mathbf{Z}}}(x^{\pi})\right] 
\\
&\leq \mathbb{E}\left[F_{\hat{\mathbf{Z}}}(X^{\pi}) - \min_{x\in\mathbb{R}^d}F_{\hat{\mathbf{Z}}}(x)\right] \leq \mathcal{I}_2 = \frac{d}{2\beta} \log\left( \frac{eM}{m} \left( \frac{b\beta}{d}+1 \right)\right) \,,
\end{align*}
where $\mathcal{I}_2$ is defined in \eqref{cor:const:I2} from Proposition 11, Raginsky \emph{et al.} \cite{RRT17}. The proof is complete.
\end{proof}

%%%%%%%%%%%%%%%%%%%%%%%%
\section{Supporting Lemmas}
This lemma shows that the object function can be upper and lower bounded by the quadratic function.

\begin{lemma}[Quadratic bounds, Lemma 2 in \cite{RRT17}] \label{slm:quadratic-bd}
If Assumptions~\ref{A1} and ~\ref{A2} hold, for all $x \in \mathbb{R}^d  $ and $z \in \mathcal{Z}  $, for some constant $c \in[0,1)$,
\begin{equation}
\lVert \nabla f(x,z) \rVert \leq M\lVert x \rVert + B\,,  \label{slm:gradient-bd}
\end{equation}
and 
\begin{equation}
\frac{m(1-c^2)}{2}\lVert x \rVert^2 +b\log c \leq f(x,z) \leq \frac{M}{2} \lVert x \rVert^2 + B\lVert x \rVert + A\,.
\end{equation}
\end{lemma}
For example, if we let $c=1/\sqrt{3}  $, we can have the quadratic bounds for the object function $f(x,z)  $ as
\begin{equation}
\frac{m}{3}\lVert x \rVert^2-\frac{b}{2}\log3 \leq f(x,z) \leq \frac{M}{2}\lVert x \rVert^2 + B\lVert x \rVert + A\,.   \label{slm:objfunc-bd}
\end{equation}

The next lemma implies the equivalence of 2-Wasserstein continuity for functions to weak convergence with the finite second moments.

\begin{lemma}[2-Wasserstein continuity for functions of quadratic growth, \cite{PW16}]
\label{slm:W2-quadratic-bd} 
Let $\mu,\nu  $ be two probability measures on $\mathbb{R}^d  $ with finite second moments, and let $g:= \mathbb{R}^d \rightarrow \mathbb{R}  $ be a $C^1  $ function obeying
\begin{equation*}
\lVert \nabla g(w) \rVert \leq c_1\lVert w \rVert+c_2, \qquad \text{for any} \quad w\in \mathbb{R}^d\,,
\end{equation*}
for some constants $c_1>0$ and $c_2>0$. Then 
\begin{equation*}
\left| \int_{\mathbb{R}^d}gd\mu-\int_{\mathbb{R}^d}gd\nu \right| \leq (c_1\sigma+c_2)\mathcal{W}_2(\mu,\nu)\,,
\end{equation*}
where 
\begin{equation*}
\sigma^2:=\int_{\mathbb{R}^d}\lVert w \rVert^2\mu(dw) \vee \int_{\mathbb{R}^d}\lVert w \rVert^2\nu(dw)\,.
\end{equation*}
\end{lemma} 

The last lemma shows the uniform stability of the Gibbs measure $\pi_{\mathbf{z}}$. Fix two $n$-truples $\mathbf{z}=(z_1,...,z_n), \overline{\mathbf{z}}=(\overline{z}_1,...,\overline{z}_n)  $ with card $\left| \{i:z_i \neq \overline{z}_i \} \right|=1$ which differ only in a single coordinate. It also implies that for any two different dataset $\mathbf{z}  $ and $\mathbf{\overline{z}}  $, the difference between the minimizer of $F$ for these two datasets can be bounded by selecting the size of the datasets.   
\begin{lemma}[Uniform stability, see Proposition 12 \cite{RRT17}] \label{slm:uniform-stability}
For any two $\mathbf{z}, \mathbf{\overline{z}} \in \mathcal{Z}^n  $ that differ only in a single coordinate, then
\begin{equation*}
\sup_{z\in\mathcal{Z}}\left| \int_{\mathbb{R}^d}f(x,z)\pi_{\mathbf{z}}(dx) -  \int_{\mathbb{R}^d}f(x,z)\pi_{\mathbf{\overline{z}}}(dx) \right| \leq \frac{4\left( \frac{M^2}{m}(b+d/\beta)+B^2 \right)\beta c_{LS}}{n}\,,
\end{equation*}
with 
\begin{equation*}
c_{LS}\leq \frac{2m^2+8M^2}{m^2M\beta}+\frac{1}{\lambda_*}\left(\frac{6M(d+\beta)}{m}+2\right)\,,
\end{equation*}
where $\lambda_{*}=\lambda_{\ast,J=0}$ is defined as the uniform spectral gap for the reversible Langevin SDE in \eqref{reversibleSDE}, such that
\footnote{In \cite{RRT17}, their formula for $\lambda_{\ast}$ missed $\beta^{-1}$ factor.}
\begin{equation}\label{uniform:spectralgap-rev}
\lambda_{*}=\inf_{\mathbf{z}\in\mathcal{Z}^n}\inf\left\{ \frac{\beta^{-1}\int_{\mathbb{R}^d}\lVert \nabla h\rVert^2d\pi_{\mathbf{z}}}{\int_{\mathbf{z}^d}h^2d\pi_{\mathbf{z}}} :h\in C^1(\mathbb{R}^d)\cap L^2(\pi_{\mathbf{z}}), h\neq 0, \int_{\mathbb{R}^d}hd\pi_{\mathbf{z}}=0 \right\}\,.
\end{equation}
\end{lemma}

\begin{lemma}\label{lem:norm}
Let $A_J = I + J$ where $J$ is a $d$-dimensional anti-symmetric matrix, $J^T = -J$, then
\begin{equation}
\lVert A_J \rVert^{2} = 1+\lVert J \rVert^2 \geq 1\,.
\end{equation}
\end{lemma}

\begin{proof}
For any $x \in \mathbb{R}^d$, then we have
\begin{equation}
\langle A_J x,A_Jx \rangle = \langle (I+J)x, (I+J)x\rangle
= \lVert x \rVert^2 + \lVert Jx\rVert^2 \,.
\end{equation}
where we use the fact $\langle x,Jx \rangle = \langle -Jx,x \rangle$. 
Since it holds for any $x\in\mathbb{R}^{d}$, we conclude that
$\lVert A_J \rVert^{2} = 1+\lVert J \rVert^2$.
\end{proof}

%%%%%%%%%%%%%%%%%%%%%%%
\section{Explicit Dependence of Constants on Key Parameters} \label{sec:dependancy-para}

In this section, we discuss explicit dependence of the performance bound of the empirical risk minimization of NSGLD on the parameters $\beta \,, d \,, J \,, \lambda_{*,J}$ that is summarized in in Section~\ref{sec:comparison}. Recall the performance bound of the empirical risk minimization of NSGLD is based on $\mathcal{I}_0(\mathbf{z},J,\varepsilon) + \mathcal{I}_1(\mathbf{z},J,\varepsilon) + \mathcal{I}_2$ from Corollary~\ref{cor:ERM-bd},
\begin{equation*}
\mathcal{I}_{0}(\mathbf{z},J,\varepsilon) : =   \left[ \left( \frac{M+B}{2} +\frac{B}{2}+A\right)\hat{C}_{\mathbf{z},J} + (M+B)\mathcal{D}_c\right]\cdot \varepsilon \,,
\end{equation*}
where $\hat{C}_{\mathbf{z},J}$ is a constant depending on data set $\mathbf{z} \in \mathcal{Z}^n$ and a $d \times d$ anti-symmetric matrix $J$. Then it follows
\begin{equation}
\mathcal{I}_0(\mathbf{z},J,\varepsilon) = \tilde{\mathcal{O}}\left(\hat{C}_{\mathbf{z},J} \, \varepsilon \right) \,.
\end{equation}
And 
\begin{equation*}
\mathcal{I}_{1}(\mathbf{z},J,\varepsilon) := \left(M\sqrt{\mathcal{C}_d}+B \right)\left( \hat{C_0} \frac{\varepsilon}{\sqrt{\lambda_{\mathbf{z},J=0}}}+ \hat{C_1}\delta^{1/4}\sqrt{\frac{2\log(1/\varepsilon)}{\lambda_{\mathbf{z},J} }} \lVert A_J \rVert \right) \sqrt{\log\left( \frac{2\log(1/\varepsilon)}{\lambda_{\mathbf{z},J}} \right)}\,,
\end{equation*}
with 
\begin{equation*}
\hat{C}_0 = \tilde{\mathcal{O}}(\sqrt{\beta}\sqrt{\beta+d}) \,, \qquad \hat{C}_1 = \tilde{\mathcal{O}}(\sqrt{\beta}) \,.
\end{equation*}
By setting the gradient noise $\delta$ equal to the step size $\eta$, we get
\begin{align}
\mathcal{I}_1(\mathbf{z},J,\varepsilon) & = \tilde{\mathcal{O}}\left(\left( \sqrt{\beta}\sqrt{\beta+d} \, \frac{\varepsilon}{\sqrt{\lambda_{\mathbf{z},J=0}}} \right)  \sqrt{\log\left( \frac{\log(1/\varepsilon)}{\lambda_{\mathbf{z},J}} \right)} \right) \nonumber \\
&  = \tilde{\mathcal{O}}\left(\left( \sqrt{\beta}\sqrt{\beta+d} \, \frac{\varepsilon}{\sqrt{\lambda_{\mathbf{z},J=0}}} \right) \left(\log\log(\varepsilon^{-1})  + \log\left(\frac{1}{\lambda_{\mathbf{z},J}}\right)\right)^{1/2}\right) \nonumber \\
& \leq \tilde{\mathcal{O}}\left( \frac{\sqrt{\beta}(\beta+d)}{\sqrt{\lambda_{\mathbf{z},J=0}}} \, \varepsilon \right)\,,
\end{align}
where the factor $\sqrt{\log\log(\varepsilon^{-1})}$ is negligible comparing to other factors,
we also used $\lambda_{\mathbf{z},J} > \lambda_{\mathbf{z}, J=0}$ and $1/\lambda_{\mathbf{z}, J=0} = e^{\tilde{\mathcal{O}}(\beta+d)}$ according to \cite{RRT17}. 
Moreover, for $\mathcal{I}_2$, we have
\begin{equation}
\mathcal{I}_2  = \frac{d}{2\beta}\log\left( \frac{eM}{m}\left( \frac{b\beta}{d}+1 \right) \right) = \tilde{\mathcal{O}}\left( \frac{d\log(1+\beta)}{\beta} \right) \,.
\end{equation}
Hence, the performance bound for empirical risk minimization of NSGLD is
\begin{equation}
\tilde{\mathcal{O}}\left( \hat{C}_{\mathbf{z},J} \, \varepsilon  +  \frac{\sqrt{\beta}(\beta+d)}{\sqrt{\lambda_{\mathbf{z},J=0}}} \, \varepsilon  + \frac{d\log(1+\beta)}{\beta} \right) \,.
\end{equation}

\end{document}